\documentclass[10pt]{amsart}
\usepackage[latin1]{inputenc}
\usepackage{amsmath,amsthm,amssymb,amsfonts}
\usepackage{a4wide}
\usepackage{enumitem}
\usepackage{color}
\usepackage{caption} 
\usepackage{subcaption}
\usepackage{here}

\usepackage{graphicx}
\usepackage{mathtools} 
\usepackage{dsfont}
\usepackage{tikz}
\usepackage{float,here}
\usetikzlibrary{decorations.pathreplacing,decorations.pathmorphing,patterns,scopes,intersections,calc}

\usepackage{algorithm,algpseudocode}
\RequirePackage[bookmarks,%
colorlinks,%
urlcolor=blue,%
citecolor=red,%
linkcolor=blue,%
hyperfigures,%
pdfcreator=LaTeX,%
breaklinks=true,%
pdfpagelayout=SinglePage,%
bookmarksopen=true,%
bookmarksopenlevel=2]{hyperref}
\numberwithin{equation}{section}

\newcommand{\N}{\mathbb{N}}
\newcommand{\R}{\mathbb{R}}

\newtheorem{corollary}{Corollary}[section]
\newtheorem{definition}[corollary]{Definition}

\newtheorem{lemma}[corollary]{Lemma}
\newtheorem{proposition}[corollary]{Proposition}
\newtheorem{remark}[corollary]{Remark}

\newtheorem{theorem}[corollary]{Theorem}
\newfont{\sBlackboard}{msbm10 scaled 900}

\newcommand{\mylabel}[1]{\label{#1}
            \ifx\undefined\stillediting
            \else \fbox{$#1$}\fi }
\newcommand{\BE}{\begin{equation}}
\newcommand{\BEQ}[1]{\BE\mylabel{#1}}
\newcommand{\EEQ}{\end{equation}}
\newcommand{\rfb}[1]{\mbox{\rm
   (\ref{#1})}\ifx\undefined\stillediting\else:\fbox{$#1$}\fi}
\newcommand{\half}   {{\frac{1}{2}}}

\newfont{\Blackboard}{msbm10 scaled 1200}
\newcommand{\bl}[1]{\mbox{\Blackboard #1}}
\newfont{\roma}{cmr10 scaled 1200}

\def\CC{\rm \hbox{C\kern-.56em\raise.4ex
         \hbox{$\scriptscriptstyle |$}\kern+0.5 em }}
\newcommand{\be}{\begin{equation}}
\newcommand{\ee}{\end{equation}}
\newcommand{\beq}{\begin{eqnarray}}
\newcommand{\eeq}{\end{eqnarray}}
\newcommand{\beqs}{\begin{eqnarray*}}
\newcommand{\eeqs}{\end{eqnarray*}}
\newcommand{\bt}{\begin{Theorem}}
\newcommand{\et}{\end{Theorem}}
\newcommand{\br}{\begin{remark}}
\newcommand{\er}{\end{remark}}
\newcommand{\bc}{\begin{Corollary}}
\newcommand{\ec}{\end{Corollary}}
\newcommand{\el}{\end{Lemma}}
\newcommand{\bd}{\begin{definition}}
\newcommand{\ed}{\end{definition}}

\newcommand{\nline}  {{\bl N}}
\newcommand{\rline}  {{\bl R}}

\newcommand{\mm}    {{\hbox{\hskip 0.5pt}}}
\newcommand{\m}     {{\hbox{\hskip 1pt}}}
\newcommand{\nm}    {{\hbox{\hskip -3pt}}}

\newcommand{\bluff} {{\hbox{\raise 15pt \hbox{\mm}}}}

\newcommand{\FORALL} {{\hbox{$\hskip 11mm \forall \;$}}}

\newcommand{\rarrow} {{\,\rightarrow\,}}

\def\be{\begin{equation}}
\def\ee{\end{equation}}

\def\ds{\displaystyle}
\usepackage{xpatch}
\makeatletter   
\xpatchcmd{\@tocline}
{\hfil\hbox to\@pnumwidth{\@tocpagenum{#7}}\par}
{\ifnum#1<0\hfill\else\dotfill\fi\hbox to\@pnumwidth{\@tocpagenum{#7}}\par}
{}{}
\makeatother    
\makeatletter
\def\l@subsection{\@tocline{2}{0pt}{4pc}{6pc}{}}
\def\l@subsubsection{\@tocline{3}{0pt}{8pc}{8pc}{}}
\makeatother
\usepackage[english]{babel}
\begin{document}

\thispagestyle{empty}
\title{Numerical stabilization method by switching time-delay}
\author{Ka\"{\i}s Ammari}
\address{LR Analysis and Control of PDEs, LR 22ES03, Department of Mathematics,
Faculty of Sciences of Monastir, University of Monastir, 5019 Monastir, Tunisia}
\email{kais.ammari@fsm.rnu.tn}

\author{St\'ephane Gerbi}
\address{Laboratoire de Math\'ematiques UMR 5127 CNRS \& Universit\'e Savoie Mont Blanc, Campus scientifique, 73376 Le Bourget du Lac Cedex, France}
\email{stephane.gerbi@univ-smb.fr}
\date{}
\begin{abstract}
In this paper, we propose a new numerical strategy for the stabilization of evolution systems. The method is based on the methodology given by Ammari-Nicaise-Pignotti in \cite{ANP1}. This method is then implemented in 1D by suitable numerical approximation techniques. Numerical experiments complete this study to confirm the theoretical announced results.
\end{abstract}


\maketitle
\tableofcontents

\section{Introduction}
\setcounter{equation}{0}

Delay effects arise in many applications and practical problems and
it is well-known that an arbitrarily small delay may destabilize a
system which is uniformly asymptotically stable in absence of delay
(see e.g. \cite{Datko,DLP,Datko97},  \cite{NPSicon})). Nevertheless recent papers reveal that particular choice of delays may restitute exponential stability property, see \cite{g,gt,guo}.

We refer also to \cite{amman,ANP,AN, ANP1,ANP2,AG,NPSicon,NVCOCV10} for stability results for systems with time delay
due to the presence of ``good'' feedbacks compensating the destabilizing delay effect.  

\medskip

In this paper we propose a numerical approach that consists in stabilizing the abstract-wave system
by a control law that uses information from the past (by  switching or not). This means that the  stabilization is obtained  by a control method, see \cite{ANP1} for more details,
and not by a feedback law. This strategy can  provide a guide to the time--delay compensation scheme. For switching control (without delay) we refer to Zuazua \cite{zuazuasc}. 

\medskip

Let $H$ be a Hilbert space equipped with the norm $||.||_{H}$, and
let $A:\mathcal{D}(A) \subset H \rightarrow H$ be a self-adjoint, positive and
invertible operator. We introduce the scale of Hilbert spaces
$H_{\alpha}$, $\alpha\in\rline$, as follows\m: for every $\alpha\geq
0$, $H_{\alpha}=\mathcal{D}(A^{\alpha})$ with the norm $\|z
\|_{\alpha}=\|A^\alpha z\|_{H}$. The space $H_{-\alpha}$ is defined
by duality with respect to the pivot space $H$ as  $H_{-\alpha}
=H_{\alpha}^*$ for $\alpha>0$. The operator $A$ can be extended (or
restricted) to each $H_{\alpha}$ such that it becomes a bounded
operator \BEQ{A0ext} A: H_{\alpha} \rarrow H_{\alpha - 1} \,
\,\mbox{for}\,\, \alpha\in\rline \m. \EEQ

The second ingredient needed for our construction is a bounded
linear operator $B~:~ U \longrightarrow ~H_{-\half},$ where $U$ is
another Hilbert space identified with its dual. The operator $B^*$
is bounded from $H_{\frac12}$ to $U$.

The system that we considered in this paper is given by the following abstract problem:
\BEQ{damped1}  
\ddot{u}(t)
+ Au(t)  
= 0, \quad 0 \leq t\leq  T_0, \EEQ 
\BEQ{int}
\ddot{u}(t)
+ Au(t)  + \mu \, BB^* \dot{u} (t- T_0) = 0, \quad t\geq  T_0,
\EEQ
\BEQ{output} u(0) \m=\m u_0,  \m 
\dot{u}(0)=u_1, \EEQ 
where
$T_0 > 0$ is the time delay, $\mu$ is a real number and the initial datum $(w_0,w_1)$ belongs to a suitable 
space.

\medskip

For this system we need to assume the closed-loop admissibility of the operator $B^*$ (see for more details \cite{AN} and \cite{tucsnakweiss}), i.e. that for all $T >0$ there exists $C(T) > 0$ such that 
{\color{black}
\be
\label{inegobsv}
\left(
\int_{0}^{T}  \left\|B^{*} \dot{\phi}(s)\right\|_{U}^2 d\,s \right)^{\frac 1 2}\leq C(T)\, \left(\|(w_0,w_1)\|_{H_\half \times H} + \|u\|_{L^2(0,T;U)} \right),
\ee
for all $(w_0,w_1) \in H_1 \times H_\half$, $u\in L^2(0,T;U)$ and $\phi \in C([0,T], H_1)\cap C^1([0,T], H_\half)$ is the solution of
the inhomogeneous evolution equation
\be \ddot{\phi}(t) + A \phi(t) = Bu(t), 
\quad t\in (0,T), \label{eq3b} \ee \be \phi(0) = w_0,
 \dot{\phi} (0) = w_1.\nonumber
  \label{eq4b}\ee
}
 
To study the well--posedness of the system
\eqref{damped1}--\eqref{output}, we write it as an abstract
Cauchy problem in a product Banach space, and use the semigroup
approach. For this take the Hilbert space ${\mathcal H}~:=~
H_{\half} \times H$ and the unbounded
linear operators
 \be
  \label{opbb} {\mathcal  A}: {\mathcal
D}({\mathcal  A}) = H_1 \times H_\half \subset {\mathcal  H} \longrightarrow {\mathcal  H}, \, {\mathcal
A}\left(
\begin{array}{ccc}u_1\\u_2\end{array}\right) = \left(
\begin{array}{cc}
 u_2 \\
 - Au_1  
\end{array}
\right) \ee

and

$$ {\mathcal  A}_d: {\mathcal
D}({\mathcal  A}_d) = \left\{(u,v) \in {\mathcal H}; \, v \in H_\half, \, Au + \mu BB^*v \in H \right\} \subset {\mathcal  H} \longrightarrow {\mathcal  H}, \, 
$$
\be
  \label{opbbb}
{\mathcal A}_d \left(
\begin{array}{ccc}u_1\\u_2\end{array}\right) = \left(
\begin{array}{cc}
u_2 \\
- Au_1  - \mu \, BB^*u_2
\end{array}
\right). \ee

The operators $({\mathcal  A}, {\mathcal  D}({\mathcal  A}) )$ and $({\mathcal  A}_d, {\mathcal  D}({\mathcal  A}_d) )$  defined by \eqref{opbbb} 
generate group of isometries of ${\mathcal H}$ and strongly continuous semigroup of contractions on ${\mathcal H}$, respectively, denoted respectively by  $({\mathcal
T}(t))_{t\geq0}$ and $({\mathcal
T}_d(t))_{t\geq0}$ (as before let $({\mathcal
T}_{-1}(t))_{t\geq0}$ be the extension of $({\mathcal
T}(t))_{t\geq0}$ to $H \times H_{- \half}$).
 
\medskip

The paper is organized as follows. The second section deals with the
well--posedness of the problem while, in section \ref{sectasanalysis}, we prove
an exponential stability result of the shifted system associated to \rfb{damped1}--\rfb{output} where $B \in \mathcal{L}(U,H)$. 

In section \ref{Numerics} we construct suitable numerical schemes to find an approximate solution of the different problems studied in this work: finite difference scheme or finite volume method. For each scheme, we design a discrete energy which must be preserved in the first step of the stabilization method. Finally the section \ref{Numerical-results} is devoted to present
numerical experiments of each studied case in order to confirm the theoretical results.

\section{Well-posedness} \label{wellp}
\setcounter{equation}{0}

We study the well-posedness of the problem \rfb{damped1}-\rfb{output} in two cases: 
\begin{itemize}
\item
general case, where $B \in \mathcal{L}(U,H_{-\half})$ 
\item
bounded case, where $B \in \mathcal{L}(U,H)$. 
\end{itemize}

\subsection{General case}

Consider the evolution {\color{black}problems}
 \be 
\label{OPEN1}
\ddot{y}^j(t) +
 Ay^j(t)  \, 
=  Bv^j(t), \,\hbox{ in } (jT_0,(j+1) T_0), \, j \in \nline^*,
\ee
\be
\label{eq3}
y^j(jT_0)= \dot{y}^j(jT_0) = 0, \, j \in \nline^*.
\ee

\be 
\label{eq4}
\ddot{\phi}(t) +
 A\phi(t)  \, 
=  0, \, \hbox{ in }  (0,+\infty),
\ee
\be
\label{OPEN2}
\phi(0)= \phi_0, \, \dot{\phi}(0) = \phi_1.
\ee
A natural question is the regularity of $y^j$ when $v^j \in L^2(jT_0,(j+1)T_0;U), \, j \in \nline^*$. By 
applying standard energy estimates we can easily check that 
$y^j \in C([jT_0,(j+1)T_0];H) \cap C^1([jT_0,(j+1)T_0];H_{- \half})$. However if $B$ satisfies a 
certain admissibility condition then $y^j$ is more regular. More precisely the 
following result, which is a version of the general transposition method 
(see, for instance, Lions and Magenes \cite{linsmag} and {\color{black}\cite{lions}}) holds true.

It is clear that the system \rfb{eq4}--\rfb{OPEN2} admits a unique solution $\phi$ 
having the regularity 
$$
\phi \in C([0,T_0];H_\half) \cap C^{1}([0,T_0];H),
$$
$$
\left(
\begin{array}{lc}
\phi \\ \dot{\phi}
\end{array}
\right)(t) = {\mathcal  T}(t)\left(
\begin{array}{ccc}\phi_0\\\phi_1\end{array}\right), \quad
0\leq t \leq T_0.
$$
Moreover, according to assumption \rfb{inegobsv},
$B^* \phi(\cdot)\in H^1(0,T_0), $ and for all $T \in (0, T_0)$ there exists a constant
$C>0$
such that 
\be
\Vert B^* \dot{\phi}(\cdot)\Vert_{L^2(0,T;U)}\le C \, \Vert (\phi_0,\phi_1)\Vert_{\mathcal{H}},
\FORALL (\phi_0,\phi_1)\in H_1 \times H_\half.
\label{CACHE1}\ee

\begin{lemma} \label{ex}
Suppose that $v^j \in L^2([jT_0,(j+1)T_0];U), \, j \in \nline^*$.  
Then the problem \rfb{OPEN1}--\rfb{eq3} admits a unique solution 
having the regularity
\be
y^j \in C([jT_0,(j + 1) T_0];H_\half) \cap C^{1}([ jT_0,(j + 1) T_0];H), \  j \in \nline^*,
\label{REG1}\ee
and
$$
\left(
\begin{array}{lc}
y^{j} \\
\dot{y}^{j}
\end{array}
\right)(t) =   \ds \int_{jT_0}^{t} {\mathcal  T}_{-1}(t-jT_0-s)\left(
\begin{array}{ccc}0\\B v_j(s)\end{array}\right) \, ds, \quad
j T_0 \leq t \leq (j+1) T_0,\  j \geq 1.
$$
\end{lemma}
\begin{proof}
If we set $Z(t) = \left(
\begin{array}{ll}
y^j(t+jT_0) \\
\nm
\ds
\dot{y}^j(t+jT_0)
\end{array}
\right)$ it is clear that \rfb{OPEN1}--\rfb{eq3} can be written as
\[
\dot{Z}^j + {\mathcal A} Z^j(t) = {\mathcal B} v^j(t+jT_0) \hbox{ on } (0, T_0), \, Z^j(0) = 0,
\]
where 
\[
{\mathcal A} = \left(
\begin{array}{cc}
0 & - I \\
\nm
\ds
A &0
\end{array}
\right): \mathcal{H} \subset [{\mathcal D}({\mathcal A})]^{\prime} \rightarrow [{\mathcal D}({\mathcal A})]^{\prime} ,
\]
\[
{\mathcal B} = \left(
\begin{array}{ll}
0 \\
\nm
\ds
B
\end{array}
\right): U \rightarrow [{\mathcal D}({\mathcal A})]^{\prime}.
\]
{\color{black} Here, these operators correspond respectively to the extensions of $$\left(
\begin{array}{cc}
0 & - I \\
\nm
\ds
A &0
\end{array}
\right) \; \hbox{and} \; \left(
\begin{array}{ll}
0 \\
\nm
\ds
B
\end{array}
\right)$$ to the extrapolation space $H \times 
H_{-\frac{1}{2}}.$}

It is well known that ${\mathcal A}$ is a skew adjoint operator so it generates a group of isometries in $[{\mathcal D}({\mathcal A})]^{\prime} $, denoted by ${\mathcal S}(t) (= {\mathcal T}_{-1}(t))$.

After simple calculations we get that the operator ${\mathcal B}^*: {\mathcal D}({\mathcal A}) \rightarrow U$ is given by 
\[
{\mathcal B}^*  \left(\begin{array}{c}
u^j \\
v^j
\end{array}
\right) = B^* v^j , \,\forall \, 
(u^j,v^j) \in  {\mathcal D}({\mathcal A}).
\]
This implies that 
\[
{\mathcal B}^*{\mathcal S}^*(t) 
\begin{pmatrix} \phi_0 \cr \phi_1 \end{pmatrix} = B^* \dot{\phi} (t), \, \forall \, 
(\phi_0,
\phi_1) \in  {\mathcal D}({\mathcal A}),
\]
with $\phi$ satisfying \rfb{eq4}--\rfb{OPEN2}. From the inequality above and 
\rfb{CACHE1} we deduce that there exists a constant {\color{black}$C:= C(T_0) > 0$} such that 
for all $T\in (0, {\color{black}T_0})$
\[
\int_{0}^{T}  \left\| {\mathcal B}^*{\mathcal S}^*(t) \left(\begin{array}{c}
\phi_0 \\
\phi_1
\end{array}
\right) \right\|^2_{U} \, dt \le C \, ||(\phi_0,\phi_1)||^2_{\mathcal{H}}, \, \forall \, 
(\phi_0,\phi_1) \in  {\mathcal D}({\mathcal A}).
\]
According to Theorem 3.1 in \cite[p.187]{ben} (see also \cite{tucsnakweiss}) the inequality above implies 
the interior regularity  \rfb{REG1}.
\end{proof}

{\color{black} We now state the following proposition.
\begin{proposition}\label{propexistunicb}
Assume that the assumption \rfb{inegobsv} holds, then the system
\rfb{damped1}--\rfb{output} is well--posed. More precisely, for every
$(u_0,u_1)\in {\mathcal  H}$, the solution of \rfb{damped1}--\rfb{output} is given by 
\be
\label{existresl}
\left(
\begin{array}{ccc}
u(t)\\ \dot{u}(t) 
\end{array}\right)= \left\{
\begin{array}{ll}
\left(
\begin{array}{ccc}
u^0(t)\\ \dot{u}^0(t)
\end{array}\right) = {\mathcal  T}(t)\left(
\begin{array}{ccc}u_0\\u_1\end{array}\right), \, 0 \leq t \leq  T_0, \\
\left(
\begin{array}{ccc}
u^{j}(t)\\ \dot{u}^{j}(t)
\end{array}\right) = {\mathcal  T}(t-j T_0)\left(
\begin{array}{ccc}u^{j-1}(j T_0)\\ \dot{u}^{j-1}(j T_0)\end{array}\right) + \\ \ds \int_{jT_0}^{t} {\mathcal  T}_{-1}(t-s)\left(
\begin{array}{ccc}0\\ - \mu \, BB^* \dot{u}^{j-1}(s-T_0)\end{array}\right) \, ds, \, \\
\hspace{5cm} j T_0 \leq t \leq (j+1) T_0, j \geq 1
\end{array}
\right.
\ee
and satisfies $(u^j,\dot{u}^j) \in C([jT_0, (j+1)T_0], {\mathcal  H}), \, j \in \nline.$ Moreover, there exist two positive constants  $M$ and $\omega$ such that
\be\label{continuitycas1}
\|(u(t),\dot{u}(t))\|_{\mathcal  H}^2\leq  M e^{\omega t} \, \left\| (u_0,u_1)\right\|^2_{{\mathcal H}}, \forall t\geq 0.
\ee
\end{proposition}

\br
The above Proposition suggests that the mapping
$$
T_t:{\mathcal H} \to {\mathcal H}: (u_0,u_1) \mapsto (u(t),\dot{u}(t))
$$
defines a strongly continuous semigroup but it is not the case since the semigroup property $T_{t+s}=T_t T_s$ is not valid in general.
\er

For any solution of problem \rfb{damped1}--\rfb{output} we define the energy
\begin{equation}\label{energy}
\hspace{3cm}
E(t) =\frac{1}{2} \left\|(u(t),\dot{u}(t)\right\|^2_{\mathcal{H}}, \, t \geq 0.
\end{equation}
}
\begin{proof}
First of all, let us prove the equation \rfb{existresl}.
The existence result for problem \rfb{damped1}--\rfb{output} is now made by induction, {\color{black} the details are omitted at this point, but we explain only the main principle}.
First on $[0,T_0]$ (case   $j=0$), we take
$$
\left(
\begin{array}{ccc}
u^0(t)\\ \dot{u}^0(t)
\end{array}\right) = {\mathcal  T}(t)\left(
\begin{array}{ccc}u_0\\u_1\end{array}\right), \, \forall t \in [0,T_0].$$
That is clearly a solution of \rfb{damped1}--\rfb{output} on $(0,T_0)$
  and  that has the regularity $(u^0,\dot{u}^0) \in C([0,T_0];{\mathcal H})$.
Now for $j\geq 1$,   we take for all $t\in [j T_0 ,(j+1) T_0]$,
\beqs
\left(
\begin{array}{ccc}
u^{j}(t)\\ \dot{u}^{j}(t)
\end{array}\right) &=&\left(
\begin{array}{ccc}
\phi(t+jT_0)\\ \dot{\phi}(t+jT_0)
\end{array}\right) + 
\left(
\begin{array}{ccc}
y^{j}(t)\\ \dot{y}^j(t)
\end{array}\right)
\\
&=&{\mathcal  T}(t{\color{black}-}jT_0)\left(
\begin{array}{ccc}u^{j-1}(jT_0) \\ \dot{u}^{j-1}(jT_0)\end{array}\right)   +\ds \int_{jT_0}^{t} {\mathcal  T}_{-1}(t-s)\left(
\begin{array}{ccc}0\\- \mu \, BB^*\dot{u}^{j-1}(s - T_0) \end{array}\right) \, ds, 
\eeqs
where $y^j$ (resp. $\phi$)  is solution of \rfb{OPEN1}--\rfb{eq3} (resp. \rfb{eq4}--\rfb{OPEN2})
with   $v^j(t) = - \mu \, B^* \dot{u}^{j-1}(t - T_0)$ (that belongs to $L^2(jT_0,(j+1)T_0;U)$ because the operator $B^*$ is an input admissible operator according to assumption \rfb{inegobsv}) and   $\phi_0 = u^{j-1}(jT_0),$  $\phi_1 = \dot{u}^{j-1}(jT_0)$.
This solution has the announced regularity due to the above arguments.

Let us now prove the estimation \rfb{continuitycas1}.

For the system
\rfb{damped1}--\rfb{output}, the estimate (\ref{inegobsv}) is used to prove our existence result by iteration.
Namely, for all 
$j \in \nline$, we prove by iteration that 
$(u^j,\dot{u}^j)$ as defined in the statement belongs to $C([jT_0, (j+1)T_0], {\mathcal  H})$, 
and satisfies
\be 
\label{inegobsiterative}
 \Lambda_{j+1}\leq C_2(T_0)(1+|\mu|) \Lambda_{j}
\ee
for some positive constant $C_2(T_0)$ and
where 
$$\Lambda_j=  
\|(u^j(jT_0),\dot{u}^j(jT_0))\|_{\mathcal H}+\left(\int_{(j-1)T_0}^{jT_0}  \left\|B^* \dot{u}^{j-1}(s)\right\|_{U}^2 d\,s\right)^{\frac12},
$$
with the convention $u^{-1}=0$.
 
Note that
$\left(
\begin{array}{ccc}
u^0(t)\\ \dot{u}^0(t)
\end{array}\right) = {\mathcal  T}(t)\left(
\begin{array}{ccc}u_0\\u_1\end{array}\right)$ is clearly in $C([0, T_0], {\mathcal  H})$, 
and satisfies \rfb{inegobsiterative} since ${\mathcal  T}(t)$ is a semigroup of contractions.

Now for $j\geq 1$, we assume that the result holds for $j-1;$ then
by (\ref{inegobsv}), we know that 
 $(u^j,\dot{u}^j)\in C([jT_0, (j+1)T_0], {\mathcal  H})$
and that
\beqs
\left(\int_{jT_0}^{(j+1)T_0}  \left\|B^{*} \dot{u}^j(s)\right\|_{U}^2 d\,s\right)^{\frac12} &\leq& C(T_0)\,
\Big(\|(u^j(jT_0),\dot{u}^j(jT_0))\|_{\mathcal H}
\\
&+&|\mu| \Big (
\int_{jT_0}^{(j+1)T_0}  \left\|B^{*} \dot{u}^{j-1}(s-T_0)\right\|_{U}^2 d\,s\Big )^{\frac12}\Big).
\eeqs

On the other hand using Theorem 4.4.3 and Proposition 4.2.2 of \cite{tucsnakweiss} (see below),
we know that
there exists $C_1(T_0)>0$ such that for all $t\in [jT_0, (j+1)T_0]$,
\be 
\label{inegobsiterative2}
\|(u^j(t),\dot{u}^j(t))\|_{\mathcal H}\leq \|(u^j(jT_0),\dot{u}^j(jT_0))\|_{\mathcal H}+|\mu| C_1(T_0) \left(
\int_{jT_0}^{(j+1)T_0}  \left\|B^{*} \dot{u}^{j-1}(s-T_0)\right\|_{U}^2 d\,s\right)^{\frac12}.
\ee
This estimate evaluated at $t= (j+1)T_0$ and added to the previous one yields 
\rfb{inegobsiterative} with $C_2(T_0)= {\color{black} 2} \max\{1, C_1(T_0), C(T_0)\}.$
This proves that the result holds for all $j$.

By iteration, \rfb{inegobsiterative} implies that
$$
\Lambda_{j}\leq C_2(T_0)^j(1+|\mu|)^j\|(u_0,u_1)\|_{\mathcal H}, \forall j\in \nline.
$$
This estimate and \rfb{inegobsiterative2} yield \rfb{continuitycas1}.
\end{proof}

\subsection{Bounded case} \label{bcase}
Here we assume that the operator $B \in \mathcal{L}(U,H)$. So we will give a well--posedness result for problem \eqref{damped1}--\eqref{output} by using semigroup theory.

\medskip

We introduce the auxiliary variable, {\color{black} as in \cite{amman,NPMCSS}},
\begin{equation}\label{defz0}
z(\rho ,t)= B^* \dot{u}(t-T_0\rho ),\quad \ \rho\in (0,1), \ t>0.
\end{equation}
Then, problem \eqref{damped1}--\eqref{output}
is equivalent to
\begin{align}
{}& \ddot{u} (t) + Au(t) + \mu \, Bz(1,t) = 0 ,  && \mbox{\rm in} \quad (0, + \infty),\label{a1bis}\\
{}& T_0 z_t(\rho,t)+z_{\rho}(\rho, t)=0\quad &&\mbox{\rm in}\quad (0,1)\times (0,+\infty ),\label{Pz}\\
{}&u(0) = u_0,  \quad  \dot{u}(0)=u_1, &&  \quad \label{a4bis}
\\
{}&z(\rho, 0) = 0,  \quad  && \mbox{\rm in} \quad (0,1),\label{a5bis}
\\
{}&z(0,t)=B^* \dot{u}(t),\quad &&    \ t>0.\label{a6}
\end{align}
If we denote
$$
U:=\left (
u,
\dot{u},
z
\right )^\top,
$$
then
$$
\dot{U}:=\left (
\dot{u},
\ddot{u},
z_t
\right )^\top=
\left (
\dot{u},
- Au - \mu \, Bz(1,\cdot),
-T_0^{-1}z_{\rho}
\right )^\top.
$$
Therefore, problem (\ref{a1bis})--(\ref{a6}) can be rewritten as
\begin{equation}\label{formulA}
\left\{
\begin{array}{l}
\dot{U} = {\mathcal A}_g U,\\
U(0) = \left (
u_0,
u_1,
0
\right )^\top,
\end{array}
\right.
\end{equation}
where the operator ${\mathcal A}_g$ is defined by
$$
{\mathcal A}_g \left (
\begin{array}{l}
u\\
v\\
z
\end{array}
\right ):=
\left (
\begin{array}{l}
v\\
- Au - \mu \, {\color{black}Bz(1,\cdot)}\\
-T_0^{-1}z_{\rho}
\end{array}
\right ),
$$
\noindent with domain
\begin{equation}\label{domain}
\displaystyle{
{\mathcal D} ({\mathcal A}_g):=\Big \{\ (u,v,z)^\top\in
H_1
\times H_\half \times H^1(0,1; U)\: B^*v=z(0)}
 \Big\},
\end{equation}
in the Hilbert space
\begin{equation}\label{space}
{\mathcal H}_g:= H_\half \times H \times L^2(0,1;U),
\end{equation}
equipped with the standard inner product
\[
((u,v,z), (u_1,  v_1,  z_1))_{{\mathcal H}_g}
=
(A^\half u,A^\half u_{1})_{H} + (v,v_1)_H \, + \,
\xi \int_0^1 (z,z_1)_{U} \, d\rho,
\]
where $\xi>0$ is a parameter fixed later on.

We will show that ${\mathcal A}_g$ generates a $C_0$ semigroup on ${\mathcal H}_g$
by proving that ${\mathcal A}_g - c Id$ is maximal dissipative for an appropriate choice of $c$ in function of $\xi, T_0, B^*$ and $\mu$.
Namely we prove the next result.
\begin{lemma}\label{lmaxdiss}
If $\xi  \geq |\mu| \,T_0$, then ${\mathcal A}_g -  \left(\frac{|\mu|}{2} + \frac{\xi}{2T_0}\right) \, \left\|B^*\right\|^2_{\mathcal{L}(H,U)} \, Id$ is maximal dissipative in ${\mathcal H}_g$.
\end{lemma}

\begin{proof}
Take $U=(u,v,z)^T\in {\mathcal D}({\mathcal A}_g).$ Then we have
{\color{black}
\beqs
  ({\mathcal A}_g (u,v,z), (u,  v,  z))_{{\mathcal H}_g}
= - \mu \, (z(1),B^*v)_U
-\xi T_0^{-1} \int_0^1 (z_\rho,z)_{U}\, d\rho + \overline{(Au,v)}_H - (Au,v)_H.
\eeqs
Hence, we get
\beqs
\Re ({\mathcal A}_g (u,v,z), (u,  v,  z))_{{\mathcal H}_g}
=- \mu \Re (z(1),B^*v)_{U}
- \frac{\xi}{2 T_0}\left\|z(1)\right\|_{U}^2\,
+\frac{\xi}{2 T_0}\|z(0)\|_{U}^2.
\eeqs
}
Hence reminding that $z(0)=B^*v$ and using Young's inequality we find that
$$
\Re ({\mathcal A}_g (u,v,z), (u,  v,  z))_{{\mathcal H}_g}
\leq \left(\frac{|\mu|}{2}-\frac{\xi}{2 T_0} \right)\|z(1)\|^2_{U}
+ \left(\frac{|\mu|}{2}
+\frac{\xi}{2 T_0} \right) \|B^*v\|^2_{U}.
$$
We find that
\beqs
\Re ({\mathcal A}_g (u,v,z), (u,  v,  z))_{{\mathcal H}_g}
\leq \left(\frac{|\mu|}{2}-\frac{\xi}{2 T_0} \right)\|z(1)\|_{U}^2
+ \left(\frac{|\mu|}{2}
+\frac{\xi}{2 T_0} \right) \|B^*\|^2_{\mathcal{L}(H,U)} \left\|v\right\|^2_H.
\eeqs
The choice of $\xi$ is equivalent to  $\frac{|\mu|}{2}-\frac{\xi}{2 T_0} \leq 0$,
and therefore for $c = \left(\frac{|\mu|}{2} + \frac{\xi}{2T_0} \right) \, \left\|B^*\right\|^2_{\mathcal{L}(H,U)}$,
\be
\label{dissp}
\Re ({\mathcal A}_g (u,v,z), (u,v,z))_{{\mathcal H}_g}
\leq \left(\frac{|\mu|}{2} - \frac{\xi}{2 T_0}\right)\|z(1)\|_{U}^2 + \left(\frac{|\mu|}{2}
+\frac{\xi}{2 T_0} \right) \|B^*\|^2_{\mathcal{L}(H,U)} \left\|v\right\|^2_H.
\ee
As $\|v\|^2_{H}\leq \|(u,v,z)\|^2_{\mathcal{H}_g}$, we get
\be
\label{dissipi}
\Re (({\mathcal A}_g - c Id)(u,v,z), (u,  v,  z))_{{\mathcal H}_g}
\leq \left(\frac{|\mu|}{2} - \frac{\xi}{2 T_0} \right)\|z(1)\|_{U}^2 \leq 0,
\ee
which directly leads to the dissipativeness of ${\mathcal A}^s_g:= {\mathcal A}_g - c Id$.

\medskip

Let us go on with the maximality, namely let us  show that $\lambda I -{\mathcal A}_g$ is surjective for a fixed $\lambda >0.$
Given $(f,g,h)^T\in {\mathcal H}_g,$ we look for a   solution $U=(u,v,z)^T\in {\mathcal D}({\mathcal A}_g) $ of
\be\label{surj}
(\lambda I- {\mathcal A}_g) \left (
\begin{array}{l}
u\\
v\\
z
\end{array}
\right )=
 \left (
\begin{array}{l}
f\\
g\\
h
\end{array}
\right ),
\ee
that is, verifying
\begin{equation}\label{max}
\left\{
\begin{array}{l}
\lambda u-v=f,\\
\lambda v + Au + \mu \, Bz(1) = g,\\
\lambda z + T_0^{-1}z_{\rho}=h.
\end{array}
\right.
\end{equation}
Suppose that we have found $u$ with the appropriate regularity. Then,
\begin{equation}\label{numerare}
v=\lambda u -f
\end{equation}
and we can determine $z.$ Indeed, by (\ref{domain}),
\begin{equation}\label{F3}
z(0)= B^*v,
\end{equation}
and, from (\ref{max}),
\begin{equation}\label{F4}
\lambda z(\rho ) + T_0^{-1} z_{\rho} (\rho )=h(\rho)\quad \mbox{\rm for } \ \rho\in (0,1).
\end{equation}
Then, by (\ref{F3}) and (\ref{F4}), we obtain
\be\label{defz}
z(\rho) = \lambda B^*u \, e^{-\lambda \rho T_0} - B^*f \, e^{-\lambda\rho T_0} + T_0 e^{-\lambda\rho T_0} \int_0^{\rho} h(\sigma ) e^{\lambda\sigma T_0} d\sigma.
\ee
In particular, we have
\begin{equation}\label{F5tilde}
z(1)=   \lambda B^*u \, e^{-\lambda T_0} - B^*f \, e^{-\lambda T_0} + z_0,
\end{equation}
with $z_0\in U$ defined by
\begin{equation}\label{F5star}
z_0=T_0 e^{-\lambda T_0}\int_0^1 h(\sigma ) e^{\lambda\sigma T_0} d\sigma.
\end{equation}
This expression in (\ref{max}) shows that the function $u$ verifies formally
$$
\lambda^2 u + Au + \lambda \, \mu \, BB^* u \, e^{-\lambda T_0} - \mu \, BB^* f \, e^{-\lambda T_0} + \mu \, Bz_0 = g+\lambda f,$$
that is,
\begin{equation}\label{F6}
\lambda^2 u + Au + \lambda \, \mu \, BB^* u \, e^{-\lambda T_0} = g+\lambda f + \mu \, BB^* f \, e^{-\lambda T_0} - \mu \, Bz_0.
\end{equation}
Problem (\ref{F6}) can be reformulated as
\begin{equation}\label{F7}
(\lambda^2 u+ Au + \lambda \, \mu \, BB^* u \, e^{-\lambda T_0}, w)_H =
(g+\lambda f + \mu \, BB^* f \, e^{-\lambda T_0} - \mu \, Bz_0,w)_H,  \quad \forall\
w \in H_\half.
\end{equation}
Using the definition of the adjoint of $B$, we get
\be 
\label{F8}
\lambda^2  (u, w)_H +(A^\half u,A^\half w)_{{\color{black}H}} +   \lambda \, \mu \, e^{-\lambda T_0} \,  (B^*u,  B^*w)_{U} =
\ee
$$
(g+\lambda f,w)_H + \mu \, (e^{-\lambda T_0} \, B^*f - z_0, B^*w)_{U},
\ \forall\,
w\in H_\half.
$$
As the left-hand side of (\ref{F8})
is coercive on $H_\half$, for $\lambda$ sufficiently large (for example {\color{black}$\lambda \geq |\mu| \, \left\|B^*\right\|^2_{\mathcal{L}(H,U)}$)},
the Lax--Milgram lemma guarantees the existence and uniqueness of a solution $u\in H_\half$ of (\ref{F8}).
Once $u$ is obtained we define $v$ by \rfb{numerare} that belongs to $H_\half$
and $z$ by \rfb{defz} that belongs to $H^1(0,1; U)$.
Hence we can set $r=A^\half u$, it belongs to $H$ but owing to \rfb{F8}, it fulfils
\[
\lambda (v,w)_H+ (r, A^\half w)_{H}= (g - \mu \, Bz(1),w)_H, \ \forall w\in H_\half,
\]
or equivalently
\[
(r, B^*w)_{H}= (g - \mu \, Bz(1) -\lambda v ,w)_H,\  \forall w\in H_\half.
\]
As $g - Bz(1) -\lambda v\in H$, this implies that $r$ belongs to $H_\half$ with
\[
A^\half r=g - Bz(1) -\lambda v.
\]
This shows that the triple $U=(u,v,z)$ belongs to ${\mathcal D} ({\mathcal A}_g)$ and
satisfies \rfb{surj}, hence $\lambda I - {\mathcal A}_g$ is surjective
for every $\lambda >0.$
\end{proof}

We have then the following result.
\begin{proposition}\label{propexistunic}
We assume that $\xi \geq |\mu| \, T_0.$ Then, the system
\rfb{damped1}--\rfb{output} is well--posed. More precisely, for every
$(u_0,u_1,0) \in \mathcal{H}_g$, there exists a unique solution $(u,v,z) \in C(0,+\infty, \mathcal{H}_g)$ of \rfb{formulA}. Moreover, if $(u_0,u_1,0)\in {\mathcal D}({\mathcal A}_g)$ then $(u,v,z) \in C(0,+\infty, \mathcal{D}({\mathcal A}_g)) \cap C^1(0,+\infty, {\mathcal H}_g)$ with $v=\dot{u}$ and $u$ is indeed a solution  of \rfb{damped1}--\rfb{output}.
\end{proposition}

\section{Asymptotic behavior} \label{sectasanalysis}
In this section, we show that the semigroup $e^{t{\mathcal A}^s_g},$  where $({\mathcal A}^s_g = {\mathcal A}_g - c Id, \mathcal{D}({\mathcal A}^s_g) = \mathcal{D}({\mathcal A}_g)),$ decays to the null steady state with an exponential decay rate. To obtain this, our technique is based on a frequency domain approach and combines a contradiction argument to carry out a special analysis of the resolvent.

\begin{theorem} \label{lr}
We assume that $\xi > |\mu| \, T_0$. Then, there exist  constants $C, \omega > 0$ such that the semigroup $e^{t{\mathcal A}^s_g}$ satisfies the following estimate
\be
\label{estexp}
\left\| e^{t {\mathcal A}^s_g}\right\|_{\mathcal L({\mathcal H}_g)} \leq C \, e^{-\omega t}, \ \forall \ t > 0.
\ee
\end{theorem}

\begin{proof} [Proof of Theorem  \ref{lr}]
We will employ the following frequency domain theorem for uniform stability from
\cite[Thm 8.1.4]{JacobZwart} of a $C_0$ semigroup of contractions on a Hilbert space:

\begin{lemma}
\label{pruss}
A $C_0$ semigroup of contractions $e^{t{\mathcal L}}$ on a Hilbert space $X$ satisfies
$$||e^{t{\mathcal L}}||_{{\mathcal L}(X)} \leq C \, e^{-\omega t},$$
for some constant $C >0$ and for $\omega>0$ 
if and only if
\be
\sigma (\mathcal L) \cap i \mathbb{R} = \emptyset, \label{1.8w} 
\ee
and 
\be 
\limsup_{\beta \in \mathbb{R}, |\beta| \rightarrow + \infty}  \|(i \beta I -{\mathcal L})^{-1}\|_{{\mathcal L}(X)} <\infty. 
\label{1.9}
\ee
where $\sigma({\mathcal L})$ denotes the spectrum of the operator
${\mathcal L}$.
\end{lemma} 

In view of this theorem we need to identify the spectrum of ${\mathcal A}^s_g$ lying on the imaginary axis. Unfortunately, as the embedding of $H^1(0,1;U)$
into $L^2(0,1;U)$ is not compact in general, ${\mathcal A}^s_g$ has not a compact resolvent. Therefore its spectrum $\sigma({\mathcal A}_g^s)$ does not consist only of eigenvalues of ${\mathcal A}^s_g$. We have then to
show that:

\begin{itemize}
\item
if $\beta$ is a real number, then $i \beta Id - {\mathcal A}_g^s$ is injective and
\item
if $\beta$ is a real number, then  $i \beta Id - {\mathcal A}_g^s$ is surjective.
\end{itemize}
It is the objective of the two following lemmas.

First we look at the point spectrum of ${\mathcal A}_g^s$.
\begin{lemma} \label{inj}
We assume that $\xi > |\mu| \, T_0$. Then, if $\beta$ is a real number, then $i \beta$ is not an eigenvalue of ${\mathcal A}_g^s$.
\end{lemma}

\begin{proof}
We will show that the equation
\be
\label{(3.5)}
{\mathcal A}_g^s Z = i \beta Z 
\ee
with $Z = (u,v,z)^\top  \in {\mathcal  D}({\mathcal A}_g)$ and $ \beta \in \mathbb{R}$ has only the trivial solution.

Equation \rfb{(3.5)} writes:
\be
\label{(3.6)}
(i \beta + c) u - v = 0,
\ee
\be
\label{(3.61)}
(i \beta + c) v + A u + \mu Bz(1) = 0,
\ee
\be
\label{(3.7)}
(i \beta + c) z + T_0^{-1} \, z_\rho = 0.
\ee

By taking the inner product of \rfb{(3.5)} with $Z$ and using \rfb{dissipi}, we get:
\be
\label{(3.8)}
\Re \left(( {\mathcal A}_g^s Z,Z )_{{\mathcal H}_g} \right) \leq  \left(\frac{|\mu|}{2} - \frac{\xi}{2T_0} \right) \, \left\|z(1) \right\|^2_U.
\ee
Thus we firstly obtain that:
$$
z(1) = 0, \, \hbox{in} \, U,
$$
and by \rfb{(3.7)} we have that $z(\rho) = z(1) \, e^{-T_0 (i\beta + c) \rho}, \, \forall \, \rho \in (0,1)$, so $z = 0, \, \hbox{in} \, H^1(0,1;U).$

\medskip

Next, according to \rfb{(3.6)}, we have  $v = (i \beta + c) u$. Moreover, 
\rfb{(3.61)} implies

$$
c (c + i\beta) \, \left\|u \right\|^2_H + \left\|A^\half u\right\|^2_H = i\beta (c + i\beta) \, \left\|u\right\|^2_H
$$
and then
$$
0 \leq c^2 \, \left\|u \right\|^2_H + \left\|A^\half u\right\|^2_H = - \beta^2 \, \left\|u\right\|^2_H \leq 0 \ .
$$
This leads to $u=0$ and next $v = 0$.

Thus the only solution of \rfb{(3.5)} is the trivial one.

\end{proof}

Next, we show that ${\mathcal A}_g^s$ has no continuous spectrum on the imaginary axis.

\begin{lemma} \label{surject}
If $\beta$ is a real number, then $i\beta$ belongs to the resolvent set $\rho({\mathcal A}_g^s)$ of $\mathcal{A}_g^s$.
\end{lemma}

\begin{proof}
In view of Lemma \ref{inj} it is enough to show that $i \beta Id -{\mathcal A}_g^s$ is surjective.

For $F = (f,g,h)^\top  \in {\mathcal H}_g,$ we look for a   solution $U=(u,v,z)^\top \in {\mathcal D}({\mathcal A}_g) $ of
\be
\label{surjsurj}
(\lambda Id- {\mathcal A}_g) \left (
\begin{array}{l}
u\\
v\\
z
\end{array}
\right )=
 \left (
\begin{array}{l}
f\\
g\\
h
\end{array}
\right ),
\ee
that is, verifying
\begin{equation}\label{maxsurj}
\left\{
\begin{array}{l}
\lambda u-v=f,\\
\lambda v + A u + \mu \, Bz(1) = g,\\
\lambda z + T_0^{-1}z_{\rho}=h,
\end{array}
\right.
\end{equation}
where $\lambda = i \beta + c$.

\medskip

Suppose that we have found $u$ with the appropriate regularity. Then,
\begin{equation}\label{surjsurjbis}
v=\lambda u -f
\end{equation}
and we can determine $z.$ Indeed, by (\ref{domain}),
\begin{equation}\label{F3surj}
z(0)= B^*v,
\end{equation}
and, from (\ref{maxsurj}),
\begin{equation}\label{F4surj}
\lambda z(\rho ) + T_0^{-1} z_{\rho} (\rho )=h(\rho)\quad \mbox{\rm for } \ \rho\in (0,1).
\end{equation}
Then, by (\ref{F3surj}) and (\ref{F4surj}), we obtain
\be\label{defzsurj}
z(\rho) = B^*v \, e^{-\lambda\rho T_0} + T_0 e^{-\lambda\rho T_0} \int_0^{\rho} h(\sigma ) e^{\lambda\sigma T_0} d\sigma.
\ee
In particular, we have
\begin{equation}\label{F5tildesurj}
z(1)=   B^*v \, e^{-\lambda T_0} + z_0,
\end{equation}
with $z_0\in U$ defined by
\begin{equation}\label{F5starsurj}
z_0=T_0 \, e^{-\lambda T_0}\int_0^1 h(\sigma ) e^{\lambda\sigma T_0} d\sigma.
\end{equation}
This expression in (\ref{maxsurj}) shows that the function $u$ verify  formally
\be
\label{F6surj}
\lambda^2 u + Au + \lambda \, \mu \, BB^* u \, e^{-\lambda T_0} = g+\lambda f + \mu \, BB^* f \, e^{-\lambda T_0} - \mu \, Bz_0.
\ee
Problem (\ref{F6surj}) can be reformulated as
\begin{equation}\label{F7surj}
\lambda \,   (u, w)_H + \frac{1}{\lambda} \, \, (A^\half u,A^\half w)_{H} + \mu \, e^{-\lambda T_0} \,  (B^*u,  B^*w)_{U} =
\ee
$$
 \frac{1}{\lambda} \, (g+\lambda f,w)_H + \frac{1}{\lambda} \, \mu \, (e^{-\lambda T_0} \, B^*f - z_0, B^*w)_{U},
\ \forall\,
w\in H_\half.
$$

\medskip

As the left-hand side of (\ref{F7surj})
is coercive sesquilinear form on $H_\half$,
the Lax--Milgram lemma guarantees the existence and uniqueness of a solution $u \in H_\half$ of (\ref{F6surj}).
Once $v$
and $z$ by \rfb{defzsurj} that belongs to $H^1(0,1; U)$.
Hence we can set $r= A^\half u$, it belongs to $H$ but owing to \rfb{maxsurj}, it fulfils
\[
\lambda (v,w)_H+ (r, A^\half w)_{H}= (g - \mu \, Bz(1),w)_H, \ \forall w\in H_\half,
\]
or equivalently
\[
(r, B^*w)_{H}= (g - \mu \, Bz(1) -\lambda v ,w)_H,\  \forall w\in H_\half.
\]
As $g - \mu \, Bz(1) -\lambda v\in H$, this implies that $r$ belongs to $H_\half$ with
\[
A^\half r=g - \mu \, Bz(1) -\lambda v.
\]

\medskip

This shows that the triple $U=(u,v,z)^\top $ belongs to ${\mathcal D} ({\mathcal A}_g)$ and satisfies \rfb{surjsurj}, hence $i \beta Id - {\mathcal A}_g^s$ is surjective.
\end{proof}

The following lemma shows that \rfb{1.9} holds with $\mathcal{L} = {\mathcal A}_g^s$.

\begin{lemma}\label{lemresolvent}
We assume that $\xi > |\mu| \, T_0$. Then, the resolvent operator of ${\mathcal A}_g^s$ satisfies condition
\be \limsup_{\beta\in \mathbb{R}, |\beta| \rightarrow + \infty}  \|(i \beta I -{\mathcal A}_g^s)^{-1}\|_{{\mathcal L}({\mathcal H}_g)} <\infty. \label{1.9bis}
\ee
\end{lemma}

\begin{proof}
Suppose that condition \eqref{1.9bis} is false.
By the Banach-Steinhaus Theorem (see \cite{brezis}), there exists a sequence of real numbers $\beta_n$ such that $|\beta_n| \rightarrow +\infty$ and a sequence of vectors
$Z_n= (u_{n},v_{n},z_n )^\top \in {\mathcal D}({\mathcal A}_{{\color{black}g}})$ with
\be
\label{cont}
\|Z_n\|_{{\mathcal H}_g} = 1
\ee
such that
\be
|| (\lambda_n Id - {\color{black}{\mathcal A}_g})Z_n||_{{\mathcal H}_g} \rightarrow 0\;\;\;\; \mbox{as}\;\;\;n\rightarrow \infty,
\label{1.12} \ee
i.e.,
\be \lambda_n u_{n} - v_{n} \equiv f_{n}\rightarrow 0 \;\;\; \mbox{in}\;\; H_\half,
\label{1.13}\ee
\be
\lambda_n v_{n} + A u_{n} + \mu \, Bz_n(1)   \equiv g_{n} \rightarrow 0 \;\;\;
\mbox{in}\;\; H,
\label{1.13b} \ee
\be \label{1.14}
 \lambda_n \, z_n + T_0^{-1} \partial_\rho z_n \equiv h_n \rightarrow 0 \; \; \; \mbox{in} \;\; L^2(0,1;U),
\ee
where $\lambda_n = c + i \beta_n, \, n \in \mathbb{N}$. 

\medskip

Our goal is to derive from \eqref{1.12} that $||Z_n||_{{\mathcal H}_g}$ converges to zero, that furnishes a contradiction.

We notice that from \rfb{dissipi} and \rfb{1.13} we have
\beqs
|| (\lambda_n Id - {\mathcal A}_g)Z_n||_{{\mathcal H}_g} \geq |\Re \left((
\lambda_n Id - {\mathcal A})Z_n, Z_n\right)_{{\mathcal H}_g} |
\geq  \, \left(- \frac{|\mu|}{2} + \frac{\xi}{2T_0}\right) \,
\left\|z_n(1)\right\|^2_{U}.
\eeqs
 
By this estimate, we deduce that
\be
\label{cvk}
z_{n}(1) \rightarrow 0,\  \hbox{ in } U,
\hbox{ as } n\to\infty.
\ee
According to \rfb{1.13}-\rfb{1.13b} we have that 
$$
c \, \left( \left\|v_{n}\right\|^2_H + 
\left\|A^\half u_{n} \right\|^2_H \right) =
$$
$$
\Re \left( \lambda_n \, \left\|v_{n}\right\|^2_H + \overline{\lambda_n} \, 
\left\|A^\half u_{n} \right\|^2_H \right) = \Re \left( (g_{n} - \mu Bz_n(1), v_{n})_H  + (A^\half u_{n}, A^\half f_{n})_H \right).
$$
Which implies that 
\be
\label{ucv}
u_{n} \rightarrow 0, \, \hbox{in} \, H_\half, \, v_{n} \rightarrow 0, \hbox{ in } H, \hbox{ as } n\to\infty.
\ee
As well as
\be
\label{c5}
z_n(0) = B^*v_{n} \rightarrow 0, \;\;\;
\mbox{in}\;\; U,  \hbox{ as } n\to\infty.
\ee
By integration of the identity \rfb{1.14}, we have
\be
\label{c7}
z_n(\rho) = z_n (0) \, e^{-T_0 \lambda_n \rho} + {\color{black}T_0} \, \int_0^\rho e^{-T_0 \lambda_n (\rho - \gamma)} \, h_n(\gamma) \, d \gamma.
\ee
Hence recalling that $\Re\lambda_n = c > 0$
\[
\int_0^1 \|z_n(\rho) \|^2_{U}\,d\rho\leq
2 \|z_n (0)\|_{U}^2
+2T_0^2 \int_0^1 \int_0^\rho \|h_n(\gamma)  \|^2_{U} \, d \gamma \rho\,d\rho\to 0, \hbox{ as } n\to\infty.
\]
All together we have shown that $\|Z_n\|_{{\mathcal H}_g}$ converges to zero, that clearly contradicts $\left\|Z_n\right\|_{{\mathcal H}_g}~=~1$.
\end{proof}

The two hypotheses of Lemma \ref{pruss} are proved, then \rfb{estexp} holds. The proof of Theorem \ref{lr} is then finished.
\end{proof}
\section{Numerical approximation in 1D}\label{Numerics}
This section is devoted to the construction of a numerical approximation of the considered problem by either a finite difference discretization or a finite volume method.
For each studied case, we will firstly construct in detail a discrete problem, present the corresponding algorithm and we will define its corresponding discrete energy which has to be conserved in the first step of the stabilization procedure.
\textcolor{black}
{\begin{remark}
Let us mention that for every case  treated in this work, simple algebraic computations show that the design of the discrete energy (which is conserved in the first step) is also made to guarantee its monotonically decreasing when no delay is acting i.e. $T = 0$.
But this design does not ensure that the discrete energy decreases monotonically overtime in the presence of the delay.
\end{remark}
\begin{remark}
In this work, we have chosen to present only explicit numerical methods that are simpler and easily implemented although they are constrained by a CFL condition between the time step $\Delta t$ and the size of the mesh $\Delta x$. 
These methods are easily transformed in implicit numerical methods by writing the problem for time $t^n = n \Delta t$ as the mean value of the problem for time 
$t^{n+1} = (n+1) \Delta t$ and  for time $t^{n-1} = (n-1) \Delta t$ altough the implicit methods needs a higher computational cost (see Remark \ref{remark-implcit} for the details of the construction when using a finite volume method)
\end{remark}
}
We will consider $\Omega = (0,\ell)$ and set $T = T_0 = 2  \ell$, {\color{black}which is precisely the optimal time of observability, which corresponds to the minimal time required for the wave to travel back and forth with speed $c=1$ over an interval of length $\ell$}.

\subsection{Boundary case}
Let $\mu \in \mathbb{R}$. We consider the following switching time delay problem:
\begin{align}
u_{tt}(x,t) - u_{xx}(x,t)&=0 \quad \mbox{\rm for }(x,t) \in (0,\ell) \times (0,+\infty) \label{wave-boundary}\\
u(0,t)  &=  0 \quad \mbox{ for } t > 0 \label{Dirichlet-boundary}\\
u_x(\ell,t) &= 0   \quad \mbox{ for } t \in (0,T)\label{Neuman}\\
u_x(\ell,t)&=  \mu u_t(\ell, t-T)  \quad \mbox{ for } t \geq T \label{Neuman-delay}
\end{align} with the following initial data:
\begin{equation} \label{init-boundary}
u(x,0)=u_0(x), \quad x \in (0,\ell)
\end{equation}
and 
\begin{equation} \label{init1-boundary}
u_t(x,0)=u_1(x) \quad x \in (0,\ell)
\end{equation}
Here, $A = - \partial^2_x$ be the unbounded operator in $H = L^2(0,\ell)$
with domain
$$H_{1}={\mathcal D}(A) = \left\{u \in H^2(0,\ell); \, u(0) = 0, \, u_x(\ell) = 0 \right\},$$
$$
H_{\frac{1}{2}}= {\mathcal D}(A^{\frac{1}{2}}) = \left\{u \in H^1(0,\ell); \, u(0) =0 \right\}$$
and
$$ U = \R, \, 
B \in {\mathcal L}(\R, H_{-\frac{1}{2}}), \, Bk = A_{-1} N k = k \, \delta_\ell, \forall \, k \in \R, \, B^*u = u(\ell), \, \forall \, u \in H_{\frac{1}{2}},
$$
where $A_{-1}$ is the extension of $A$ to $H_{-1} = ({\mathcal D}(A))^\prime$ and $N$ is the Neumann map defined by :
\begin{align*}
\partial_x^2(Nk) &= 0 \mbox{ on } (0,\ell)\\
Nk(0) &= 0 \\
\partial_x(Nk)(\ell) &= k 
\end{align*}
and 
$H_{-\frac{1}{2}} = H_{\frac{1}{2}}^\prime$ (the duality is in the sense of $H$).

We have according to \cite[Ammari-Chentouf-Smaoui]{ACS1,ACS2} (which generalize results of Gugat \cite{g} and \cite{gt}) the following stability result:
\begin{theorem}\cite[Ammari-Chentouf-Smaoui]{ACS1} \label{princf}
\begin{enumerate}
\item 
For any $\mu \in (0, 1)$
there exist  positive constants $C_1, C_2$ such that for all initial data in ${\mathcal H}$, the  solution of problem
\eqref{wave-boundary}-\eqref{init1-boundary} satisfies
\begin{equation}\label{expestimate-boundary}
E(t)\le C_1 \, e^{- \, C_2t}.
\end{equation}
The constant $C_1$ depends  on the initial data, on $\ell$ and on $\mu$, while $C_2$
depends only  on $\ell$ and on $\mu$. 
\item 
For $\mu = 1$, if we denote $\mathcal{S}_{b,1}(t)$ the propagator of the boundary delayed control problem, we have by definition: 
\[
\forall t > 0 \,,\, \mathcal{S}_{b,1}(t)
\begin{pmatrix}
u_{0} \\
u_{1}
\end{pmatrix} =
\begin{pmatrix}
u \\
u_{t}
\end{pmatrix}. 
\]
And we have the following:
\[
\forall t > 0 \,,\, \forall n \in \N^* \,,\, \mathcal{S}_{b,1}(t+ 2n T)
\begin{pmatrix}
u_{0} \\
u_{1}
\end{pmatrix} = - \, 
\mathcal{S}_{b,1}(t)
\begin{pmatrix}
u_{0} \\
u_{1}
\end{pmatrix} 
\]
wehereas 
\[
\forall t > 0 \,,\, \forall n \in \N \,,\, \mathcal{S}_{b,1}(t+ (2n+1) T)
\begin{pmatrix}
u_{0} \\
u_{1}
\end{pmatrix} = 
\mathcal{S}_{b,1}(t)
\begin{pmatrix}
u_{0} \\
u_{1}
\end{pmatrix}.
\]
In particular, we have that $\mathcal{S}_{b,1}$ is $3T$- periodic.
\end{enumerate}
\end{theorem}

We note here that the proof of the second assertion of the above Theorem \ref{princf} is a simple adaptation of the proof of \cite[Theorem 2.3 and Corollary 2.5]{ACS1} for $\mu = 1$. 

%
%
\subsubsection{\textbf{\textit{Construction of the numerical scheme}}}
Let $N$ be a non negative integer. Let $\Delta x~=~\dfrac{\ell}{N} \ $. Consider the uniform subdivision of $[0,\ell]$ given by:
\[
0=x_{0} <x_{1}< ...<x_{N-1}<x_{N}=\ell, \quad \textit{ i.e. }  x_{j} =j \Delta x \,,\, j = 0,\ldots,N \ .
\]
Set $t^{n+1}-t^{n}=\Delta t$ for all $n \in \mathbb{N}$. We will suppose that $T = K \times \Delta t \,,\, K \in \N^{*}$ to write easily the discretization of the delay term. 
We will also suppose that $T_{f} = M \Delta t$, with $M > K $, be the final time. 

\textbf{First step: for time $\boldsymbol{t \in [0,T)}.$}

For interior points $x \in (0,\ell)$ and for time $0 \leq t < T$, the explicit finite-difference discretization of equation \eqref{wave-boundary} writes for $n =0, \ldots, K-1 \,,\, $ and $j= 1, \ldots, N-1$: 
\begin{align} \label{discrete-wave-boundary}
\dfrac{u_j^{n+1}-2u_j^n +u_j^{n-1}}{\Delta t^2} - \dfrac{u_{j+1}^{n} -2 u_j^n +u_{j-1}^{n}}{\Delta x^2} &= 0
\end{align}
The Dirichlet boundary condition \eqref{Dirichlet-boundary} at $x = 0$ reads: for $ n = 0, \ldots, K-1$, 
\[ 
u_{0}^{n} = 0 \ .
\] 
The Neumann boundary condition \eqref{Neuman} is commonly written as: $u_{N} = u_{N-1}$ since the derivative $u_{x}(x_{N},t^{n}) = 0$ is approximated by the quotient of difference:
\[
u_{x}(x_{N},t^{n}) \approx \dfrac{u_{N}^{n} - u_{N-1}^{n}}{\Delta x} \ .
\]
By proceeding this way, the spatial order of discretization becomes now $\Delta x$ and it may induce instabilities.

Thus we proceed as follow. The Neumann boundary condition reads: for $ n = 0, \ldots, K-1$,
\[  
\dfrac{u_{N+1}^{n} - u_{N-1}^{n}}{2 \Delta x}= 0 
\]
 where $u_{N+1}^{n}$ is the value of $u$ in the ``ghost'' space cell $(\ell,\ell+\Delta x)$. Putting the value of $u_{N+1}^{n}$ in the numerical discretization \eqref{discrete-wave-boundary}
permits us to write the equation verified by $u_{N}^{n+1}$  as: \\
for $ n = 0, \ldots, K-1$,
\begin{align} \label{discrete-Neumann}
\dfrac{u_N^{n+1}-2u_N^n +u_N^{n-1}}{\Delta t^2} - 2 \dfrac{u_{N-1}^{n} - u_N^n }{\Delta x^2} &= 0.
\end{align}

According to the initial conditions given by equations \eqref{init-boundary}, we have firstly: for $j =1,\ldots, N$,
\begin{equation*}
u^{0}_{j} = u_{0}(x_{j})\label{u0}
\end{equation*}
We can use the second initial conditions \eqref{init1-boundary} to find the values of $u$ at time $t^{1} = \Delta t$, by employing a ``ghost'' time-boundary \textit{i.e.} $t^{-1}= - \Delta t$
and the second-order central difference formula:
\begin{equation*}
\mbox{for } j =1,\ldots, N \,,\, u_{1}(x_{j}) =\left.\dfrac{\partial u}{\partial t}\right|_{x_{j},0} =  \dfrac{ u_{j}^{1} - u_{j}^{-1} } {2 \Delta t} + O(\Delta t^{2}) .
\end{equation*}
Thus we have $\mbox{for } j =1,\ldots, N$:
\begin{equation*}
 u_{j}^{-1}=  u_{j}^{1}  - 2 \Delta t \ u_{1}(x_{j}) \ .
 \end{equation*} 
Setting $n=0$, in the numerical scheme \eqref{discrete-wave-boundary}, the previous equalities permits to compute $(u_{j}^{1})_{j=0,N}$. 
Finally, the solution $u$ can be computed at any time $t^n$.

In order to compute the solution $u$ beyond the time $T$, we have to compute the quantity $u_t(\ell, t)$ for time $t \in [0,T)$. The centered difference scheme is used and we compute:
\[
v^{0} = u_{1}(\ell),
\]
\[
\mbox{ for } n=1,\ldots, K-1 \,,\, v^{n}  =  u_{t}(\ell,t^{n}) \approx \dfrac{u_{N}^{n+1} - u_{N}^{n-1}}{2 \Delta t}.
\]
We set $s = \left(\dfrac{\Delta t}{\Delta x}\right)^{2}$. Let us now summarize the computation of the solution and the Neumann boundary delay term:

\textbf{First step: for time $\boldsymbol{t \in [0,T)}$}

\noindent$\mbox{Initialization } \mbox{ for } j = 0,\ldots, N \,,\, u^{0}_{j} = u_{0}(x_{j}) \quad v^{0} = u_{1}(\ell).$\\
$ \mbox{Solution for } t = \Delta t $\\
$\hspace*{2cm} \mbox{for } j = 1,\ldots, N-1 \,,\, u^{1}_{j} = \dfrac{s}{2}\Big(u^{0}_{j+1} + u^{0}_{j-1}\Big)+ (1 - s) \, u^{0}_{i} + \Delta t \, u_{1}(x_{j}).$\\
$\hspace*{2cm} \mbox{Dirichlet boundary condition } \quad \quad u^{1}_{0} = 0.$\\
$\hspace*{2cm} \mbox {Neumann boundary condition}\quad \quad u^{1}_{N} = s\, u^{0}_{N-1} + (1 - s) \, u^{0}_{N} + \Delta t u_{1}(x_{N}).$\\
$ \mbox{Solution for } t \in (\Delta t,T) \mbox{ i.e. for } n=1,\ldots, K-1$\\
$\hspace*{2cm} \mbox{for } j = 1,\ldots, N-1 \,,\, u^{n+1}_{j} = s \, \Big(u^{n}_{j+1} + u^{n}_{j-1}\Big)+ 2 (1 - s) \, u^{n}_{i} - u^{n-1}_{j}.$\\
$\hspace*{2cm} \mbox{Dirichlet boundary condition } \quad \quad u^{n+1}_{0} = 0.$\\
$\hspace*{2cm} \mbox{Neumann boundary condition } \quad \ u^{n+1}_{N} = 2 (1-s)\, u^{n}_{N} + 2  s \, u^{n}_{N-1} - u^{n-1}_{N}.$\\
$\hspace*{2cm} \mbox{Delayed boundary equation } \quad \ v^{n} = \dfrac{u_{N}^{n+1} - u_{N}^{n-1}}{2 \Delta t}.$

\textbf{Second step: for time $\boldsymbol{t \in [T,T_{f}]}$}

The only novelty comes from the Neumann boundary condition \eqref{Neuman-delay}. As we have set $T = K \Delta t$,
the discretization of the Neumann boundary condition \eqref{Neuman-delay} reads: for $n=K,\ldots,M$,
\[
\dfrac{u_{N+1}^{n} - u_{N-1}^{n}}{2 \Delta x}= \mu v^{n-K}.
\]
Thus, this boundary condition leads to: $u_{N+1}^{n}= u_{N-1}^{n}  + 2 \Delta x \mu v^{n-K}$. Inserting this value in the numerical scheme \eqref{discrete-wave-boundary} permits us to write
\[
u^{n+1}_{N} = 2 (1-s)\, u^{n}_{N} + 2  s \, u^{n}_{N-1} - u^{n-1}_{N} + 2 s \Delta x \mu v^{n-K}.
\]

\noindent$\mbox{Solution for } t \in [T,T_{f}] \ i.e. \ \mbox{for } n=K,\ldots, M-1$\\
$\hspace*{2cm} \mbox{for } j = 1,\ldots, N-1 \,,\, u^{n+1}_{j} = s \, \Big(u^{n}_{j+1} + u^{n}_{j-1}\Big)+ 2 (1 - s) \, u^{n}_{i} - u^{n-1}_{j}$\\
$\hspace*{2cm} \mbox{Dirichlet boundary condition } \quad \quad u^{n+1}_{0} = 0.$\\
$\hspace*{2cm} \mbox{Neumann delay boundary condition } \quad \ u^{n+1}_{N} = 2 (1-s)\, u^{n}_{N} + 2  s \, u^{n}_{N-1} - u^{n-1}_{N} + 2 s \Delta x \mu v^{n-K}.$\\
$\hspace*{2cm} \mbox{Delayed boundary equation } \quad \ v^{n} = \dfrac{u_{N}^{n+1} - u_{N}^{n-1}}{2 \Delta t}.$
\subsubsection{\textbf{\textit{Discrete energy and CFL condition}}}
The aim of this section is to design a discrete energy that is preserved in the first step that is the free wave equation with Dirichlet and Neumann boundary condition. To this end, let us define:
\begin{itemize}
\item the discrete kinetic energy  as: 
\begin{equation}\label{discrete-kinetic-energy-boundary}
\displaystyle  E_{k}^{n} = \dfrac{1}{2} \sum_{j=0}^{N-1}\left(\dfrac{u_j^{n+1}-u_j^n}{\Delta t}\right)^2 + \dfrac{1}{4}\left(\dfrac{u_N^{n+1}-u_N^n}{\Delta t}\right)^{2}
\end{equation}
\item the discrete potential energy as: 
\begin{equation}\label{discrete-potential-energy-boundary}
\displaystyle  E_{p}^{n} = \dfrac{1}{2} \sum_{j=0}^{N-1}\left(\dfrac{u_{j+1}^{n}-u_j^n}{\Delta x}\right) \left(\dfrac{u_{j+1}^{n+1}-u_j^{n+1}}{\Delta x}\right).
\end{equation}
\end{itemize}
The total discrete energy is then defined as
\begin{equation}\label{discrete-energy}
	\mathcal{E}^{n} = E_{k}^{n} + E_{p}^{n}.
\end{equation}

\begin{proposition}  The discrete energy is conserved for all $t = 0, \ldots, T-\Delta t,$ \textit{ i.e.}
\[
\forall \, n = 1, \ldots, K-1 \,,\, \mathcal{E}^{n+1} = \mathcal{E}^{n} \, .
\] 
\end{proposition}
\begin{proof} For this sake, we multiply the equation \eqref{discrete-wave-boundary} by $(u_j^{n+1}-u_j^{n-1})$, we sum over $j = 1,\ldots,N~-~1$ and we obtain:
\begin{align}
	&\displaystyle\sum_{j=1}^{N-1}\dfrac{u_j^{n+1}-2u_j^n + u_j^{n-1}}{\Delta t^2}(u_j^{n+1}-u_j^{n-1})
	- \sum_{j=1}^{N-1} \dfrac{u_{j+1}^{n} - 2 u_j^n + u_{j-1}^{n}}{\Delta x^2} (u_j^{n+1}-u_j^{n-1})  = 0.
\label{discrete-energy-boundary}
\end{align}
\textit{Estimation of the first term of \eqref{discrete-energy-boundary}} We firstly have:
\begin{align}
	\displaystyle\sum_{j=1}^{N-1}\dfrac{u_j^{n+1}-2u_j^n + u_j^{n-1}}{\Delta t^2}(u_j^{n+1}-u_j^{n-1}) &= \displaystyle\sum_{j=1}^{N-1}\dfrac{(u_j^{n+1}-u_j^n)-(u_j^n -u_j^{n-1})}{\Delta t^2}
	\Big((u_j^{n+1}-u_j^n)+(u_j^n-u_j^{n-1})\Big)\nonumber\\
	&= \sum_{j=1}^{N-1} \bigg(\dfrac{u_j^{n+1}-u_j^n}{\Delta t}\bigg)^2 - {\color{black}\sum_{j=1}^{N-1} \bigg(\dfrac{u_j^{n}-u_j^{n-1}}{\Delta t}\bigg)^{2}}. \label{discrete-Ec-boundary}
\end{align}
\textit{Estimation of the second term of \eqref{discrete-energy-boundary}.} Using the same trick we have: 
\begin{align*}
	-\sum_{j=1}^{N-1} \dfrac{u_{j+1}^{n} - 2 u_j^n + u_{j-1}^{n}}{\Delta x^2} (u_j^{n+1}-u_j^{n-1}) &= 
	-\sum_{j=1}^{N-1} \dfrac{(u_{j+1}^{n} - u_j^n)-( u_j^n-u_{j-1}^{n})}{\Delta x^2} (u_j^{n+1}-u_j^{n-1}) \\
	&=  -\ \sum_{j=1}^{N-1}\dfrac{(u_{j+1}^n-u_j^n)(u_j^{n+1}-u_j^{n-1})}{\Delta x^2} \\
	&+ \ \sum_{j=1}^{N-1}\dfrac{(u_{j}^n-u_{j-1}^n)(u_j^{n+1}-u_j^{n-1})}{\Delta x^2}.
\end{align*}
So, by translation of index in the second term in the previous sum, and since $u_{0}^{n+1} = u_{0}^{n-1}=0$, we will have: 
\begin{align*}
	- \ \sum_{j=1}^{N-1} \dfrac{u_{j+1}^{n} - 2 u_j^n + u_{j-1}^{n}}{\Delta x^2} (u_j^{n+1}-u_j^{n-1}) &= -  \  \sum_{j=0}^{N-1}\dfrac{(u_{j+1}^n-u_j^n)(u_j^{n+1}-u_j^{n-1})}{\Delta x^2}  \nonumber\\
	&+  \ \sum_{j=0}^{N-2}\dfrac{(u_{j+1}^n-u_{j}^n)(u_{j+1}^{n+1}-u_{j+1}^{n-1})}{\Delta x^2}. \nonumber
\end{align*}
Thus we have: 
\begin{align}
	- \sum_{j=1}^{N-1} \dfrac{u_{j+1}^{n} - 2 u_j^n + u_{j-1}^{n}}{\Delta x^2} (u_j^{n+1}-u_j^{n-1}) &=   \sum_{j=0}^{N-2}\dfrac{(u_{j+1}^{n+1}-u_j^{n+1})(u_{j+1}^{n}-u_j^{n})}{\Delta x^2}  \nonumber\\
	&-\sum_{j=0}^{N-2}\dfrac{(u_{j+1}^{n-1}-u_{j}^{n-1})(u_{j+1}^{n}-u_{j}^{n})}{\Delta x^2} \nonumber\\
	& - \dfrac{(u_{N}^{n}-u_{N-1}^{n})(u_{N-1}^{n+1} - u_{N-1}^{n-1})}{\Delta x^{2}}. \label{discrete-Ep-boundary}
\end{align}
To treat the last term of \eqref{discrete-Ep-boundary}, we multiply \eqref{discrete-Neumann} by $(u_{N}^{n+1}-u_{N}^{n-1})$, to obtain:
\[
\left(\dfrac{u_N^{n+1}-2u_N^n +u_N^{n-1}}{\Delta t^2} - 2 \dfrac{u_{N-1}^{n} - u_N^n }{\Delta x^2}\right)(u_{N}^{n+1}-u_{N}^{n-1}) = 0.
\]
Thus we have: 
\begin{align*}
- \dfrac{(u_{N}^{n}-u_{N-1}^{n})(u_{N}^{n+1} - u_{N}^{n-1})}{\Delta x^{2}} &= \dfrac{1}{2}\left(\dfrac{u_N^{n+1}-2u_N^n +u_N^{n-1}}{\Delta t^2} \right)(u_{N}^{n+1} - u_{N}^{n-1}) \nonumber\\
& =\dfrac{1}{2} \dfrac{\left((u_N^{n+1}-u_N^n)-(u_N^n - u_N^{n-1}) \right)\left((u_{N}^{n+1} - u^{n}_{N})+(u^{n}_{N} - u_{N}^{n-1})\right)}{\Delta t^2} \nonumber \\
& = \dfrac{1}{2} \dfrac{(u_N^{n+1}-u_N^n)^{2}-(u_N^n - u_N^{n-1})^{2}}{\Delta t^{2}}.  \label{Neuman-Energy}
\end{align*}
Thus we obtain: 
 \begin{equation}
 \label{Neuman-Energy}
\dfrac{1}{2} \dfrac{(u_N^{n+1}-u_N^n)^{2}-(u_N^n - u_N^{n-1})^{2}}{\Delta t^{2}} + \dfrac{(u_{N}^{n}-u_{N-1}^{n})(u_{N}^{n+1} - u_{N}^{n-1})}{\Delta x^{2}} = 0.  
\end{equation}
Adding \eqref{discrete-energy-boundary} to \eqref{Neuman-Energy} and substituing \eqref{discrete-Ep-boundary} and \eqref{discrete-Ec-boundary} into \eqref{discrete-energy-boundary}, we get: \\ 
\begin{align} \label{Energy-conservation}
 \sum_{j=1}^{N-1} &\bigg(\dfrac{u_j^{n+1}-u_j^n}{\Delta t}\bigg)^2 + \dfrac{1}{2}\dfrac{(u_N^{n+1}-u_N^n)^{2}}{\Delta t^{2}} + \sum_{j=0}^{N-1}\dfrac{(u_{j+1}^{n+1}-u_j^{n+1})(u_{j+1}^{n}-u_j^{n})}{\Delta x^2} \nonumber\\ 
 &=\sum_{j=1}^{N-1} \bigg(\dfrac{u_j^{n}-u_j^{n-1}}{\Delta t}\bigg)^2 + \dfrac{1}{2}\dfrac{(u_N^{n}-u_N^{n-1})^{2}}{\Delta t^{2}} + \sum_{j=0}^{N-1}\dfrac{(u_{j+1}^{n}-u_j^{n})(u_{j+1}^{n-1}-u_j^{n-1})}{\Delta x^2}. 
 \end{align}
The preceding equation gives: 
\[
\forall \, n = 1, \ldots, K-1\,,\, \mathcal{E}^{n+1} = \mathcal{E}^{n} \, .
\] 
 \end{proof} 
By a standard von Neumann stability analysis (that is a discrete Fourier analysis, see for instance \cite{Ames-1992}), the numerical scheme is stable if and only if, the following Courant-Friedrichs-Lewy, CFL, condition holds:
\[ \Delta t\leq \Delta x \ .
\]
The number $\dfrac{\Delta t }{\Delta x}$ is called the CFL number and is denoted in the following by $CFL$.
\subsection{Internal case}
Let $\mu \in \mathbb{R}$. We consider the following switching time delay problem:
\begin{align}
u_{tt}(x,t) - u_{xx}(x,t)&=0 \quad \mbox{\rm for }(x,t) \in (0,\ell) \times (0,2 \ell), \label{wave-int}\\
u_{tt}(x,t)- u_{xx}(x,t) + \mu \, d(x) \, u_t(x,t-2 \ell) &= 0  \quad  \mbox{\rm in } (0,\ell) \times [2 \ell,+\infty),\label{wave-int-delay}\\
u(0,t)  = u(\ell,t) &=  0 \  \,\mbox{ for } t > 0, \label{dirichlet-int}
\end{align}
where $d \in L^\infty(0,\ell)$ is a positive function which satisfies that there exists an nonempty set $I \subset (0,\ell)$ such that $d(x) \geq C >0, \, a.e. \ x \in I,$ and $C$ is a constant,  
with the following initial data:
\begin{equation} \label{init-internal}
u(x,0)=u_0(x), \quad x \in (0,\ell)
\end{equation}
and 
\begin{equation} \label{init1-internal}
u_t(x,0)=u_1(x) \quad x \in (0,\ell).
\end{equation}
Here, $A = - \partial^2_x$ be the unbounded operator in $H = L^2(0,\ell)$
with domain:
\[
H_{1}={\mathcal D}(A) = H^2(0,\ell) \cap H^1_0(0,\ell), H_{\frac{1}{2}} = {\mathcal D}(A^{\frac{1}{2}} ) = H_0^1(0,\ell)
\]
and
\[U = H = L^2(0,\ell), 
B \in {\mathcal L}(H), \, Bk = B^*k = \sqrt{d} \, k , \, \forall \, k \in H \ .
\]

%
%
\subsubsection{\textbf{\textit{Construction of the numerical scheme}}}

Let $N$ be a non negative integer. Let $\Delta x~=~\dfrac{\ell}{N} \ $. Consider the uniform subdivision of $[0,\ell]$ given by:
\[
0=x_{0} <x_{1}< ...<x_{N-1}<x_{N}=\ell, \quad \textit{ i.e. }  x_{j} =j \Delta x \,,\, j = 0,\ldots,N \ .
\]
For the sake of simplicity, we will suppose that the set $I \subset (0,\ell)$ is chosen as $I = [x_{i_{0}}, x_{i_{1}}] \,,\, x_{i_{0}} \mbox{ and }  x_{i_{1}}$ being two mesh points.

Set $t^{n+1}-t^{n}=\Delta t$ for all $n \in \mathbb{N}$. We will suppose that $T = K \times \Delta t \,,\, K \in \N^{*}$ to write easily the discretization of the delay term. 
We will also suppose that $T_{f} = M \Delta t$, with $M > K $, be the final time.

\textbf{First step: for time $\boldsymbol{t \in [0,T)}$.}

We proceed exactly the same way as in the previous section, the only novelty is now the Dirichlet boundary condition at $x = \ell$.
For interior points $x \in (0,\ell)$ and for time $0 \leq t < T$, the explicit finite-difference discretization of equation \eqref{wave-int} writes for $n =0, \ldots, K-1\,,\, $ and $j= 1, \ldots, N-1$: 
\begin{align} \label{discrete-wave-internal}
\dfrac{u_j^{n+1}-2u_j^n +u_j^{n-1}}{\Delta t^2} - \dfrac{u_{j+1}^{n} -2 u_j^n +u_{j-1}^{n}}{\Delta x^2} &= 0.
\end{align}
The Dirichlet boundary condition \eqref{dirichlet-int} at $x = 0 \,,\, x = \ell$ reads: for $ n = 0, \ldots, K-1$, 
\begin{align} \label{discrete-dirichlet-internal}
u_{0}^{n} = u_{N}^{n} = 0 \ .
\end{align}
As in the previous section, using the initial condition \eqref{init-internal}-\eqref{init1-internal}, the scheme is defined for $t^{1} = \Delta t$.
Finally, the solution $u$ can be computed at any time $t^n$.

In order to compute the solution $u$ beyond the time $T$, we have to compute the quantity $d(x) u_t(x, t)$ for $x \in  [x_{i_{0}}, x_{i_{1}}]$ and for time $t \in [0,T)$. 
The centered difference scheme is used and we compute:
\begin{align*}
\mbox{for } j = i_{0},\,\ldots,\, i_{1}\,,\ & v_{j}^{0} = d(x_{j})u_{1}(x_{j}), \\
\mbox{ for } n=1,\ldots, K-1 \,,\, \mbox{for } j = i_{0},\,\ldots,\, i_{1}\,,\ &v_{j}^{n} =  d(x_{j}) \dfrac{u_{j}^{n+1} - u_{j}^{n-1}}{2 \Delta t}.
\end{align*}

\textbf{Second step: for time $\boldsymbol{t \in [T,T_{f}]}$}

The only novelty comes from the internal delay term. As we have set $T = K \Delta t$,
the discretization of the wave equation with internal delayed damping  and Dirichlet boundary condition reads: for $n=K,\ldots,M$,
\begin{align*}
\mbox{for } j = 1,\,\ldots,\, i_{0} - 1 \quad &\dfrac{u_j^{n+1}-2u_j^n +u_j^{n-1}}{\Delta t^2} - \dfrac{u_{j+1}^{n} -2 u_j^n +u_{j-1}^{n}}{\Delta x^2} = 0, \\
\mbox{for } j = i_{0},\,\ldots,\, i_{1}  \quad &\dfrac{u_j^{n+1}-2u_j^n +u_j^{n-1}}{\Delta t^2} - \dfrac{u_{j+1}^{n} -2 u_j^n +u_{j-1}^{n}}{\Delta x^2} \boldsymbol{+ \mu v_{j}^{n-K}}= 0, \\
\mbox{for } j = i_{1}+1,\,\ldots,\, N - 1 \quad &\dfrac{u_j^{n+1}-2u_j^n +u_j^{n-1}}{\Delta t^2} - \dfrac{u_{j+1}^{n} -2 u_j^n +u_{j-1}^{n}}{\Delta x^2} = 0, \\
&u_{0}^{n} = u_{N}^{n} = 0 
\end{align*}
whereas the computation of the internal delay damping term reads: for $n=K,\ldots,M$,
\begin{align*}
\mbox{for } j = i_{0},\,\ldots,\, i_{1}\,,\ &v_{j}^{n} =  d(x_{j}) \dfrac{u_{j}^{n+1} - u_{j}^{n-1}}{2 \Delta t}.
\end{align*}

As in the previous section, we set $s = \left(\dfrac{\Delta t}{\Delta x}\right)^{2}$. Let us now summarize the computation of the solution and the internal delay term.

\textbf{First step: for time $\boldsymbol{t \in [0,T)}$}

\noindent$\mbox{Initialization } \mbox{ for } j = 0,\ldots, N \,,\, u^{0}_{j} = u_{0}(x_{j}) \quad \mbox{for } j= i_{0},\,\ldots,\, i_{1}\,,\ v_{j}^{0} = d(x_{j})u_{1}(x_{j}),$\\
$ \mbox{Solution for } t = \Delta t$\\
$\hspace*{2cm} \mbox{for } j = 1,\ldots, N-1 \,,\, u^{1}_{j} = \dfrac{s}{2}\Big(u^{0}_{j+1} + u^{0}_{j-1}\Big)+ (1 - s) \, u^{0}_{i} + \Delta t \, u_{1}(x_{j})$\\
$\hspace*{2cm} \mbox{Dirichlet boundary condition } \quad \quad u^{1}_{0} = 0.$\\
$\hspace*{2cm} \mbox {Dirichlet boundary condition}\quad \quad u^{1}_{N} = 0.$\\
$ \mbox{Solution for } t \in (\Delta t,T) \mbox{ i.e. for } n=1,\ldots, K-1,$\\
$\hspace*{2cm} \mbox{for } j = 1,\ldots, N-1 \,,\, u^{n+1}_{j} = s \, \Big(u^{n}_{j+1} + u^{n}_{j-1}\Big)+ 2 (1 - s) \, u^{n}_{i} - u^{n-1}_{j}.$\\
$\hspace*{2cm} \mbox{Dirichlet boundary condition } \quad u^{n+1}_{0} = 0.$\\
$\hspace*{2cm} \mbox {Dirichlet boundary condition }\quad u^{n+1}_{N} = 0$\\
$\hspace*{2cm} \mbox {Internal delay term}\hspace*{2.1cm}\mbox{for } j = i_{0},\,\ldots,\, i_{1}\,,\ v_{j}^{n} =  d(x_{j}) \dfrac{u_{j}^{n+1} - u_{j}^{n-1}}{2 \Delta t}.$\\

\textbf{Second step: for time $\boldsymbol{t \in [T,T_{f}]}$}

Let $M \geq K $. We denote $T_{f} = M \Delta t$ the final time. The only novelty comes from the internal delay term defined for $x \in  [x_{i_{0}}, x_{i_{1}}]$.\\
$ \mbox{Solution for } t \in [T,T_{f}] \ i.e. \ \mbox{for } n=K,\ldots, M-1,$\\
$\hspace*{2cm} \mbox{for } j = 1,\ldots, i_{0}-1\,,\, u^{n+1}_{j} = s \, \Big(u^{n}_{j+1} + u^{n}_{j-1}\Big)+ 2 (1 - s) \, u^{n}_{i} - u^{n-1}_{j},$\\
$\hspace*{2cm} \mbox{for } j = i_{0},\ldots, i_{1}\,,\, u^{n+1}_{j} = s \, \Big(u^{n}_{j+1} + u^{n}_{j-1}\Big)+ 2 (1 - s) \, u^{n}_{i} - u^{n-1}_{j} - \mu \Delta t^{2} v_{j}^{n-K},$\\
$\hspace*{2cm} \mbox{for } j = i_{0}+1,\ldots, N-1,\,,\, u^{n+1}_{j} = s \, \Big(u^{n}_{j+1} + u^{n}_{j-1}\Big)+ 2 (1 - s) \, u^{n}_{i} - u^{n-1}_{j}.$\\
$\hspace*{2cm} \mbox{Dirichlet boundary condition } \quad \quad u^{n+1}_{0} = 0.$\\
$\hspace*{2cm} \mbox {Dirichlet boundary condition }\quad \quad u^{n+1}_{N} = 0.$\\
$\hspace*{2cm} \mbox{Internal delay term}\hspace*{2.1cm}\mbox{for } j = i_{0},\,\ldots,\, i_{1}\,,\ v_{j}^{n} =  d(x_{j}) \dfrac{u_{j}^{n+1} - u_{j}^{n-1}}{2 \Delta t}.$%
\subsubsection{\textbf{\textit{Discrete energy}}}
The aim of this section is to design a discrete energy that is preserved in the first step that is the free wave equation with Dirichlet boundary conditions. To this end, let us define:
\begin{itemize}
\item the discrete kinetic energy  as: 
\begin{equation*}
\displaystyle  E_{k}^{n} = \dfrac{1}{2} \sum_{j=1}^{N-1}\left(\dfrac{u_j^{n+1}-u_j^n}{\Delta t}\right)^2 \ ,
\end{equation*}
\item the discrete potential energy as: 
\begin{equation*}
E_{p}^{n} = \dfrac{1}{2} \sum_{j=0}^{N-1}\left(\dfrac{u_{j+1}^{n}-u_j^n}{\Delta x}\right) \left(\dfrac{u_{j+1}^{n+1}-u_j^{n+1}}{\Delta x}\right) \ .
\end{equation*}
\end{itemize}
The total discrete energy is then defined as
\begin{equation}\label{discrete-energy-internal}
	\mathcal{E}^{n} = E_{k}^{n} + E_{p}^{n}.
\end{equation}
\begin{proposition}  The discrete energy is conserved for all $t = 0, \ldots, T-\Delta t,$ \textit{ i.e.}
\[
\forall \, n = 1, \ldots, K-1 \,,\, \mathcal{E}^{n+1} = \mathcal{E}^{n} \, .
\] 
\end{proposition}
\begin{proof} For this sake, we multiply the equation \eqref{discrete-wave-internal} by $(u_j^{n+1}-u_j^{n-1})$, we sum over $j = 1,\ldots,N~-~1$ and we obtain:
\begin{align}
	&\displaystyle\sum_{j=1}^{N-1}\dfrac{u_j^{n+1}-2u_j^n + u_j^{n-1}}{\Delta t^2}(u_j^{n+1}-u_j^{n-1})
	- \sum_{j=1}^{N-1} \dfrac{u_{j+1}^{n} - 2 u_j^n + u_{j-1}^{n}}{\Delta x^2} (u_j^{n+1}-u_j^{n-1})  = 0.
\label{discrete-energy-internal-multiply}
\end{align}
\textit{Estimation of the first term of \eqref{discrete-energy-internal-multiply}} We firstly have:
\begin{align}
	\displaystyle\sum_{j=1}^{N-1}\dfrac{u_j^{n+1}-2u_j^n + u_j^{n-1}}{\Delta t^2}(u_j^{n+1}-u_j^{n-1}) &= \displaystyle\sum_{j=1}^{N-1}\dfrac{(u_j^{n+1}-u_j^n)-(u_j^n -u_j^{n-1})}{\Delta t^2}
	\Big((u_j^{n+1}-u_j^n)+(u_j^n-u_j^{n-1})\Big)\nonumber\\
	&= \sum_{j=1}^{N-1} \bigg(\dfrac{u_j^{n+1}-u_j^n}{\Delta t}\bigg)^2 -\sum_{j=1}^{N-1} \bigg(\dfrac{u_j^{n+1}-u_j^{n-1}}{\Delta t}\bigg)^{2}. \label{discrete-Ec-internal}
\end{align}
\textit{Estimation of the second term of \eqref{discrete-energy-internal-multiply}.} Using the same trick we have: 
\begin{align*}
	-\sum_{j=1}^{N-1} \dfrac{u_{j+1}^{n} - 2 u_j^n + u_{j-1}^{n}}{\Delta x^2} (u_j^{n+1}-u_j^{n-1}) &= 
	-\sum_{j=1}^{N-1} \dfrac{(u_{j+1}^{n} - u_j^n)-( u_j^n-u_{j-1}^{n})}{\Delta x^2} (u_j^{n+1}-u_j^{n-1}) \\
	&=  -\ \sum_{j=1}^{N-1}\dfrac{(u_{j+1}^n-u_j^n)(u_j^{n+1}-u_j^{n-1})}{\Delta x^2} \\
	&+ \ \sum_{j=1}^{N-1}\dfrac{(u_{j}^n-u_{j-1}^n)(u_j^{n+1}-u_j^{n-1})}{\Delta x^2}.
\end{align*}
So, by translation of index in the second term in the previous sum, and since 
\begin{align*}
u_{0}^{n+1} &= u_{0}^{n-1}=0 \ , \\
u_{N}^{n+1} &= u_{N}^{n-1}=0  \ ,
\end{align*}
we will have: 
\begin{align*}
	- \ \sum_{j=1}^{N-1} \dfrac{u_{j+1}^{n} - 2 u_j^n + u_{j-1}^{n}}{\Delta x^2} (u_j^{n+1}-u_j^{n-1}) &= -  \  \sum_{j=0}^{N-1}\dfrac{(u_{j+1}^n-u_j^n)(u_j^{n+1}-u_j^{n-1})}{\Delta x^2}  \nonumber\\
	&+  \ \sum_{j=0}^{N-1}\dfrac{(u_{j+1}^n-u_{j}^n)(u_{j+1}^{n+1}-u_{j+1}^{n-1})}{\Delta x^2}. \nonumber
\end{align*}
Thus we have: 
\begin{align}
	- \sum_{j=1}^{N-1} \dfrac{u_{j+1}^{n} - 2 u_j^n + u_{j-1}^{n}}{\Delta x^2} (u_j^{n+1}-u_j^{n-1}) &=   \sum_{j=0}^{N-1}\dfrac{(u_{j+1}^{n+1}-u_j^{n+1})(u_{j+1}^{n}-u_j^{n})}{\Delta x^2}  \nonumber\\
	&-\sum_{j=0}^{N-1}\dfrac{(u_{j+1}^{n-1}-u_{j}^{n-1})(u_{j+1}^{n}-u_{j}^{n})}{\Delta x^2}. \label{discrete-Ep-internal}
\end{align}
Substituting \eqref{discrete-Ep-internal} and \eqref{discrete-Ec-internal} into \eqref{discrete-energy-internal-multiply}, we get \\ 
\begin{align} \label{Energy-conservation-internal}
 \sum_{j=1}^{N-1} &\bigg(\dfrac{u_j^{n+1}-u_j^n}{\Delta t}\bigg)^2 + \sum_{j=0}^{N-1}\dfrac{(u_{j+1}^{n+1}-u_j^{n+1})(u_{j+1}^{n}-u_j^{n})}{\Delta x^2} \nonumber\\ 
 &=\sum_{j=1}^{N-1} \bigg(\dfrac{u_j^{n}-u_j^{n-1}}{\Delta t}\bigg)^2  + \sum_{j=0}^{N-1}\dfrac{(u_{j+1}^{n}-u_j^{n})(u_{j+1}^{n-1}-u_j^{n-1})}{\Delta x^2}.
 \end{align}
The preceding equation gives: 
\[
\forall \, n = 1, \ldots, K-1\,,\, \mathcal{E}^{n+1} = \mathcal{E}^{n} \, .
\] 
 \end{proof}
As in the previous case, in order to obtain a stable numerical scheme, the following Courant-Friedrichs-Lewy, CFL, condition holds:
\[ \Delta t\leq \Delta x \ .
\]
\subsection{Pointwise case}
Let $\mu \in \mathbb{R}$, $\ell > 0 \,,\, \xi \in (0,\ell)$. We consider the following switching time delay problem:
\begin{align}
	u_{tt}(x,t) - u_{xx}(x,t)&=0 \quad \mbox{\rm for }(x,t) \in (0,\ell) \times (0,2 \ell) \label{wave-pointwise-free}\\
	u_{tt}(x,t) - u_{xx}(x,t) + \mu \, u_{t}(\xi,t-2 \ell) \, \delta_{\xi} &=0 \quad \mbox{\rm for }(x,t) \in (0,\ell) \times (2 \ell,+\infty) \label{wave-pointwise}\\
	u(0,t)  &=  0 \quad \mbox{ for } t > 0 \label{dirichlet-pointwise}\\
	u_x(\ell,t) &= 0   \quad \mbox{ for } t > 0\label{Neuman-pointwise}
\end{align} 
with the following initial data:
\begin{equation} \label{init-pointwise}
	u(x,0)=u_0(x), \quad x \in (0,\ell)
\end{equation}
and 
\begin{equation} \label{init1-pointwise}
	u_t(x,0)=u_1(x) \quad x \in (0,\ell)
\end{equation}
Here, $A = - \partial^{2}_{x}$ be the unbounded operator in $H = L^2(0,\ell)$
with domain
\[ H_{1}={\mathcal D}(A) = \left\{u \in H^{2}(0,\ell); \, u(0) = 0, \, u_{x}(\ell) = 0 \right\}, \quad H_{\frac{1}{2}}= {\mathcal D}(A^{\frac{1}{2}}) = \left\{u \in H^{1}(0,\ell); \, u(0) =0 \right\}
\]
and
\[
U = \mathbb{R}, 
B \in {\mathcal L}(\R, H_{-{\frac{1}{2}}}), \, Bk = k \, \delta_{\xi}, \forall \, k \in \R, \, B^{*}u = u(\xi), \, \forall \, u \in H_{\frac{1}{2}} \ .
\]
We have according to \cite{ANP1} (see \cite{aht,aht2} for the case without delay) the following stability result:
\begin{theorem}\cite{ANP1} \label{princ}
\begin{enumerate}
\item 
	We suppose that $\xi = \frac{\ell}{2}$. Then for any $\mu \in (0, 2)$
	there exist  positive constants $C_1, C_2$ such that for all initial data in ${\mathcal H}$, the  solution of problem \eqref{wave-pointwise-free}-\eqref{init1-pointwise}
	satisfies
	\begin{equation}\label{expestimate}
		E(t)\le C_{1} \, e^{- \, C_{2} t}.
	\end{equation}
	The constant $C_{1}$ depends  on the initial data, on $\ell$ and on $\mu$, while $C_{2}$ depends only  on $\ell$ and on $\mu$. 
 \item 
For $\xi = \frac{\ell}{2}$ and $\mu =2$, if we denote $\mathcal{S}_{p,2} (t)$ the propagator of the pointwise delayed control problem with $\mu = 2$, we have by definition: 
\[
\forall t > 0 \,,\, \mathcal{S}_{p,2}(t)
\begin{pmatrix}
u_{0} \\
u_{1}
\end{pmatrix} =
\begin{pmatrix}
u \\
u_{t}
\end{pmatrix}.
\]
And we have that
\[
\forall t > 0 \,,\, \forall n \in \N^* \,,\, \mathcal{S}_{p,2}(t+ 2n T)
\begin{pmatrix}
u_{0} \\
u_{1}
\end{pmatrix} = - \, 
\mathcal{S}_{p,2}(t)
\begin{pmatrix}
u_{0} \\
u_{1}
\end{pmatrix}
\]
wehereas 
\[
\forall t > 0 \,,\, \forall n \in \N^* \,,\, \mathcal{S}_{p,2}(t+ (2n+1) T)
\begin{pmatrix}
u_{0} \\
u_{1}
\end{pmatrix} = 
\mathcal{S}_{p,2}(t)
\begin{pmatrix}
u_{0} \\
u_{1}
\end{pmatrix}.
\]
So in particular, we have that $\mathcal{S}_{p,2}$ is $3T$-periodic.
 \end{enumerate}
\end{theorem}

We note here that the proof of the second assertion of the above Theorem \ref{princ} is a simple adaptation of the proof of \cite[Theorem 1.2]{ANP} for $\mu = 2$ and $\xi = \frac{\ell}{2}$. 

\begin{remark}
	In the sequel we will thus choose $\xi = \frac{\ell}{2}$.
\end{remark}
The proof of the preceding result conducted by Ammari, Nicaise and Pignotti in \cite{ANP1} is based on the following equivalent formulation of 
\eqref{wave-pointwise}-\eqref{dirichlet-pointwise}-\eqref{Neuman-pointwise} for $t \geq 2\ell$:
\begin{align}
	u_{tt}^{-}(x,t) - u_{xx}^{-}(x,t)&=0 \quad \mbox{\rm for } x \in (0,\xi) \,,\, t \geq 2 \ell \label{wave-moins}\\
	u_{tt}^{+}(x,t) - u_{xx}^{+}(x,t)&=0 \quad \mbox{\rm for } x \in (\xi,\ell) \,,\, t \geq 2 \ell \label{wave-plus}\\
	u^{-}(\xi^{-},t) &= u^{+}(\xi^{+},t)\quad \mbox{\rm for } t \geq 2 \ell \label{continuity-xi}\\
	u_{x}^{-}(\xi^{-},t) - u_{x}^{+}(\xi^{+},t) & = -\mu  u_{t}^{-}(\xi^{-},t-2 \ell) = - \mu  u_{t}^{+}(\xi^{+},t-2 \ell) \quad \mbox{\rm for } t \geq 2 \ell \label{grad-xi}\\
	u(0,t)  &=  0 \quad \mbox{ \rm for } t \geq 2 \ell  \label{dirichlet-plus}\\
	u_x(\ell,t) &= 0   \quad \mbox{ \rm for } t \geq 2 \ell \label{Neuman-plus}
\end{align} 
where we define for $\xi^{-}$ (for instance):
\[
u^{-}(\xi^{-},t) = \displaystyle \lim_{\underset{x < \xi}{x \to \xi}} u^{-}(x,t)\,,\, u_{x}^{-}(\xi^{-},t) = \displaystyle \lim_{\underset{x < \xi}{x \to \xi}} u_{x}^{-}(x,t).
\]
A singularity is thus occurring at the point $\xi$; so a finite difference discretization is not well adapted to furnish a good approximation. We thus decide to use rather a finite volume scheme well adapted to deal with
the case of a point source term \cite{ERG2000}

\subsubsection{\textbf{\textit{Construction of the numerical scheme: the finite volume discretization}}}
Let $N$ be a non negative even integer. Let $\Delta x~=~\dfrac{\ell}{N} \ $.
Following \cite{ERG2000}, the uniform admissible mesh $\mathcal{T}$ of the interval $(0,\ell)$ is given by the family $\big\{ K_{j} \,  ,\,j \in \{1, \cdots, N\} \big\}$ 
of control volumes such that $K_{j}  = ( x_{j - \frac{1}{2} } , x_{j+  \frac{1}{2}} )$ and the family $( x_{j})_{j = 0, \cdots, N+1 }$ assumed to be the center of $(K_{j})_{j= 1, \cdots, N}$ such that:
\[
0=x_{0}=x_{\frac{1}{2}} <x_{1}<x_{\frac{3}{2}}< x_{2}<...<x_{N-\frac{1}{2}}<x_{N}<x_{N+\frac{1}{2}}=x_{N+1}=\ell \\
\]
that is $x_{j+\frac{1}{2}} =j\Delta x \;,\; j = 0,\ldots,N \,,\, x_{j} =(j-\frac{1}{2})\Delta x \,,\, j= 1,\ldots,N \ .$
For $j=1,\ldots \,,\, N$, we denote
\begin{align*}
	h_{j} &= x_{j+\frac{1}{2}} - x_{j-\frac{1}{2}} = \Delta x \, \\
	h_{j}^{-} &= x_{j} - x_{j-\frac{1}{2}} = \frac{\Delta x }{2} \,  \\
	h_{j}^{+} &= x_{j+1} - x_{j+\frac{1}{2}} =\frac{\Delta x}{2} \\
	h_{j+\frac{1}{2}}&= x_{j+1} - x_{j-1} = \Delta x\\
	K_{j}&=(x_{j-\frac{1}{2}},x_{j+\frac{1}{2}})
\end{align*}

As $0=x_{0}=x_{\frac{1}{2}}$, we have $h_{\frac{1}{2}}= \dfrac{\Delta x}{2}$ and as $x_{N+\frac{1}{2}}=x_{N+1}=\ell$, we also have $h_{N + \frac{1}{2}} = \dfrac{\Delta x}{2}$ 

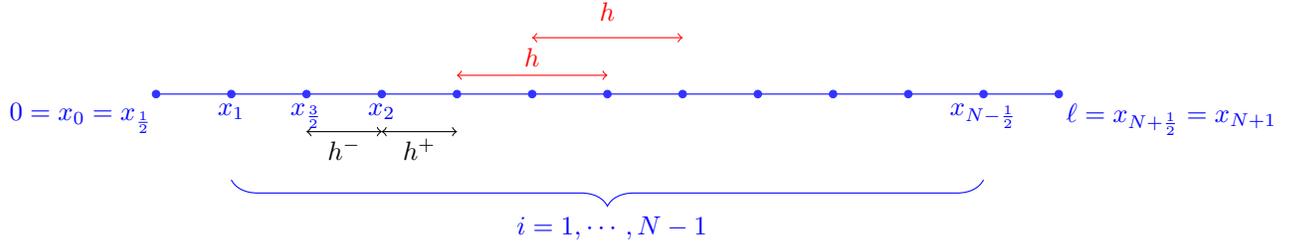
\begin{figure}[h!]
	\begin{center}
		\begin{tikzpicture}
			\draw[blue,-](-2,0)--(10,0);
			\node[blue,below] at (-3,0){\scalebox{1}{$0=x_{0}= x_{\frac{1}{2}}$}};
			\node at (-2,0) [circle, scale=0.3, draw=blue!80,fill=blue!80] {};
			
			\node[blue,below] at (11.5,0){\scalebox{1}{$\ell = x_{N+\frac{1}{2}} = x_{N+1}$}};
			\node at (10,0) [circle, scale=0.3, draw=blue!80,fill=blue!80] {};
			
			\draw[red,<->](2,0.25)--(4,0.25);
			\node[red,above] at (3,0.25){\scalebox{1}{$h$}};
			
			\draw[red,<->](3,0.75)--(5,0.75);
			\node[red,above] at (4,0.85){\scalebox{1}{$h$}};
			
			\draw[black,<->](1,-0.5)--(2,-0.5);
			\node[black,above] at (1.5,-1.0){\scalebox{1}{$h^{+}$}};
			
			\draw[black,<->](0,-0.5)--(1,-0.5);
			\node[black,above] at (0.5,-1.0){\scalebox{1}{$h^{-}$}};
			
			\draw [blue,decorate,decoration={brace,amplitude=10pt,mirror,raise=4pt},yshift=0pt]
			(-1,-1) -- (9,-1)  ;		
			\node[blue,below] at (4,-1.5){\scalebox{1}{ $i= 1, \cdots, N-1$}}; 
			
			
			\node[blue,below] at (-1,0){\scalebox{1}{$x_{1}$}};
			\node at (-1,0) [circle, scale=0.3, draw=blue!80,fill=blue!80] {};	
			\node[blue,below] at (0,0){\scalebox{1}{$x_{\frac{3}{2}}$}};
			\node at (0,0) [circle, scale=0.3, draw=blue!80,fill=blue!80] {};
			
			\node[blue,below] at (1,0){\scalebox{1}{$x_{2}$}};
			\node at (1,0) [circle, scale=0.3, draw=blue!80,fill=blue!80] {};
			\node at (2,0) [circle, scale=0.3, draw=blue!80,fill=blue!80] {};
			\node at (3,0) [circle, scale=0.3, draw=blue!80,fill=blue!80] {};
			\node at (4,0) [circle, scale=0.3, draw=blue!80,fill=blue!80] {};
			\node at (5,0) [circle, scale=0.3, draw=blue!80,fill=blue!80] {};
			\node at (6,0) [circle, scale=0.3, draw=blue!80,fill=blue!80] {};		
			\node at (7,0) [circle, scale=0.3, draw=blue!80,fill=blue!80] {};		
			\node at (8,0) [circle, scale=0.3, draw=blue!80,fill=blue!80] {};		
			\node at (9,0) [circle, scale=0.3, draw=blue!80,fill=blue!80] {};
			\node[blue,below] at (9,0){\scalebox{1}{$x_{N-\frac{1}{2}}$}};

		\end{tikzpicture}
	\end{center}
	\caption{A model representing the admissible one-dimensional mesh}\label{fig1}
\end{figure}
Set $t^{n+1}-t^{n}=\Delta t$ for all $n \in \mathbb{N}$. We will suppose that $2 \ell = T = K \times \Delta t \,,\, K \in \N^{*}$ to write easily the discretization of the delay term.
Using this uniform discretization, the point $x  = \dfrac{\ell}{2}$ is the point $x_{j_{0}+\frac{1}{2}} = \dfrac{\ell}{2}$ where $j_{0} = N/2$.
For $j=1,\ldots \,,\, N$, and $t > 0$, we denote
\[
u_{j}(t) = \frac{1}{h_{j}} \int_{x_{j-\frac{1}{2}}}^{x_{j+\frac{1}{2}}} u(x,t) dx 
\]
the mean value of $u$ on the cell $K_{j}$ and for $t = t_{n}$, we denote $u_{j}^{n} = u_{j}(t_{n})$.

As $\forall t \in [0,T] \,,\, u(0,t) = 0$, we will denote  $\forall n=1,\ldots N \,,\, u_{0}^{n} = 0$. If needed, we also denote $u_{N+1}^{n}~=~u(\ell,t^{n})$.

\medskip 

\textbf{First step: for time $\boldsymbol{t \in [0,T)}.$}

\medskip 

As originally pointed out by Eymard, Gallou\"et and Herbin \cite{ERG2000}, the principle of the finite volume method for conservation laws is to integrate the equation \eqref{wave-pointwise-free} on each cell $K_{j}$ and then approximate \textit{the fluxes at the interface}.
Thus for $j=1,\ldots,N$, and for $t=t^{n}$, we got:
\begin{align*}
	\int_{K_{j}} u_{tt}(x,t^{n}) dx - \left(u_{x}(x_{j+\frac{1}{2}},t^{n}) - u_{x}(x_{j-\frac{1}{2}},t^{n})\right) &= 0.
\end{align*}
At this stage, we denote by $F_{j+\frac{1}{2}}^{n}$ \textit{the numerical flux} which is an approximation of the flux at the interface  $-u_{x}(x_{j+\frac{1}{2}},t^{n})$.

As $u_{x}(\ell,t^{n}) =0$, we already have $F_{N+\frac{1}{2}}^{n} = 0$ and we can set $u_{N+1}^{n}~=u_{N}^{n}$.

Since $u_{j}^{n}$ may also be viewed as an approximation of $u(x_{j},t^{n})$ and $h_{j+\frac{1}{2}}~=~x_{j+1}-x_{j}$, as $u(0,t^{n}) = 0$, as done in \cite{ERG2000}, a reasonable choice for the computation of the numerical fluxes $F_{j+\frac{1}{2}}^{n}$ is given by:
\begin{equation*} 
	F_{j+{\frac{1}{2}}}^{n} =
	\left\{
	\begin{array}{ll}
		- \dfrac{u_{1}^{n}}{h_{\frac{1}{2}}} , & \mbox{ if } j= 0, \\[0.4cm]
		- \left(\dfrac{u_{j+1}^{n} - u_{j}^{n}}{h_{j+\frac{1}{2}}}\right) & j = 1, \ldots, N-1 \\[0.4cm]
		0& \mbox{ if } j= N \quad 
	\end{array}
	\right.
	= 
	\left\{
	\begin{array}{ll}
		- \dfrac{2 u_{1}^{n}}{\Delta x} , & \mbox{ if } j= 0, \\[0.4cm]
		- \left(\dfrac{u_{j+1}^{n} - u_{j}^{n}}{\Delta x}\right) & j = 1, \ldots, N-1, \\[0.4cm]
		0& \mbox{ if } j= N \quad .
	\end{array}
	\right.
\end{equation*}
For the sake of homogeneity, we will denote $\alpha_{j+\frac{1}{2}}$ as
\[
\alpha_{j+\frac{1}{2}} = 
\left\{\begin{array}{ll}
	\dfrac{1}{h_{\frac{1}{2}}}, & \mbox{ if } j= 0, \\[0.4cm]
	\dfrac{1}{h_{j+\frac{1}{2}}} & j = 1, \ldots, N-1 \\[0.4cm]
	0& \mbox{ if } j= N \quad 
\end{array}
\right.
=
\left\{\begin{array}{ll}
	\dfrac{2}{\Delta x}, & \mbox{ if } j= 0, \\[0.4cm]
	\dfrac{1}{\Delta x} & j = 1, \ldots, N-1 \\[0.4cm]
	0& \mbox{ if } j= N \quad .
\end{array}
\right.
\]
such that
\[
\mbox{forall } j = 1, \ldots, N \,,\,  F_{j+{\frac{1}{2}}}^{n} = -\alpha_{j+\frac{1}{2}} \left({u_{j+1}^{n} - u_{j}^{n}}\right).
\]

Moreover since $\displaystyle \int_{K_{j}} u_{tt}(x,t^{n}) dx  = \dfrac{d^{2}\displaystyle \int_{K_{j}} u(x,t^{n}) dx}{dt^{2}}$, the numerical approximation of the equation \eqref{wave-pointwise-free} is given by: for $j = 1, \ldots, N-1$,
\begin{align*}
	h_{j} u_{j}(t^{n})_{tt} + (F_{j+{\frac{1}{2}}}^{n} - F_{j-{\frac{1}{2}}}^{n}) = 0
\end{align*}
The central difference approximation of the second time derivative permits us to finally write: $\mbox{ for } n \geq 1 \,,\, \mbox{ for } j = 1, \ldots, N$,
\begin{align*}
	h_{j} \dfrac{u_{j}^{n+1} - 2 u_{j}^{n} + u_{j}^{n-1}}{{\Delta t}^{2}} + (F_{j+{\frac{1}{2}}}^{n} - F_{j-{\frac{1}{2}}}^{n}) = 0
\end{align*}
or in its homogeneous form: $\mbox{ for } n \geq 1 \,,\, \mbox{ for } j = 1, \ldots, N$,
\begin{align}\label{discrete-wave-pointwise}
	h_{j} \dfrac{u_{j}^{n+1} - 2 u_{j}^{n} + u_{j}^{n-1}}{{\Delta t}^{2}} 
	- \Big[ \alpha_{j+\frac{1}{2}} \left({u_{j+1}^{n} - u_{j}^{n}}\right) - 
	\alpha_{j-\frac{1}{2}} \left({u_{j}^{n} - u_{j-1}^{n}}\right) \Big]= 0 \ .
\end{align}
According to the initial conditions given by equations \eqref{init-pointwise}, we firstly choose: for $j =1,\ldots, N$,
\begin{equation*}
	u^{0}_{j} = u_{0}(x_{j}).
\end{equation*}
We can use the second initial conditions \eqref{init1-pointwise} to find the values of $u$ at time $t^{1} = \Delta t$, by employing a ``ghost'' time-boundary \textit{i.e.} $t^{-1}= - \Delta t$
and the second-order central difference formula:
\begin{equation*} 
	\mbox{for } j =1,\ldots, N \,,\, u_{1}(x_{j}) =\left.\dfrac{\partial u}{\partial t}\right|_{x_{j},0} =  \dfrac{ u_{j}^{1} - u_{j}^{-1} } {2 \Delta t} + O(\Delta t^{2}) .
\end{equation*}
Thus we have $\mbox{for } j =1,\ldots, N$:
\begin{equation*}
	u_{j}^{-1}=  u_{j}^{1}  - 2 \Delta t \ u_{1}(x_{j}) \ .
\end{equation*} 
Setting $n=0$, in the numerical scheme \eqref{discrete-wave-pointwise}, the previous equalities permits to compute $(u_{j}^{1})_{j=0,N}$. 
Finally, the solution $u$ can be computed at any time $t^n$.

In order to compute the solution $u$ beyond the time $T$, we have to compute the quantity $u_t(\xi,t)$ for time $t \in [0,T)$.
As $\xi = x_{j_{0}+\frac{1}{2}}$, since the mesh is uniform, we use the mean value of $u_{j_{0}}$ and $u_{j_{0}+1}$ as an approximation
of $u(x_{j_{0}+\frac{1}{2}},t)$. We will see in the next construction of the numerical fluxes that this formula is well adapted.
We use the centered difference scheme in time to finally compute:
\begin{align*}
	v^{0} &= u_{1}(x_{j_{0}}),\\
	\mbox{ for } n=1,\ldots, K-1 \,,\,  u_{t}(\xi,t^{n}) \approx v^{n} &= \dfrac{\dfrac{\left(u_{j_{0}+1}^{n+1} + u_{j_{0}-1}^{n+1} \right)}{2}- \dfrac{\left(u_{j_{0}+1}^{n-1} + u_{j_{0}-1}^{n-1} \right)}{2}}
	{2 \Delta t}.
\end{align*}

\medskip 

\textbf{Second step: for time $\boldsymbol{t > T}$.}

\medskip 

We will use the equivalent formulation \eqref{wave-moins}-\eqref{Neuman-plus} to construct the numerical scheme recalling that $x_{j_{0}+\frac{1}{2}} = \xi$.
As done before, we integrate the equation \eqref{wave-moins} and  \eqref{wave-plus} on each cell $K_{j}$ and then approximate \textit{the fluxes at the interface}.
Remarking that the cell $K_{j_{0}-1} = (x_{j_{0}-\frac{1}{2}},x_{j_{0}+\frac{1}{2}}) $  and $K_{j_{0}} = (x_{j_{0}+\frac{1}{2}},x_{j_{0}+\frac{3}{2}}) $, we denote
\begin{itemize}
	\item for $j = 1,\ldots,j_{0}$, 
	\[
	u^{{-,n}}_{j} =  \frac{1}{h_{j}} \int_{x_{j-\frac{1}{2}}}^{x_{j+\frac{1}{2}}} u^{-}(x,t) dx 
	\]
	\item for $j = j_{0}+1,\ldots N$, 
	\[
	u^{{+,n}}_{j} =  \frac{1}{h_{j}} \int_{x_{j-\frac{1}{2}}}^{x_{j+\frac{1}{2}}} u^{+}(x,t) dx. 
	\]
\end{itemize}
Since the flux $u_{x}(x_{j_{0}+\frac{1}{2}},t)$ is not define, we have to treat the two cells $K_{j_{0}} \,,\, K_{j_{0}+1} $ separately. 

As done before, integrating \eqref{wave-moins} on the cell $K_{j}$ gives:
\begin{align}\label{discrete-wave-pointwise-moins}
	\mbox{for } j = 1,\ldots,j_{0}-1 \quad h_{j} \dfrac{u_{j}^{-,n+1} - 2 u_{j}^{-,n} + u_{j}^{-,n-1}}{{\Delta t}^{2}} + (F_{j+{\frac{1}{2}}}^{-,n} - F_{j-{\frac{1}{2}}}^{-,n}) = 0
\end{align}
with 
\begin{equation*} 
	\mbox{for } j = 1,\ldots,j_{0}-1\,,\, F_{j+{\frac{1}{2}}}^{-,n} =
	\left\{
	\begin{array}{ll}
		- \dfrac{u_{1}^{-,n}}{h_{\frac{1}{2}}} & \mbox{ if } j= 0 \\[0.3cm]
		- \left(\dfrac{u_{j+1}^{-,n} - u_{j}^{-,n}}{h_{j+\frac{1}{2}}}\right) & 
	\end{array}
	\right.
	=
	\left\{
	\begin{array}{ll}
		- \dfrac{2 u_{1}^{-,n}}{\Delta x} & \mbox{ if } j= 0 \\[0.3cm]
		- \left(\dfrac{u_{j+1}^{-,n} - u_{j}^{-,n}}{\Delta x}\right). & 
	\end{array}
	\right.
\end{equation*}

Again, integrating \eqref{wave-plus} on the cell $K_{j}$ gives:

\begin{align}\label{discrete-wave-pointwise-plus}
	\mbox{for } j = j_{0}+2,N \quad h_{j}  \dfrac{u_{j}^{+,n+1} - 2 u_{j}^{+,n} + u_{j}^{+,n-1}}{{\Delta t}^{2}} + (F_{j+{\frac{1}{2}}}^{+,n} - F_{j-{\frac{1}{2}}}^{+,n}) = 0
\end{align}
with 
\begin{equation*} 
	\mbox{for } j = j_{0}+2\ldots,N \,,\, F_{j+{\frac{1}{2}}}^{+,n} =
	\left\{
	\begin{array}{ll}
		-\left(\dfrac{u_{j+1}^{+,n} - u_{j}^{+,n}}{h_{j+\frac{1}{2}}} \right) &  \\[0.3cm]
		0& \mbox{ if } j= N 
	\end{array}
	\right.
	=
	\left\{
	\begin{array}{ll}
		-\left(\dfrac{u_{j+1}^{+,n} - u_{j}^{+,n}}{\Delta x} \right) &  \\[0.3cm]
		0& \mbox{ if } j= N \quad .
	\end{array}
	\right.
\end{equation*}
Let us now treat the two cells $K_{j_{0}}$ and $K_{j_{0}+1}$.

\medskip

We integrate \eqref{wave-moins} on the cell $K_{j_{0}}$ to obtain:
\begin{align*}
	\int_{K_{j_{0}}} u_{tt}^{-}(x,t^{n}) dx - \left(u_{x}^{-}(x_{j_{0}+\frac{1}{2}},t^{n}) - u_{x}^{-}(x_{j_{0}-\frac{1}{2}},t^{n})\right) &= 0
\end{align*}
As $x_{j_{0}+\frac{1}{2}} = \xi$, we replace $u_{x}^{-}(x_{j_{0}+\frac{1}{2}},t^{n}) $ by $u_{x}^{-}(\xi^{-},t^{n})$.

Of course we already have approximated $- u_{x}^{-}(x_{j_{0}-\frac{1}{2}},t^{n})$ by the flux $F_{j_{0}-{\frac{1}{2}}}^{-,n}\ $.

\medskip

We integrate \eqref{wave-plus} on the cell $K_{j_{0}+1}$ to obtain:
\begin{align*}
	\int_{K_{j_{0}}+1} u_{tt}^{+}(x,t^{n}) dx - \left(u_{x}^{+}(x_{j_{0}+\frac{3}{2}},t^{n}) - u_{x}^{+}(x_{j_{0}+\frac{1}{2}},t^{n})\right) &= 0
\end{align*}
As $x_{j_{0}+\frac{1}{2}} = \xi$, we replace $u_{x}^{+}(x_{j_{0}+\frac{1}{2}},t^{n}) $ by $u_{x}^{+}(\xi^{+},t^{n})$.

Of course we already approximated $- u_{x}^{+}(x_{j_{0}+\frac{3}{2}},t^{n})$ by the flux $F_{j_{0}+{\frac{3}{2}}}^{+,n} \ $.

Let us now define the auxiliary variable $u_{j_{0}+\frac{1}{2}}^{n}$ as an approximation of $u^{+}(\xi^{+},t^{n})$. Because of the continuity equation \eqref{continuity-xi}, we firstly have $u_{j_{0}+\frac{1}{2}}^{n} = u^{+}(\xi^{+},t^{n}) = u^{-}(\xi^{-},t^{n})$.

So a good approximation of $u_{x}^{-}(\xi^{-},t^{n})$  and $u_{x}^{+}(\xi^{+},t^{n})$ is respectively:
\[
u_{x}^{-}(\xi^{-},t^{n}) \approx - F_{j_{0}+{\frac{1}{2}}}^{-,n} = \dfrac{u_{j_{0}+\frac{1}{2}}^{n} - u_{j_{0}}^{-,n}}{h_{j_{0}}^{-}} \quad \mbox{ and }  \quad u_{x}^{+}(\xi^{+},t^{n}) \approx - F_{j_{0}+{\frac{1}{2}}}^{+,n} =\dfrac{u_{j_{0}+1}^{+,n} - u_{j_{0}+\frac{1}{2}}^{n}}{h_{j_{0}+1}^{+}} \ .
\]
As the numerical scheme must verify the transmission condition \eqref{grad-xi}, to compute $u_{j_{0}+\frac{1}{2}}^{n} $, we use the discrete analog of \eqref{grad-xi}:
\[
- F_{j_{0}+{\frac{1}{2}}}^{-,n} + F_{j_{0}  + {\frac{1}{2}}}^{+,n} = - \mu  u_{t}(\xi,t^{n}-2 \ell) \ .
\]
Replacing the two fluxes by their expression leads to: 
\[
\dfrac{u_{j_{0}+\frac{1}{2}}^{n} - u_{j_{0}}^{-,n}}{h_{j_{0}}^{-}}- \dfrac{u_{j_{0}+1}^{+,n} - u_{j_{0}+\frac{1}{2}}^{n}}{h_{j_{0}+1}^{+}} = - \mu  u_{t}(\xi,t^{n}-2 \ell). 
\]
As we have considered a regular mesh, that is: $h_{j_{0}}^{-}= h_{j_{0}+1}^{+} = \Delta x /2$, we obtain finally:
\begin{align}
	u_{j_{0}+\frac{1}{2}}^{n} &= - \frac{\Delta x}{4}  \mu  u_{t}(\xi,t^{n}-2 \ell) + \dfrac{u_{j_{0}+1}^{+,n} + u_{j_{0}}^{-,n}}{2}. \label{auxiliary}
\end{align}
So the two fluxes are computed by: 
\begin{align}
	F_{j_{0}+{\frac{1}{2}}}^{-,n} &= \dfrac{ \mu u_{t}(\xi,t^{n}-2 \ell)}{2}  - \dfrac{u_{j_{0}+1}^{+,n} - u_{j_{0}}^{-,n}}{\Delta x} \label{fluxmoins}\\
	F_{j_{0}+{\frac{1}{2}}}^{+,n} &= -\dfrac{ \mu u_{t}(\xi,t^{n}-2 \ell)}{2}  - \dfrac{u_{j_{0}+1}^{+,n} - u_{j_{0}}^{-,n}}{\Delta x}. \label{fluxplus}
\end{align}

On the cell $K_{j_{0}}$, we get the fully discrete numerical scheme:
\begin{align*}
	h_{j_{0}}\dfrac{u_{j_{0}}^{-,n+1} - 2 u_{j_{0}}^{-,n} + u_{j_{0}}^{-,n-1}}{{\Delta t}^{2}} + (F_{j_{0}+{\frac{1}{2}}}^{-,n} - F_{j_{0}-{\frac{1}{2}}}^{-,n}) = 0 \quad ,
\end{align*}
which writes:
\begin{align}
	\Delta x \; \dfrac{u_{j_{0}}^{-,n+1} - 2 u_{j_{0}}^{-,n} + u_{j_{0}}^{-,n-1}}{{\Delta t}^{2}} - \left(\dfrac{u_{j_{0}+1}^{+,n} - u_{j_{0}}^{-,n}}{\Delta x} - \dfrac{u_{j_{0}}^{-,n} - u_{j_{0}-1}^{-,n}}{\Delta x}\right) = - \dfrac{ \mu u_{t}(\xi,t^{n}-2 \ell)}{2} \quad . \label{discrete-wave-pointwise-j0-moins}
\end{align}

On the cell $K_{j_{0}+1}$, we get the fully discrete numerical scheme:
\begin{align*}
	h_{j_{0}+1}\dfrac{u_{j_{0}+1}^{+,n+1} - 2 u_{j_{0}+1}^{+,n} + u_{j_{0}+1}^{+,n-1}}{{\Delta t}^{2}} + (F_{j_{0}+{\frac{3}{2}}}^{+,n} - F_{j_{0}+{\frac{1}{2}}}^{+,n}) = 0 \quad
\end{align*}
which writes:
\begin{align}
	\Delta x \; \dfrac{u_{j_{0}+1}^{+,n+1} - 2 u_{j_{0}+1}^{+,n} + u_{j_{0}+1}^{+,n-1}}{{\Delta t}^{2}} - \left(\dfrac{u_{j_{0}+2}^{+,n} - u_{j_{0}+1}^{+,n}}{\Delta x} - \dfrac{u_{j_{0}+1}^{+,n} - u_{j_{0}}^{-,n}}{\Delta x}\right) = - \dfrac{ \mu u_{t}(\xi,t^{n}-2 \ell)}{2} . \label{discrete-wave-pointwise-j0-plus}
\end{align}
So one can remark that the delay pointwise term is equally splitted on the cell $K_{j_{0}}$ and $K_{j_{0}+1}$ due to the fact that the mesh is uniform.

At this stage, it remains to compute the value of $u_{t}(\xi,t^{n})$.  We first approximate the partial time derivative by the centered difference
\[
u_{t}(\xi,t^{n}) \approx \dfrac{u_{j_{0}+\frac{1}{2}}^{n+1} - u_{j_{0}+\frac{1}{2}}^{n-1}}{2 \Delta t}
\]
Then we replace the value of $u_{j_{0}+\frac{1}{2}}^{n+1}$ and $u_{j_{0}+\frac{1}{2}}^{n-1}$ using the equation \eqref{auxiliary} and as $T = K \Delta T$ we use the saved value $v^{n+1-K} \approx u_{t}(\xi,t^{n+1}-2 \ell)$ and
$v^{n-1-K} \approx u_{t}(\xi,t^{n-1}-2 \ell)$ to obtain:
\[
u_{t}(\xi,t^{n}) \approx v^{n} = \dfrac{- \frac{\Delta x}{4}  \mu  \left(v^{n+1-K} - v^{n-1-K}) \right)+ 
	\dfrac{\left(u_{j_{0}+1}^{+,n+1} + u_{j_{0}-1}^{-,n+1} \right)}{2}- \dfrac{\left(u_{j_{0}+1}^{+,n-1} + u_{j_{0}-1}^{-,n-1} \right)}{2}}
{2 \Delta t}.
\]

We set $s = \left(\dfrac{\Delta t}{\Delta x}\right)^{2}$. Let us now summarize, the computation of the solution and the pointwise delay term in the case of a uniform mesh.
Because the definition of the numerical fluxes is not the same for the first cell and the last cell, we have to treat separately these two cells. Let us also remark that contrary to the two preceding cases,
there is no need to compute boundary values since the boundary conditions are used to construct the numerical fluxes.

\textbf{First step: for time $\boldsymbol{t \in [0,T)}$}\\
\noindent$\mbox{Initialization } \mbox{ for } j = 1,\ldots, N \,,\, u^{0}_{j} = u_{0}(x_{j}) \quad v^{0} = u_{1}(\ell/2) $\\
$ \mbox{Solution for } t = \Delta t$ \\
$\hspace*{2cm} \mbox{ for } j = 1\,,\quad  u^{1}_{1} = \dfrac{s}{2}\Big(u^{0}_{2} - 3 u^{0}_{1}\Big)+ u^{0}_{1} + \Delta t \, u_{1}(x_{1})$\\
$\hspace*{2cm} \mbox{ for } j = 2,\ldots,N-1\,,\quad  u^{1}_{j} = \dfrac{s}{2}\Big(u^{0}_{j+1} - 2 u^{0}_{j} - u^{0}_{j-1}\Big)+ u^{0}_{j} + \Delta t \, u_{1}(x_{j})$\\
$\hspace*{2cm} \mbox{ for } j = N\,,\quad  u^{1}_{N} = \dfrac{s}{2}\Big(u^{0}_{N-1} -  u^{N}_{1}\Big)+ u^{0}_{N} + \Delta t \, u_{1}(x_{N})$\\
$ \mbox{Solution for } t \in (\Delta t,T) \mbox{ i.e. for } n=1,\ldots, K-1$\\
$\hspace*{2cm} \mbox{ for } j = 1\,,\quad  u^{n+1}_{1} = s\Big(u^{n}_{2} - 3 u^{n}_{1}\Big)+ 2 u^{n}_{1} -u_{1}^{n-1}$\\
$\hspace*{2cm} \mbox{ for } j = 2,\ldots,N-1\,,\quad  u^{n+1}_{j} =s \Big(u^{n}_{j+1} - 2 u^{n}_{j} - u^{n}_{j-1}\Big)+ 2 u^{n}_{j} -u_{j}^{n-1}$\\
$\hspace*{2cm} \mbox{ for } j = N\,,\quad  u^{n+1}_{N} = s\Big(u^{n}_{N-1} -  u^{n}_{N}\Big)+ 2 u^{n}_{N} - u^{n-1}_{N} $\\
$\hspace*{2cm} \mbox{Delayed pointwise term} \quad v^{n} = \dfrac{\dfrac{\left(u_{j_{0}+1}^{n+1} + u_{j_{0}-1}^{n+1} \right)}{2}- \dfrac{\left(u_{j_{0}+1}^{n-1} + u_{j_{0}-1}^{n-1} \right)}{2}}
{2 \Delta t}.$

\textbf{Second step: for time $\boldsymbol{t \in [T,T_{f}]}$}

\noindent$\mbox{Solution for } t \in [T,T_{f}] \ i.e. \ \mbox{for } n=K,\ldots, M-1$\\
$\hspace*{2cm} \mbox{ for } j = 1\,,\quad  u^{-,n+1}_{1} = s\Big(u^{-,n}_{2} - 3 u^{-,n}_{1}\Big)+ 2 u^{-,n}_{1} -u_{1}^{-,n-1}$\\
$\hspace*{2cm} \mbox{ for } j = 2,\ldots,j_{0}-1\,,\quad  u^{-,n+1}_{j} =s \Big(u^{-,n}_{j+1} - 2 u^{-,n}_{j} - u^{-,n}_{j-1}\Big)+ 2 u^{-,n}_{j} -u_{j}^{-,n-1}$\\
$\hspace*{2cm} \mbox{ for } j = j_{0}\,,\quad  u^{-,n+1}_{j_{0}} =s \Big(u^{+,n}_{j_{0}+1} - 2 u^{-,n}_{j_{0}} - u^{-,n}_{j_{0}-1}\Big)+ 2 u^{-,n}_{j_{0}} -u_{j_{0}}^{-,n-1} - \dfrac{s \Delta x \mu v^{n-K}}{2}$\\
$\hspace*{2cm} \mbox{ for } j = j_{0}+1\,,\quad  u^{+,n+1}_{j_{0}+1} =s \Big(u^{+,n}_{j_{0}+2} - 2 u^{+,n}_{j_{0}+1} - u^{-,n}_{j_{0}}\Big)+ 2 u^{+,n}_{j_{0}+1} -u_{j_{0}+1}^{+,n-1} - \dfrac{s \Delta x \mu v^{n-K}}{2}$\\
$\hspace*{2cm} \mbox{ for } j = j_{0}+2,\ldots,N-1\,,\quad  u^{+,n+1}_{j} =s \Big(u^{+,n}_{j+1} - 2 u^{+,n}_{j} - u^{+,n}_{j-1}\Big)+ 2 u^{+,n}_{j} -u_{j}^{+,n-1}$\\
$\hspace*{2cm} \mbox{ for } j = N\,,\quad  u^{+,n+1}_{N} = s\Big(u^{+,n}_{N-1} -  u^{+,n}_{N}\Big)+ 2 u^{+,n}_{N} - u^{+,n-1}_{N} $\\
$\hspace*{2cm} \mbox{Delayed pointwise term}$\\
$\hspace*{2cm}v^{n} = \dfrac{- \frac{\Delta x}{4}  \mu  \left(v^{n+1-K} - v^{n-1-K}) \right)+ 
	\dfrac{\left(u_{j_{0}+1}^{+,n+1} + u_{j_{0}-1}^{-,n+1} \right)}{2}- \dfrac{\left(u_{j_{0}+1}^{+,n-1} + u_{j_{0}-1}^{-,n-1} \right)}{2}}
{2 \Delta t}.$

\subsubsection{\textbf{\textit{Discrete energy and CFL condition}}}
The aim of this section is to design a discrete energy that is preserved in the first step that is the free wave equation with Dirichlet and Neumann boundary condition in the context of a discretization by a finite volume method. To this end, let us define:
\begin{itemize}
	\item the discrete kinetic energy  as: 
	\begin{equation*}
		\displaystyle  E_{k}^{n} = \dfrac{1}{2} \sum_{i=1 }^{N}h_{i}\left(\dfrac{u_i^{n+1}-u_i^n}{\Delta t}\right)^2
	\end{equation*}
	\item the discrete potential energy as: 
	\begin{equation*}
		\displaystyle  E_{p}^{n} = \dfrac{1}{2} \sum_{i=0}^{N}\alpha_{i+{\frac{1}{2}}} (u_{i+1}^{n+1} - u_{i}^{n+1})(u_{i+1}^{n}- u_i^{n}).
	\end{equation*}
\end{itemize}
\begin{proposition} \label{energy-internal}
	The discrete energy is conserved for all $t = 0, \ldots, T-\Delta t$ \textit{ i.e.}
	\[
	\forall \, n = 1, \ldots, K-1 \,,\, \mathcal{E}^{n+1} = \mathcal{E}^{n} \, .
	\] 
\end{proposition}
\begin{proof}
	The proof is similar to the continuous case: we multiply the discrete problem by the approximation of $u_t$.
	We multiply the left hand side and right hand side of \eqref{discrete-wave-pointwise} by $(u_i^{n+1}- u_i^{n-1})$ and we sum over 
	$i = 1, \ldots, N$ to obtain:
	\begin{align}
		&\sum_{i=1}^{N} h_i \bigg(\frac{u_i^{n+1}-2 u_i^n + u_i^{n-1}}{\Delta t^2} \bigg) (u_i^{n+1}-u_i^{n-1}) \nonumber \\
		&- \sum_{i=1}^{N}  \Big[\alpha_{i+{\frac{1}{2}}} (u_{i+1}^n - u_i^n) - \alpha_{i-{\frac{1}{2}}} (u_{i}^n - u_{i-1}^n)\Big](u_i^{n+1}-u_i^{n-1}) \label{Discrete-EQ-1}\\
		&= 0. \nonumber
	\end{align}\textbf{Estimation of the first term of \eqref{Discrete-EQ-1}}: 
	\begin{equation} \label{result-energy-kinetic}
		\begin{array}{lll}
			&\displaystyle\sum_{i=1}^{N} h_i \bigg(\dfrac{u_i^{n+1}-2 u_i^n + u_i^{n-1}}{\Delta t^2} \bigg) (u_i^{n+1}-u_i^{n-1})\\
			\displaystyle&=\displaystyle\sum_{i=1}^{N} h_i \bigg[\dfrac{(u_i^{n+1}- u_i^n) -( u_i^n - u_i^{n-1})}{\Delta t^2} \bigg] [(u_i^{n+1}-u_i^n) + (u_i^n -u_i^{n-1})] \\
			\displaystyle&=  \displaystyle\sum_{i=1}^{N} h_i \bigg(\dfrac{u_{i}^{n+1}-u_{i}^n}{\Delta t}\bigg)^2 -  \sum_{i=1}^{N_{max}} h_i \bigg(\dfrac{u_{i}^{n}-u_{i}^{n-1}}{\Delta t}\bigg)^{2} \\
			\displaystyle&= 2(E_k^{n}-E_k^{n-1}).
		\end{array}
	\end{equation}
	\textbf{Estimation of the second term of \eqref{Discrete-EQ-1}}:
	\begin{equation*}
		\begin{array}{lll}
			- \displaystyle \sum_{i=1}^{N} \Big[\alpha_{i+{\frac{1}{2}}} (u_{i+1}^n - u_i^n) - \alpha_{i-{\frac{1}{2}}} (u_{i}^n - u_{i-1}^n)\Big](u_i^{n+1}-u_i^{n-1}) \\
			= - \displaystyle \sum_{i=1}^{N} \alpha_{i+{\frac{1}{2}}} (u_{i+1}^n - u_i^n)(u_i^{n+1}-u_i^{n-1}) + \displaystyle \sum_{i=1}^{N} \alpha_{i-{\frac{1}{2}}} (u_{i}^n - u_{i-1}^n)(u_i^{n+1}-u_i^{n-1}).
		\end{array}
	\end{equation*}
	By translation of index $i$ of the second term of the right hand side of the above equation, we obtain:
	\begin{equation*}
		\begin{array}{lll}
			- \displaystyle \sum_{i=1}^{N} \Big[\alpha_{i+{\frac{1}{2}}} (u_{i+1}^n - u_i^n) - \alpha_{i-{\frac{1}{2}}} (u_{i}^n - u_{i-1}^n)\Big](u_i^{n+1}-u_i^{n-1}) \\
			= - \displaystyle \sum_{i=1}^{N} \alpha_{i+{\frac{1}{2}}} (u_{i+1}^n - u_i^n)(u_i^{n+1}-u_i^{n-1}) + \displaystyle \sum_{i=0}^{N-1} \alpha_{i+{\frac{1}{2}}} (u_{i+1}^n - u_{i}^n)(u_{i+1}^{n+1}-u_{i+1}^{n-1}).
		\end{array}
	\end{equation*}
	Taking into consideration that $u_0^{n} = 0$ and $\alpha_{N+\frac{1}{2}}=0$, we obtain:
	\begin{equation*}
		\begin{array}{lll}
			- \displaystyle \sum_{i=1}^{N} \Big[\alpha_{i+{\frac{1}{2}}} (u_{i+1}^n - u_i^n) - \alpha_{i-{\frac{1}{2}}} (u_{i}^n - u_{i-1}^n)\Big](u_i^{n+1}-u_i^{n-1}) \\
			=   \displaystyle \sum_{i=0}^{N} \alpha_{i+{\frac{1}{2}}} (u_{i+1}^n - u_{i}^n)[(u_{i+1}^{n+1}- u_i^{n+1})-(u_{i+1}^{n-1} -u_i^{n-1})]\\
			=  \displaystyle \sum_{i=0}^{N} \alpha_{i+{\frac{1}{2}}} (u_{i+1}^n - u_{i}^n)(u_{i+1}^{n+1}- u_i^{n+1}) - \displaystyle \sum_{i=0}^{N} \alpha_{i+{\frac{1}{2}}} (u_{i+1}^n - u_{i}^n)(u_{i+1}^{n-1} -u_i^{n-1})\\
			= 2(E_p^{n}-E_p^{n-1}).
		\end{array}
	\end{equation*}
	We finally obtain:
	\[
	\forall \, n = 1, \ldots, K-1 \,,\, \mathcal{E}^{n+1} = \mathcal{E}^{n} \, .
	\]
\end{proof}
As in the two previous cases, in order to obtain a stable numerical scheme, the following Courant-Friedrichs-Lewy, CFL, condition holds:
\[ \Delta t\leq \Delta x \ .
\]
\begin{remark}\label{remark-implcit}[Implit scheme]
We have made the choice to present the three \textbf{explicit} numerical schemes for the sake of clarity and conciseness 
although these schemes generate a CFL condition to ensure the stability.

But let us remark that to avoid the CFL condition, which represents a restriction on the time step, for the three numerical schemes presented above, and to obtain an unconditionally stable numerical scheme,
we may construct the equivalent implicit scheme by replacing 
$u^{n}_{j}$ for $j=i-1\,,\, i\,,\, i+1$ in the approximation of the space derivative by the mean value: 
\[
\dfrac{u^{n+1}_j + u^{n-1}_j}{2} \ .
\]
The definition of the kinetic energy remains unchanged whereas for the 
potential energy one has to replace the term: 
\[
\left(\dfrac{u_{j+1}^{n}-u_j^n}{\Delta x}\right) 
\left(\dfrac{u_{j+1}^{n+1}-u_j^{n+1}}{\Delta x}\right)
\]
by 
\[
\dfrac{1}{2}\left(\left(\dfrac{u_{j+1}^{n+1}-u_j^{n+1}}{\Delta x}\right)^{2} +\left(\dfrac{u_{j+1}^{n-1}-u_j^{n-1}}{\Delta x}\right)^{2}\right)
\]
for the finite difference approximation and one has to replace the term:
\[
\alpha_{i+{\frac{1}{2}}} 
(u_{i+1}^{n+1} - u_{i}^{n+1})
(u_{i+1}^{n-1}- u_i^{n-1})
\]
by 
\[
\dfrac{1}{2} \ \alpha_{i+{\frac{1}{2}}} \Big(
(u_{i+1}^{n+1} - u_{i}^{n+1})^{2} + 
(u_{i+1}^{n-1} - u_{i}^{n-1})^{2} \Big)
\]
for the finite volume method.
This will guarantee that the total energy remains positive and decreasing in both cases. 

The boundedness and the positivity of the discrete energy  is used to prove the convergence of the numerical approximation when $\Delta x$ and $\Delta t$ tends to 0 towards the strong solution for the finite difference scheme \cite{Lax-Richtmyer-1956} and the convergence towards the weak solution for the finite volume method \cite{ERG2000} and \cite{Tcheques-2018} for the 2D case.
\end{remark}

\section{Numerical experiment and validation of the theoretical results in 1D}\label{Numerical-results}
Without loss of generality (up to a spatial rescaling), for every numerical experiment, we have chosen $\ell = 1$ so that $T = 2$. We construct these experiments in order to validate the theoretical results.

We have taken $N=100$ points and $T_{f} = 200 \times T = 400$. The $CFL$ number is chosen as $CFL = 1$.
 
\subsection{The boundary case}
The initial condition must satisfy Dirichlet boundary condition at the point $x = 0$ and Neumann boundary condition at the point $x = 1$. So we have chosen: 
\[
\forall x \in \, [0,1] \,,\, u_{0}(x) = x^{2} - 2 x \,,\, u_{1}(x) = -(x^{2} - 2 x).
\]

We present firstly the numerical results for $\mu \in (0,1)$ for different values of $\mu$ on the same graphics to study the influence of this parameter. 
Figure \ref{Energy-boundary} confirms the decreasing of the energy for any $0 < \mu < 1$ although
one cannot conclude on the influence of the damping parameter $\mu$ from figures \ref{Energy-boundary-1} and \ref{Energy-boundary-2}. 

Therefore, we have decide to plot the quantity $-\ln(E(t))$ versus $t$ in figure \ref{Ln-Energy-boundary}.
Figures \ref{Ln-Energy-boundary-1} and \ref{Ln-Energy-boundary-2} show that we obtain almost an increasing straight line whose slope is the exponential decay speed $\omega$ that is for large time $t$,
one has $E(t) < C e^{-\omega t}$. 
By plotting the quantity $-\frac{\ln(E(t))}{t}$ versus $t$,  in figures \ref{Exp-Energy-boundary}, we obtain almost an horizontal line so that you may conclude that the quotient $-\frac{\ln(E(t))}{t}$ stays bounded by above and by below by a constant. This constant represents the exponential decay speed. As announced by Ammari, Chentouf and Smaoui,
in \cite{ACS1}, if we set $\mu_{0} = 3 - 2 \sqrt{2}~\simeq~0.17$, we see 
in figures \ref{Exp-Energy-boundary-1} and \ref{Exp-Energy-boundary-2}, that this speed is increasing for $\mu \in (0,\mu_{0})$ then it decreases for $\mu \in (\mu_{0},0.95)$ but it stays strictly positive. So $\mu_{0}$ is the optimal parameter to choose to stabilize the system by the boundary delay damping term.

We have then choose $\mu = 1$. This case shows a surprising but predicted behavior, see Theorem \ref{princf}. Indeed figure \ref{Energy-boundary-mu-1} shows the conservation of the discrete energy for $t \in [0,20]$. 
Thus we wanted to know how was the final profile at time $T=20$. Figure \ref{Profile-init} shows the initial profile and figure \ref{Profile-final-1} shows the profile at time
$T_{f} = 10 \times T = 20$. These two profiles are opposite  whereas the final profile at time $T_{f} = 11 \times T = 22$ are equal as seen in figure \ref{Profile-final-2}.
We have then performed several tests for different final time and for different initial condition $u_{0}\,,\, u_{1}$. The same surprising (but predicted) behavior shows up.

We continue our study by  presenting the numerical results for $\mu \in (1,2]$ . Figure \ref{Energy-boundary-sup} shows the growth of the energy for any $1 < \mu < 2$. We may conclude that the energy is increasing and the system cannot
be stabilized.


We end up this study by  presenting the numerical results for $\mu < 0 $ . Figure \ref{Energy-boundary-neg} shows the growth of the energy for any $\mu < 0$. We may conclude that the energy is increasing and the system cannot
be stabilized.


\subsection{The internal case}
The initial condition must satisfy Dirichlet boundary condition at the point $x = 0$ and at the point $x = 1$. So we have chosen: 
\[
\forall x \in \, [0,1] \,,\, u_{0}(x) = x (x-1) \,,\, u_{1}(x) = -x (x-1).
\]
The internal acting delay term acts on $[x_{i_{0}},x_{i_{1}}]$ where we have chosen $x_{i_{0}} = \frac{1}{4} \,,\, x_{i_{1}} = \frac{3}{4}$.

We present firstly the numerical results for $\mu \in (0,2)$ for different values of $\mu$ on the same graphics to study the influence of this parameter. 
Figure \ref{Energy-internal} shows that there exists $\mu_{0} \in (1.7,1.8)$ such that the energy is decreasing  for any $0 < \mu < \mu_{0}$ while
it seems that the energy is increasing for $\mu > \mu_{0}$. To be convinced by this numerical argument, we plot in figure \ref{Ln-Energy-internal} 
the quantity $-\ln(E(t))$ versus $t$. Again for any $0 < \mu < \mu_{0}$, we obtain almost an increasing straight line whose slope is the exponential decay speed $\omega$ that is for large time $t$, one has $E(t) < C e^{-\omega t}$, while for  any $\mu_{0} < \mu$, we obtain almost a decreasing straight line whose slope is the exponential growth speed $\omega$ that is for large time $t$, one has $E(t) > C e^{\omega t}$. By plotting the quantity $-\frac{\ln(E(t))}{t}$ versus $t$,  in figures \ref{Exp-Energy-internal}, 
we obtain almost an horizontal positive line for $0 < \mu <\mu_{0}$ so that you may conclude that is for large time $t$, one has $E(t) < C e^{-\omega t}$ while
we obtain an horizontal negative line for $\mu_{0}< \mu $ so that you may conclude that is for large time $t$, one has $E(t) >C e^{\omega t}$.

We present secondly the numerical results for $\mu <0$ for different values of $\mu$ on the same graphics.
Figure \ref{Energy-internal-neg} shows the growth of the energy for any $\mu < 0$. We may conclude that the energy is increasing and the system cannot
be stabilized.
%
%

\subsection{The pointwise case}
The initial condition must satisfy Dirichlet boundary condition at the point $x = 0$ and Neumann boundary condition at the point $x = 1$. So we have chosen: 
\[
\forall x \in \, [0,1] \,,\, u_{0}(x) = x^{2} - 2 x \,,\, u_{1}(x) = -(x^{2} - 2 x)
\]

We present firstly the numerical results for $\mu \in (0,2)$ for different values of $\mu$ on the same graphics to study the influence of this parameter. 
Figure \ref{Energy-pointwise} confirms the decreasing of the energy for any $0 < \mu < 2$ although
one cannot conclude on the influence of the damping parameter $\mu$ from figures \ref{Energy-pointwise-1} and \ref{Energy-pointwise-2}. 

Therefore, we have decided to plot the quantity $-\ln(E(t))$ versus $t$ in figure \ref{Ln-Energy-pointwise}.
Figures \ref{Ln-Energy-pointwise-1} and \ref{Ln-Energy-pointwise-2} show that we obtain almost an increasing straight line whose slope is the exponential decay speed $\omega$ that is for large time $t$,
one has $E(t) < C e^{-\omega t}$. By plotting the quantity $-\frac{\ln(E(t))}{t}$ versus $t$,  in figures \ref{Exp-Energy-pointwise}, we obtain almost an horizontal line so that you may conclude that the quotient $-\frac{\ln(E(t))}{t}$ stays bounded by above and by below by a constant. This constant represents the exponential decay speed. 
In figures \ref{Exp-Energy-pointwise-1} and \ref{Exp-Energy-pointwise-2}, we cannot really conclude about the variations of the exponential decay rate versus the value of the parameter $\mu$.

We have then choose $\mu = 2$. Again this case shows a surprising but predicted behavior, as in Theorem \ref{princ}. Indeed figure \ref{Energy-pointwise-mu-2} shows the conservation of the discrete energy for $t \in [0,20]$. Thus we wanted to know how was the final profile at time $T=20$. Figure \ref{Profile-init-pointwise-mu-2} shows the initial profile and figure \ref{Profile-final-pointwise-1} shows the profile at time
$T_{f} = 10 \times T = 20$. These two profiles are equal  whereas the final profile at time $T_{f} = 11 \times T = 22$ are opposite as seen in figure \ref{Profile-final-pointwise-2}. Again, we have then performed several tests for different final time and for different initial condition $u_{0}\,,\, u_{1}$. The same surprising (but predicted) behavior shows up.

We continue our study by  presenting the numerical results for $\mu > 2$ . Figure \ref{Energy-pointwise-sup} shows the growth of the energy for any $2 < \mu $.  We may conclude that the energy is increasing and the system cannot
be stabilized.

We end up this study by  presenting the numerical results for $\mu < 0 $ . Figure \ref{Energy-pointwise-neg} shows the growth of the energy for any $\mu < 0$. 
We may conclude that the energy is increasing and the system cannot
be stabilized.

\section{Conclusion and perspectives.}
Although delay effects arise in many applications and practical problems, we have {\color{black}seen in} this work that these effects could be overcame by choosing  a control law that uses information from the past (by switching or not). 

Moreover by constructing well adapted numerical experiments, one can choose the values of some parameters to optimize the decay rate of the energy.

We think that this type of control laws may also be used for the stabilization of partially damped coupled systems. This will be the future work we plan to investigate. 

\section*{Acknowledgments}
The authors would like to thank a lot the anonymous reviewers for their valuable remarks and comments which helped in improving  the article.

\bibliographystyle{plain}

\newpage
\begin{figure}[H]
	\centering
	\subcaptionbox{Evolution of the discrete energy for $0 < \mu \leq 0.5 \,,\, t \in [0,20]$.\label{Energy-boundary-1}}
	{\includegraphics[width=0.9\textwidth]{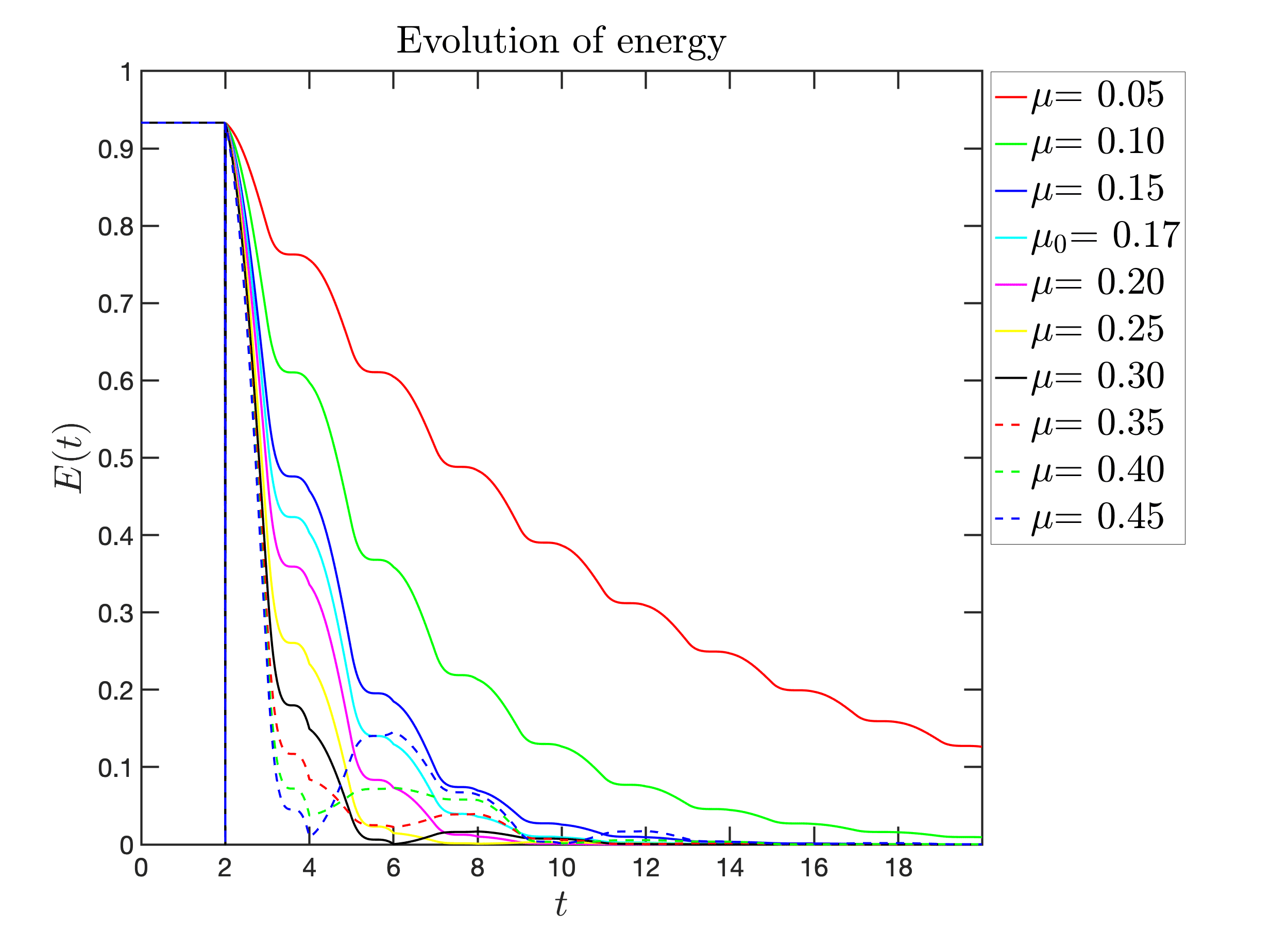}}
	\subcaptionbox{Evolution of the discrete energy for $0.5 < \mu < 1 \,,\, t \in [0,20]$.\label{Energy-boundary-2}}
	{\includegraphics[width=0.9\textwidth]{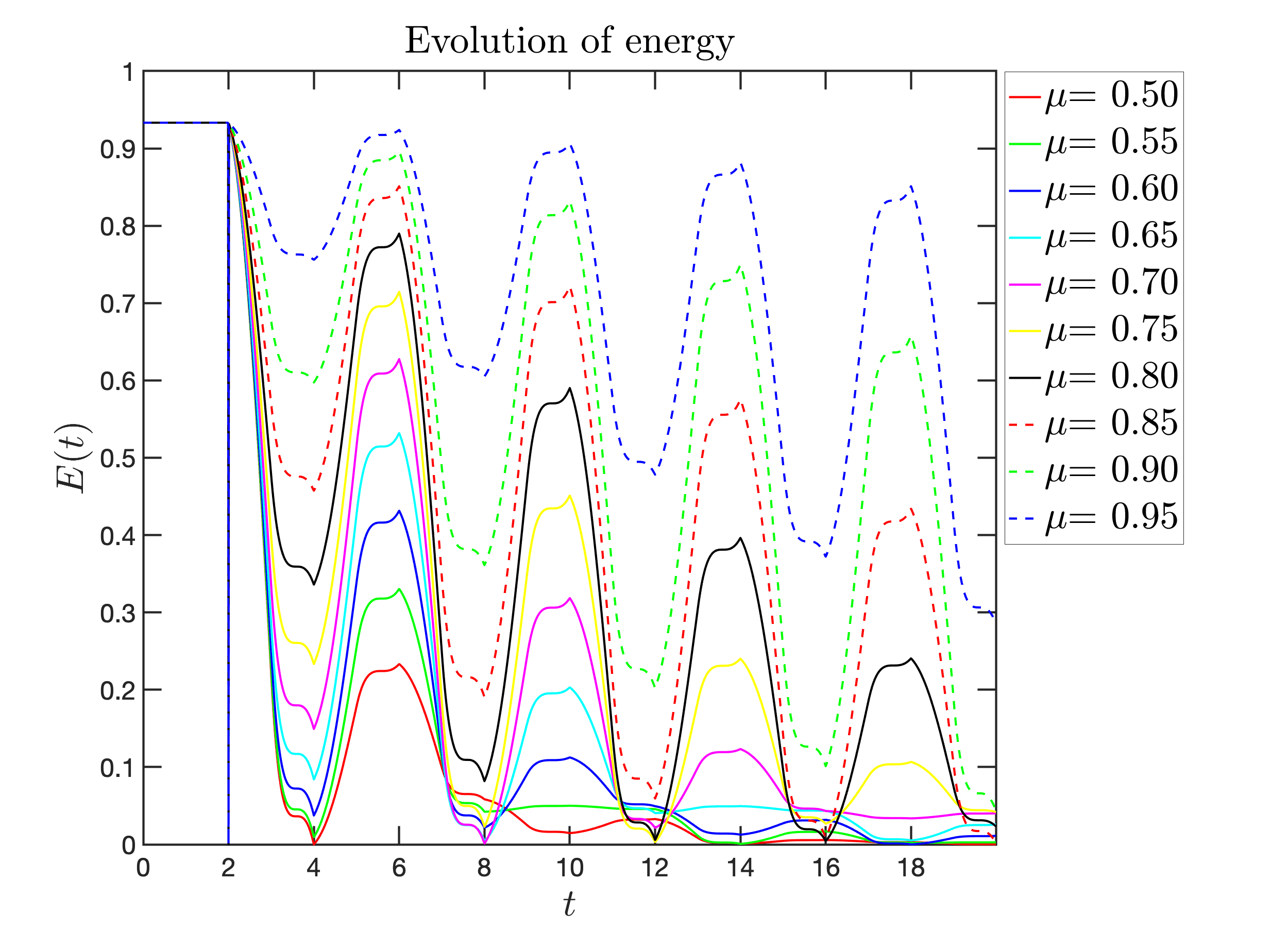}}
	\captionsetup{justification=centering}
	\caption{Boundary delayed control. Energy when $0 < \mu < 1$}
	\label{Energy-boundary}
\end{figure}

\newpage

\begin{figure}[H]
	\centering
	\subcaptionbox{Evolution of $-\ln(E(t))$ for  $0 < \mu \leq 0.5 $.\label{Ln-Energy-boundary-1}}
	{\includegraphics[width=0.9\textwidth]{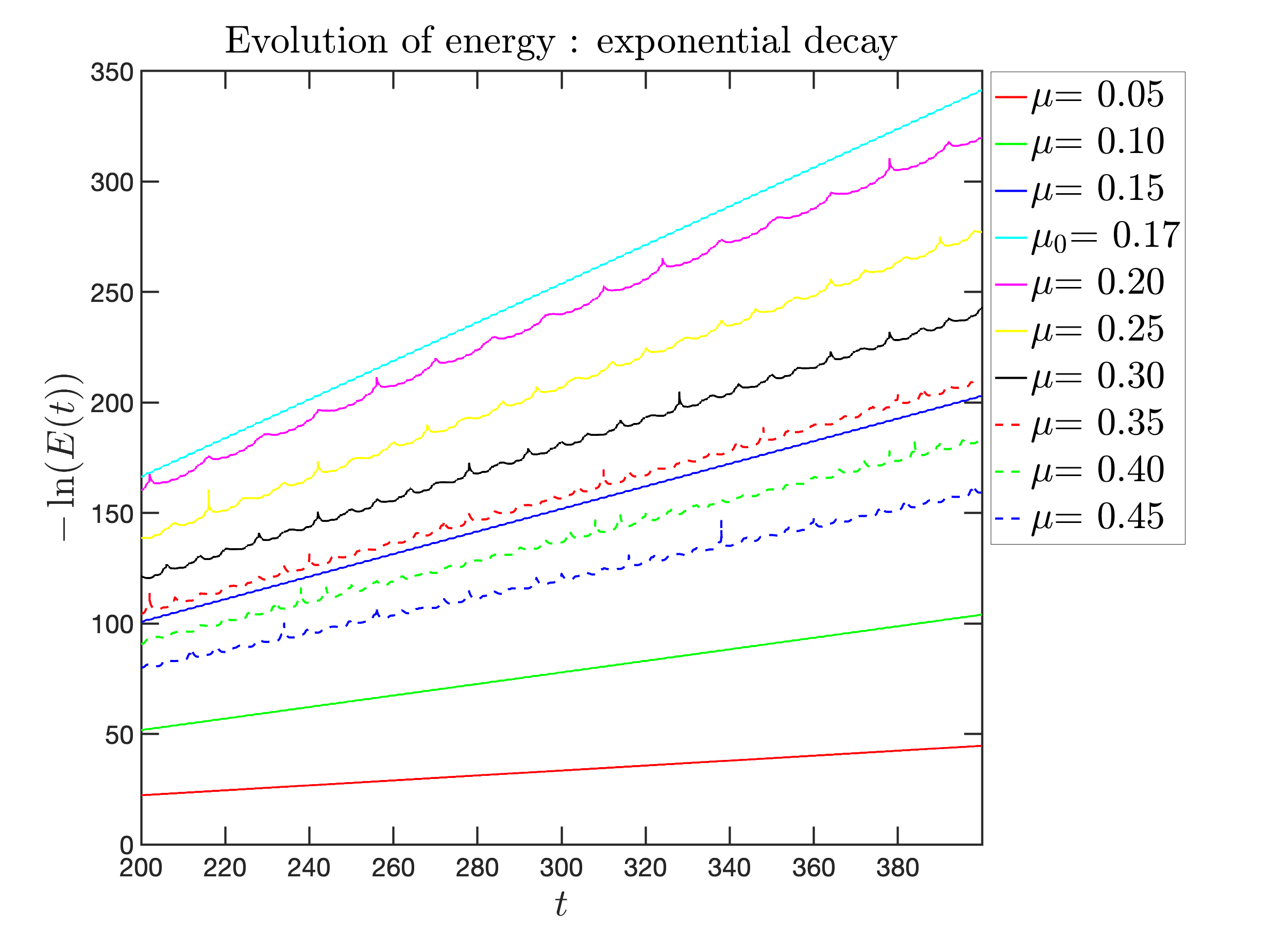}}
	\subcaptionbox{Evolution of $-\ln(E(t))$ for  $0.5 < \mu < 1$.\label{Ln-Energy-boundary-2}}
	{\includegraphics[width=0.9\textwidth]{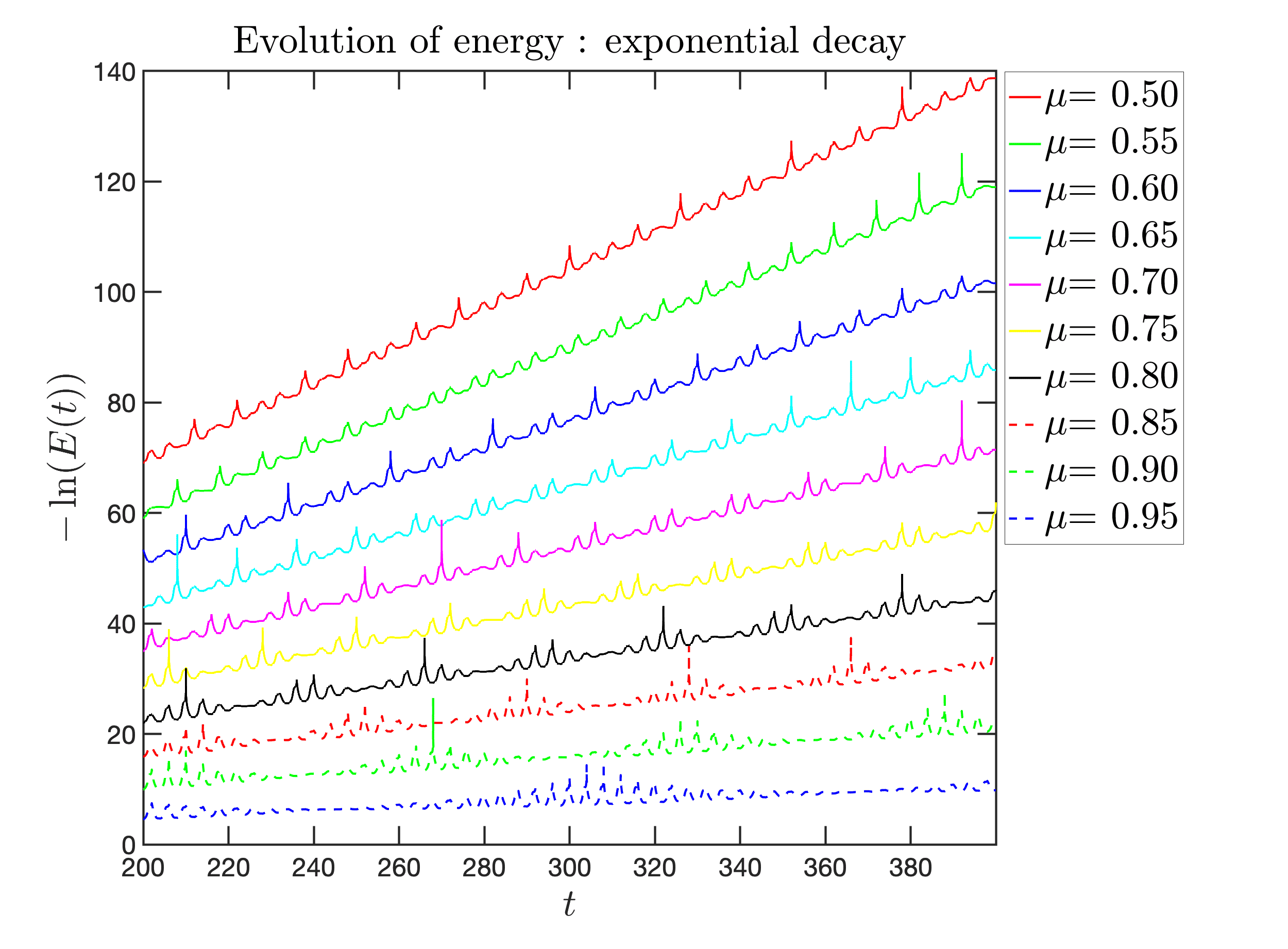}}
	\captionsetup{justification=centering}
	\caption{Boundary delayed control. Exponential decay $0 < \mu < 1$}
	\label{Ln-Energy-boundary}
\end{figure}

\newpage

\begin{figure}[H]
	\centering
	\subcaptionbox{Evolution of $- \dfrac{\ln(E(t))}{t}$ for  $0 < \mu \leq 0.5 $.\label{Exp-Energy-boundary-1}}
	{\includegraphics[width=0.9\textwidth]{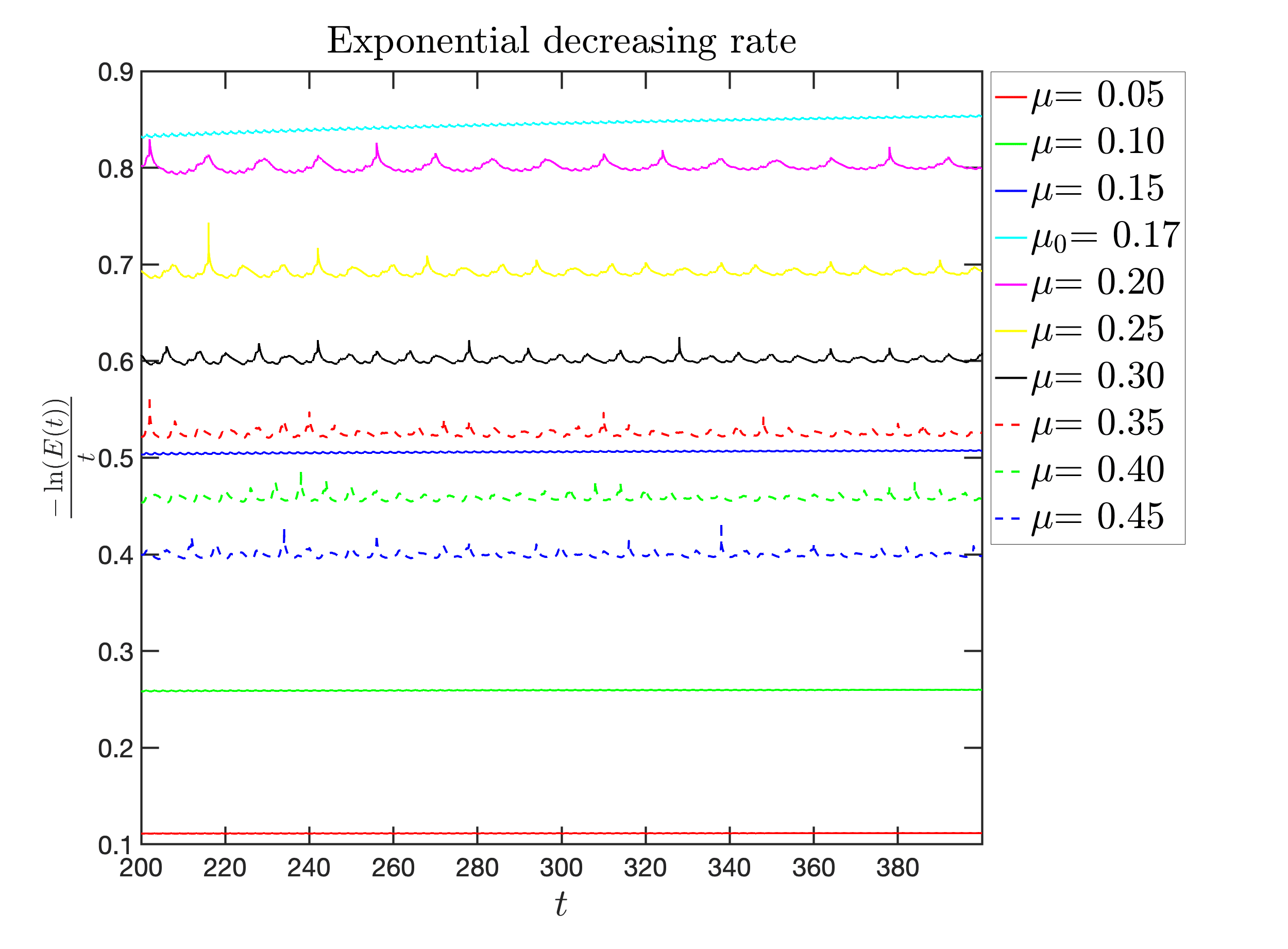}}
	\subcaptionbox{Evolution of $- \dfrac{\ln(E(t))}{t}$ for  $0.5 < \mu < 1$.\label{Exp-Energy-boundary-2}}
	{\includegraphics[width=0.9\textwidth]{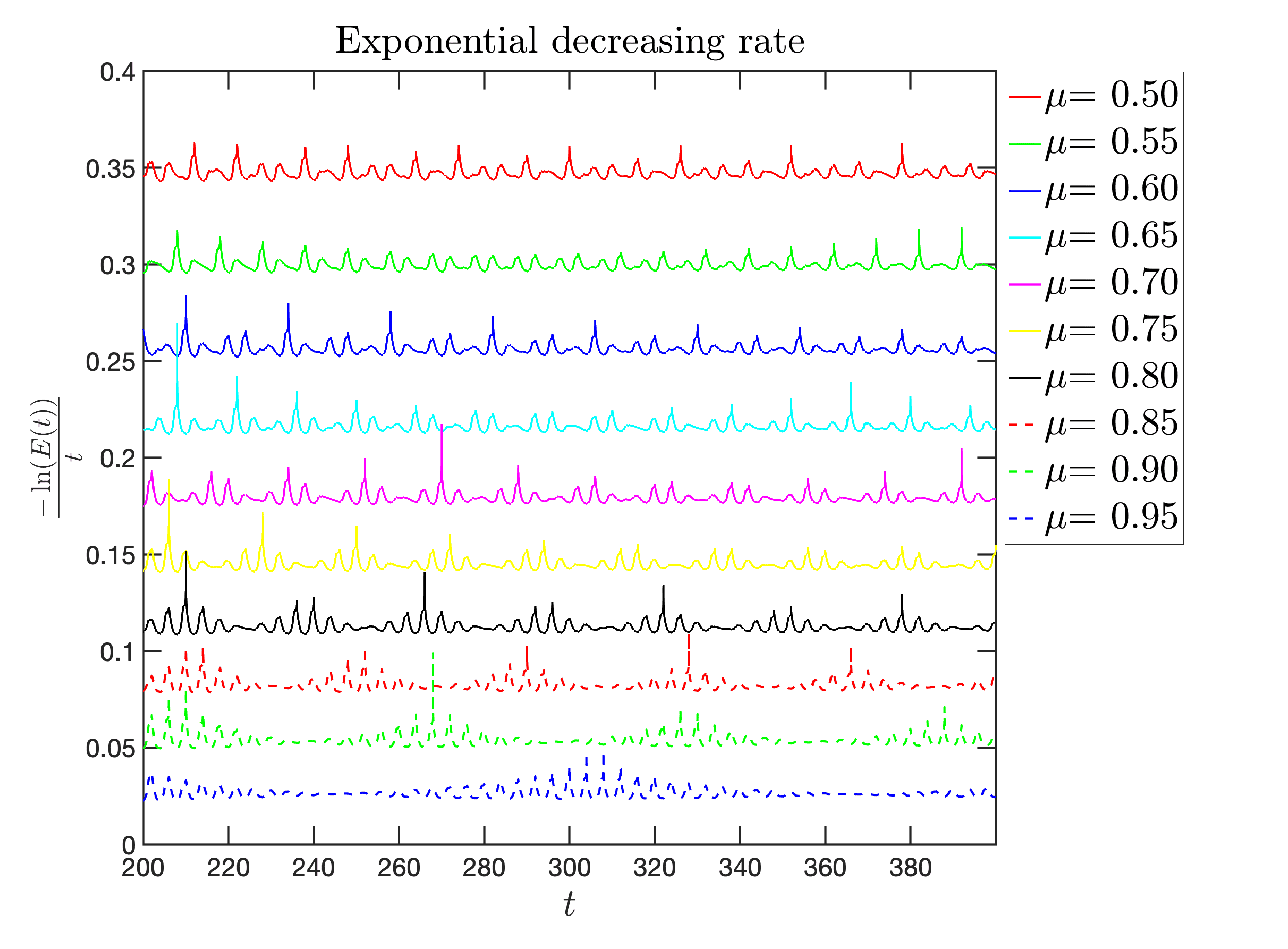}}
	\captionsetup{justification=centering}
	\caption{Boundary delayed control. Exponential decay rate $0 < \mu < 1$}
	\label{Exp-Energy-boundary}
\end{figure}

\newpage

\begin{figure}[H]
	\centering
	\subcaptionbox{Conservation of the energy for $\mu = 1$. \label{Energy-boundary-mu-1}}	
	{\includegraphics[width=0.9\textwidth]{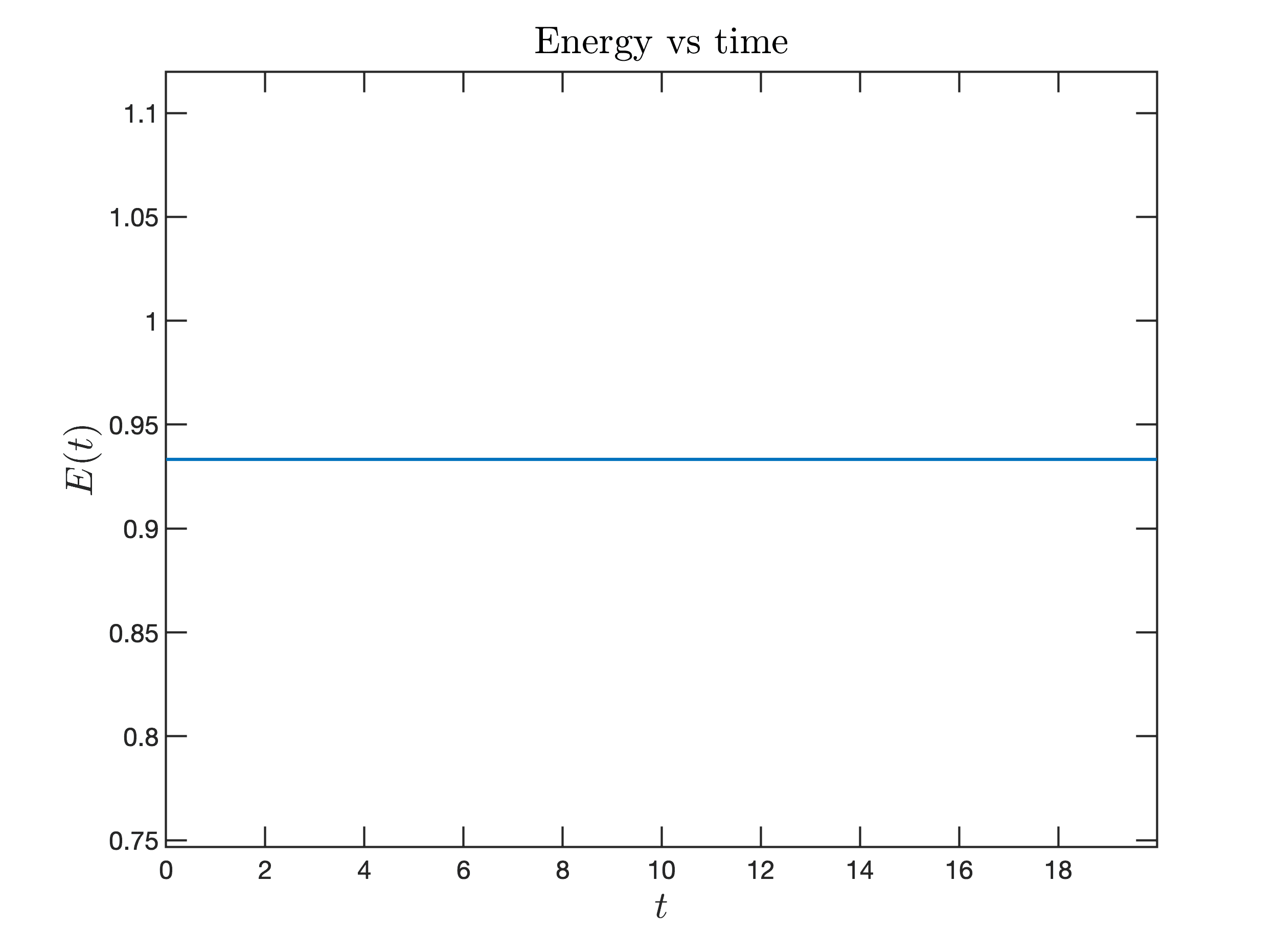}}	
	\subcaptionbox{Initial profile.\label{Profile-init}}	
	{\includegraphics[width=0.9\textwidth]{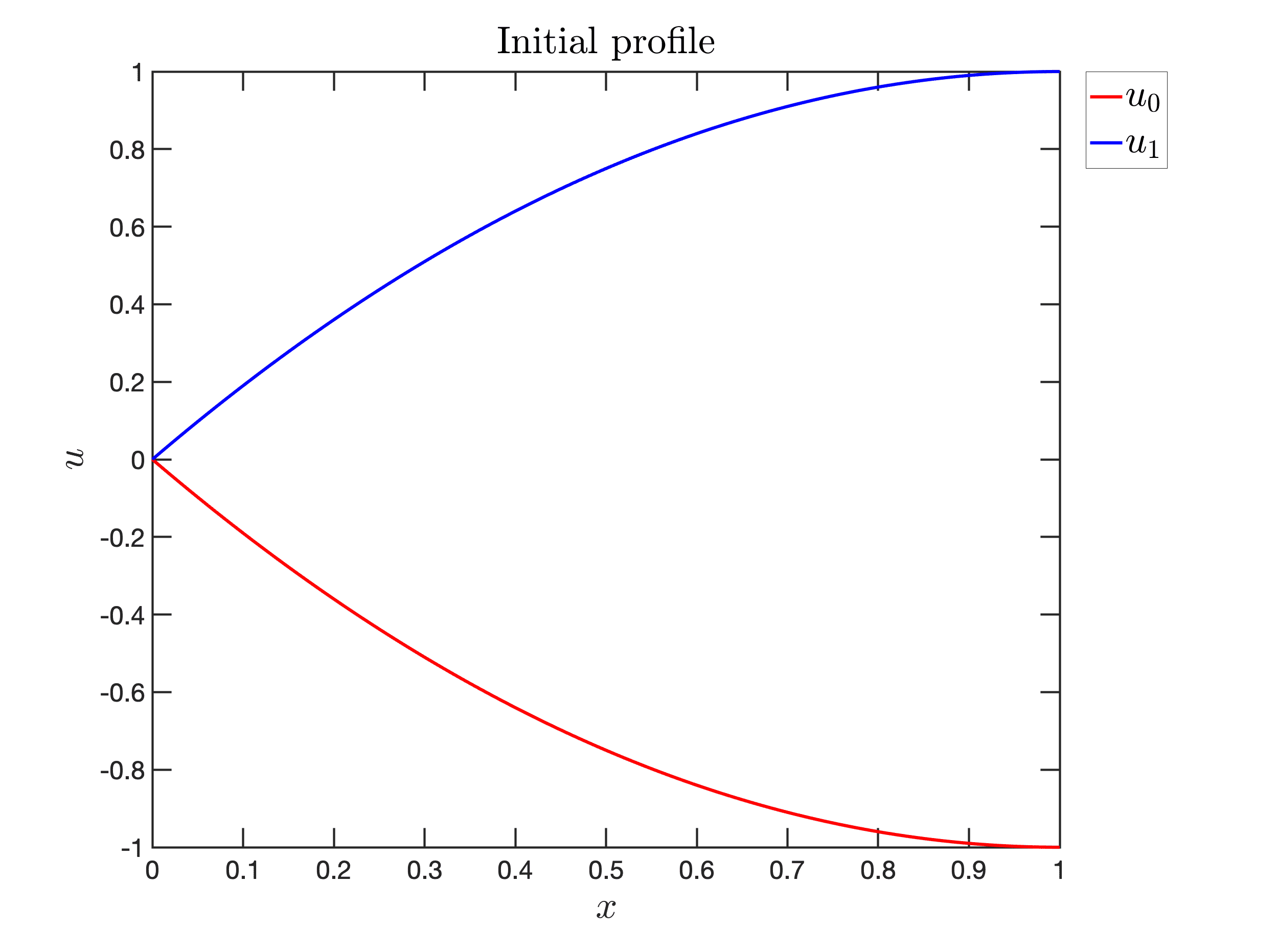}}	
	\captionsetup{justification=centering}
	\caption{Boundary delayed control. The surprising case $\mu = 1$}
	\label{Energy-boundary-mu-1-all}
\end{figure}
\newpage

\begin{figure}[H]
	\centering
	\subcaptionbox{Final profile for $T_{f} = 10 \times T = 20$\label{Profile-final-1}}
	{\includegraphics[width=0.9\textwidth]{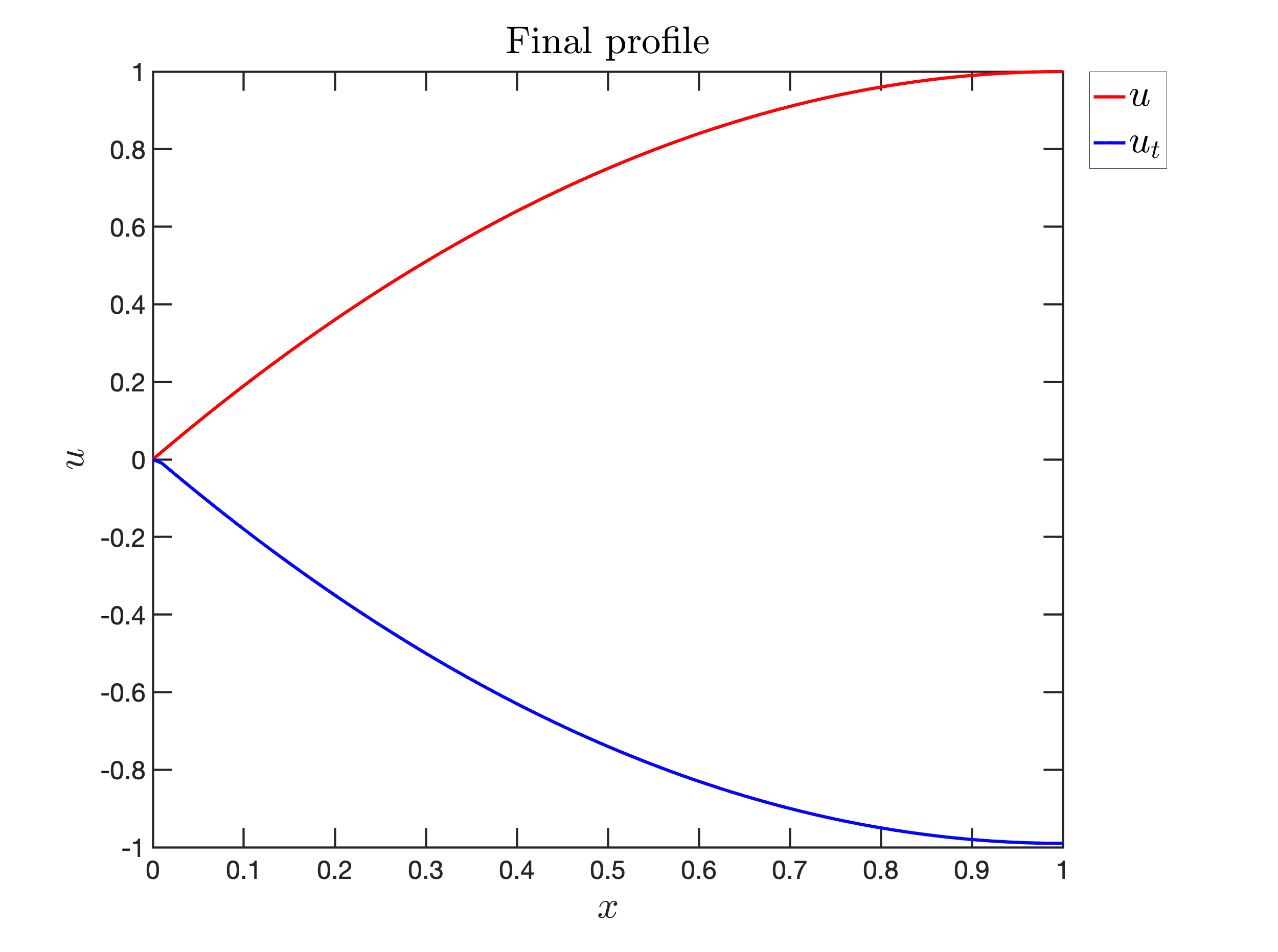}}
	\subcaptionbox{Final profile for $T_{f} = 11 \times T = 22$ .\label{Profile-final-2}}
	{\includegraphics[width=0.9\textwidth]{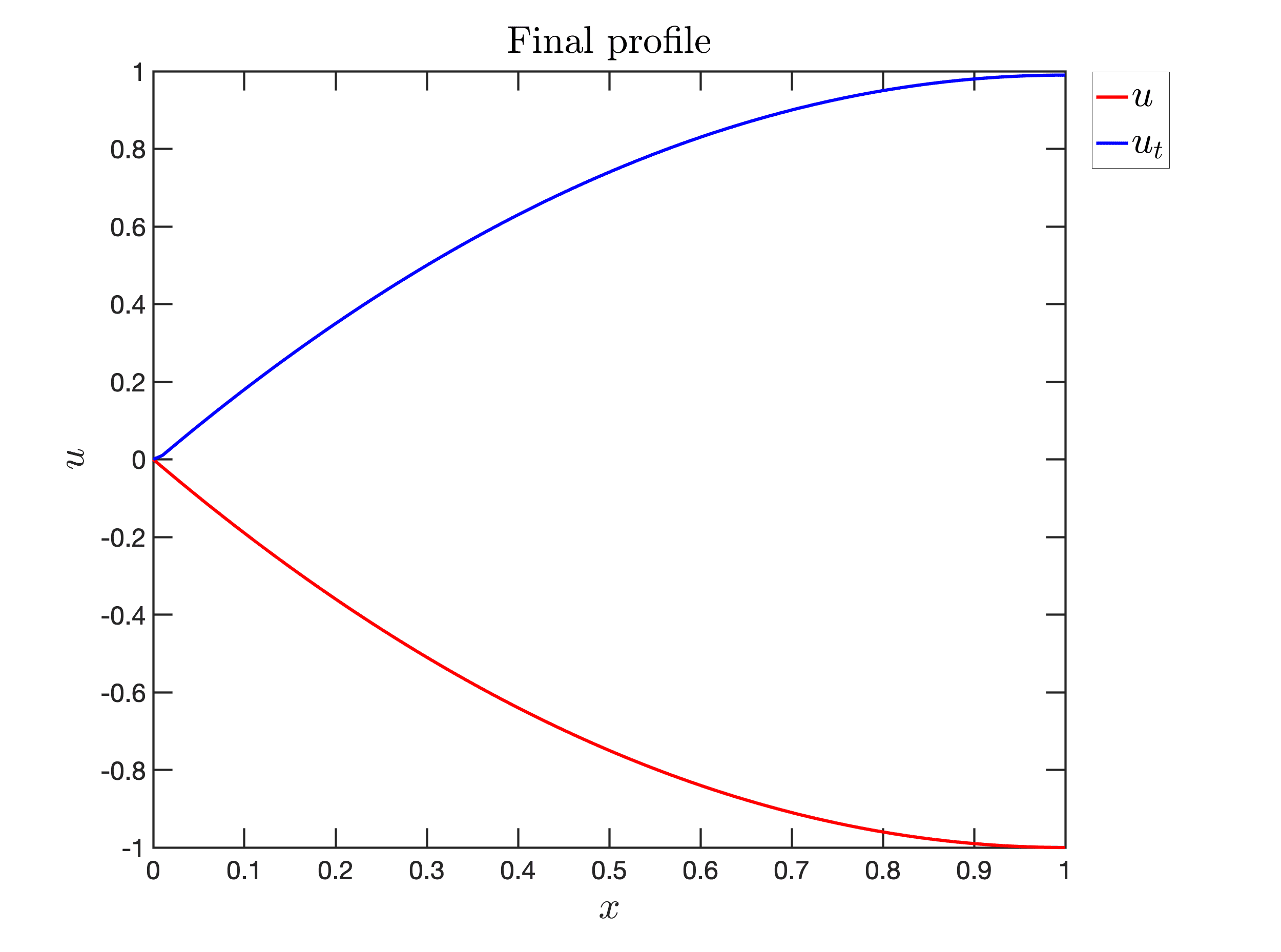}}
	\captionsetup{justification=centering}
	\caption{Boundary delayed control.Final profiles for $\mu = 1$}
	\label{Final}
\end{figure}

\newpage

\begin{figure}[H]
	\centering
	\subcaptionbox{Evolution of the discrete energy for $1 < \mu \leq 1.5 \,,\, t \in [0,20]$.\label{Energy-boundary-sup-1}}
	{\includegraphics[width=0.9\textwidth]{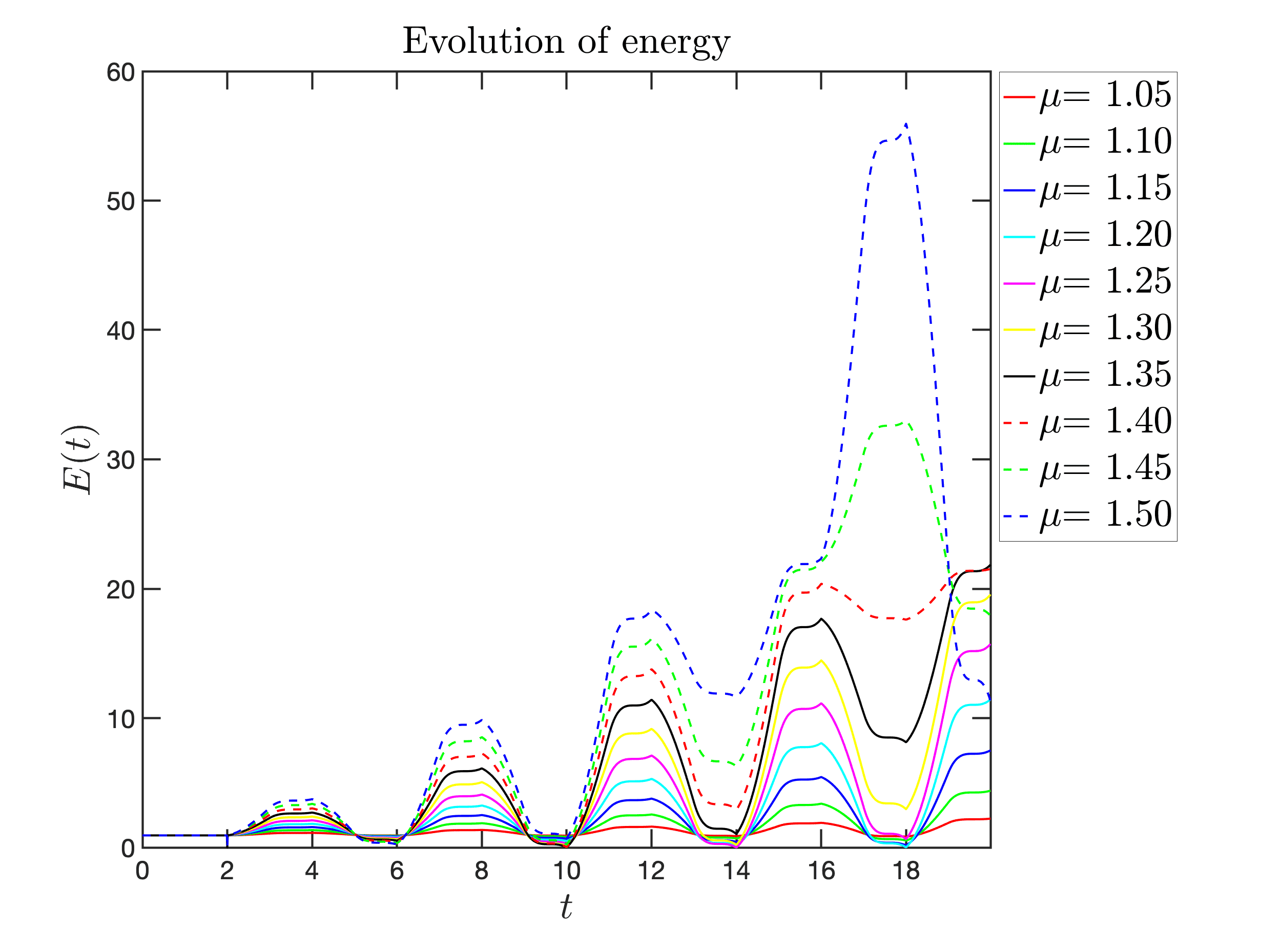}}
	\subcaptionbox{Evolution of the discrete energy for $1.5 < \mu \leq 2 \,,\, t \in [0,20]$.\label{Energy-boundary-sup-2}}
	{\includegraphics[width=0.9\textwidth]{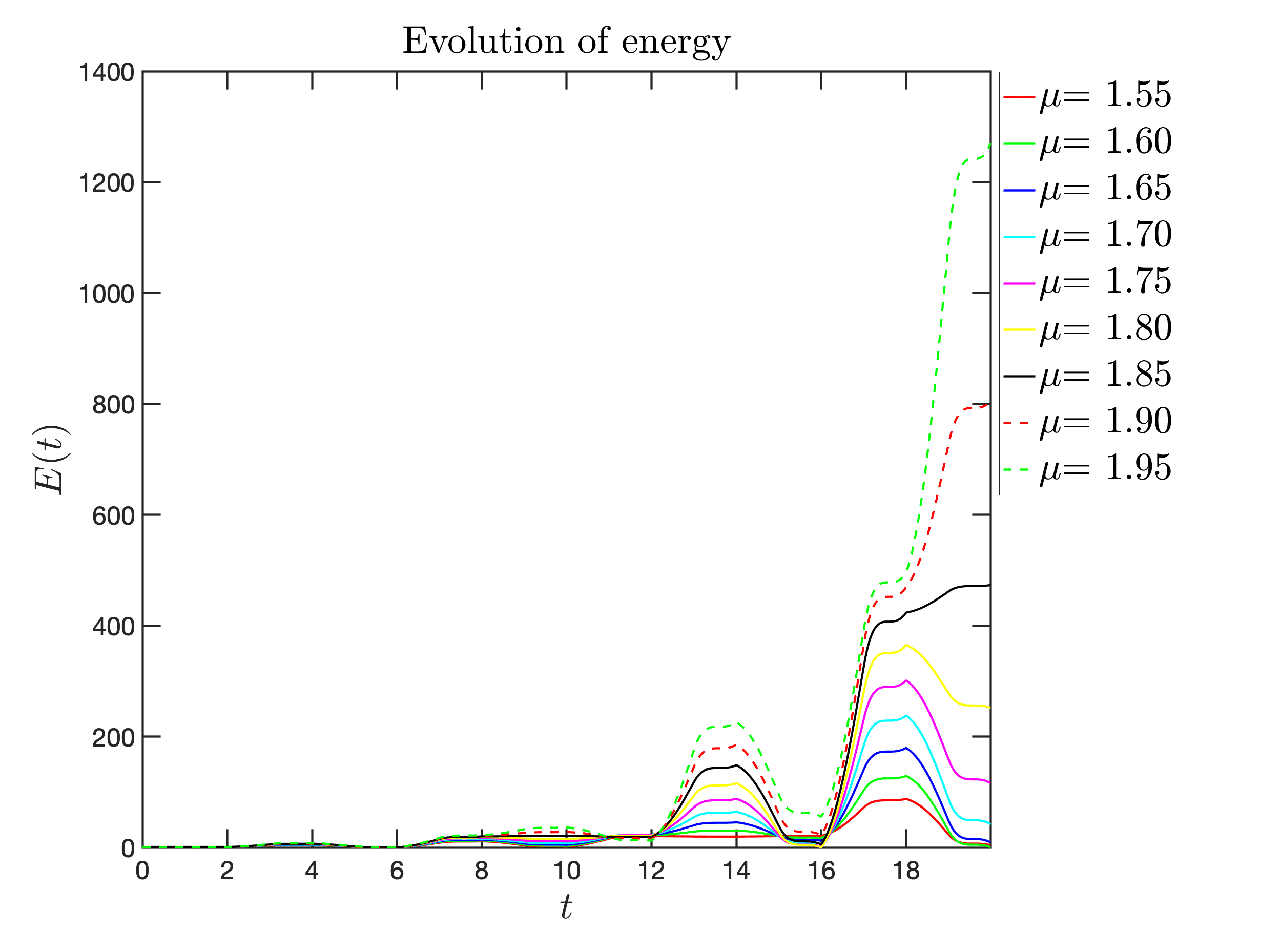}}
	\captionsetup{justification=centering}
	\caption{Boundary delayed control. Energy when $1 < \mu \leq 2$}
	\label{Energy-boundary-sup}
\end{figure}

%
%
%
%
%
%
\newpage

\begin{figure}[H]
	\centering
	\subcaptionbox{Evolution of the discrete energy for $-0.5 \leq \mu < 0 \,,\, t \in [0,20]$.\label{Energy-boundary-neg-1}}
	{\includegraphics[width=0.9\textwidth]{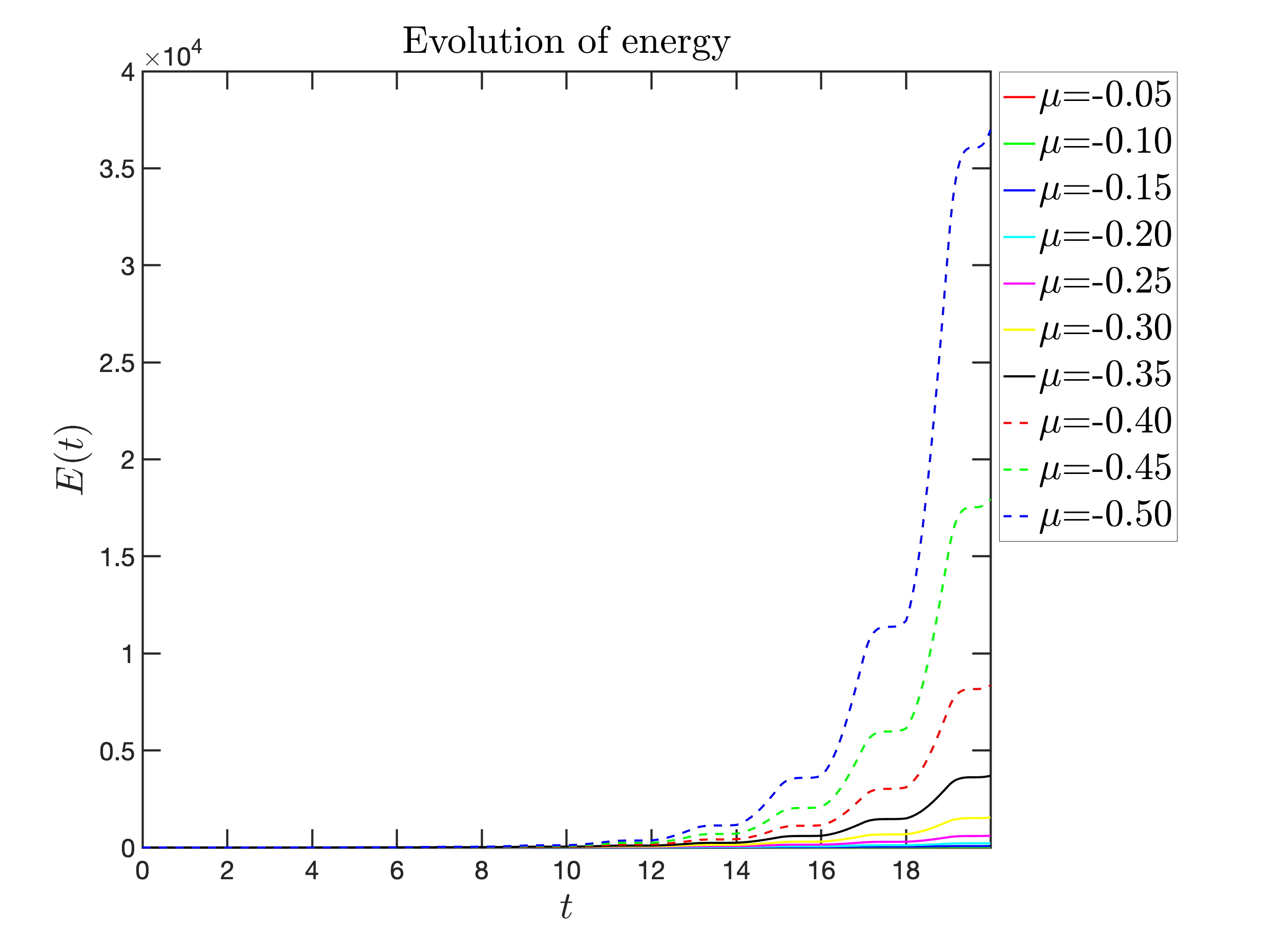}}
	\subcaptionbox{Evolution of the discrete energy for $-1 \leq \mu < -0.5 \,,\, t \in [0,20]$.\label{Energy-boundary-neg-2}}
	{\includegraphics[width=0.9\textwidth]{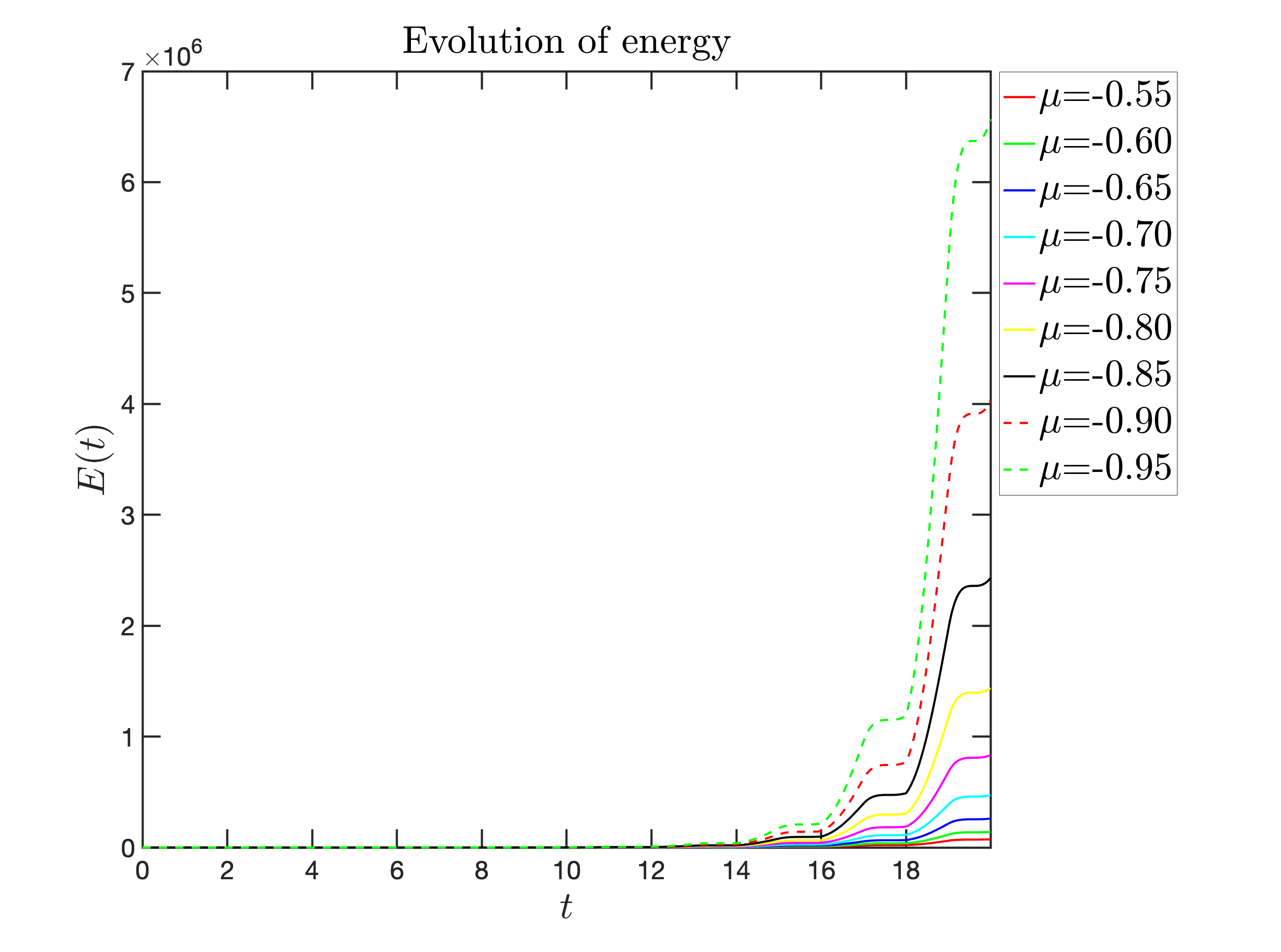}}
	\captionsetup{justification=centering}
	\caption{Boundary delayed control. Energy when $ -1 \leq  \mu < 0$}
	\label{Energy-boundary-neg}
\end{figure}


%
%
%
%

\newpage

\begin{figure}[H]
	\centering
	\subcaptionbox{Evolution of the discrete energy for $0 < \mu \leq 1 \,,\, t \in [0,20]$.\label{Energy-internal-1}}
	{\includegraphics[width=0.9\textwidth]{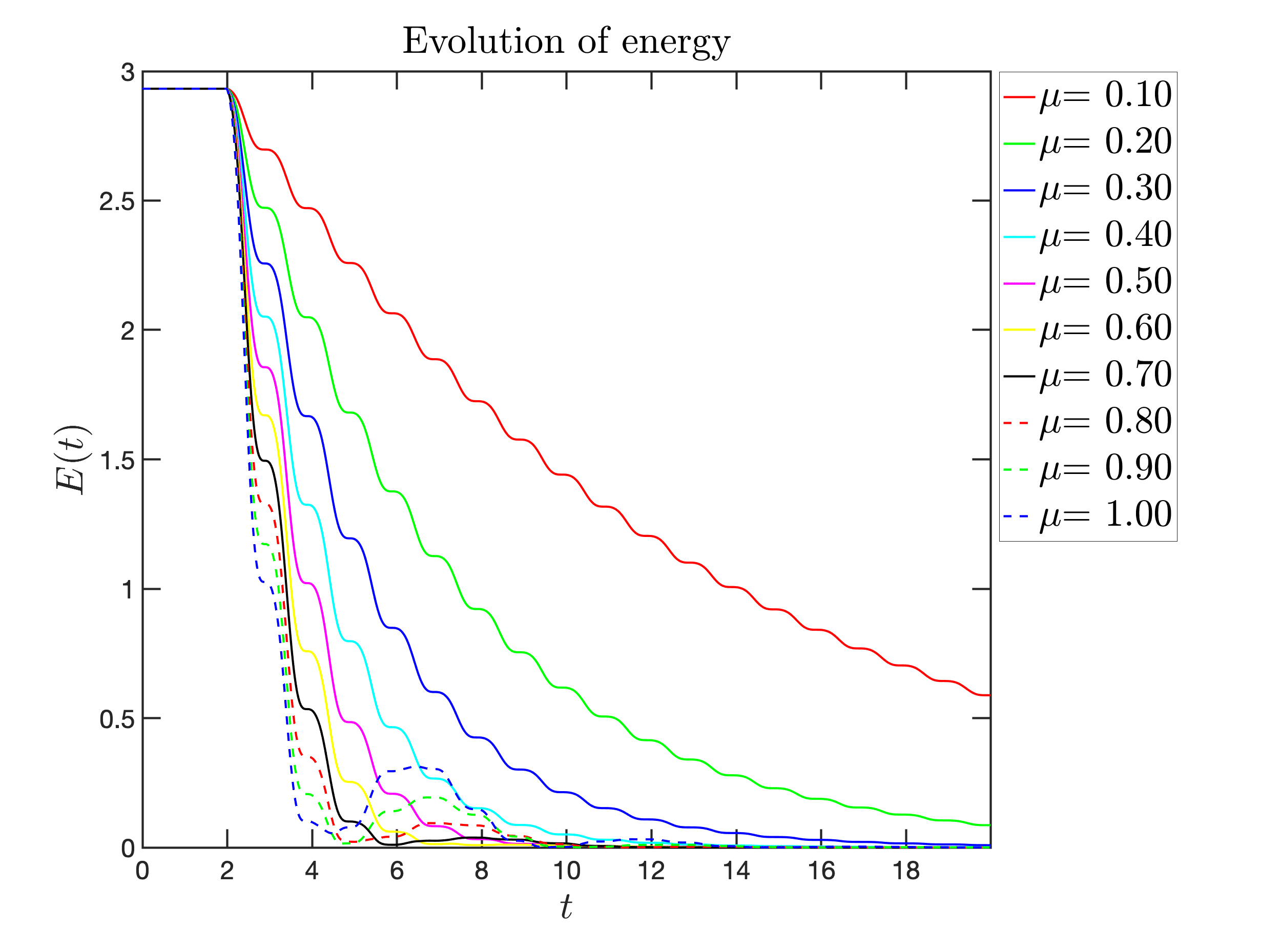}}
	\subcaptionbox{Evolution of the discrete energy for $1 < \mu < 2 \,,\, t \in [0,20]$.\label{Energy-internal-2}}
	{\includegraphics[width=0.9\textwidth]{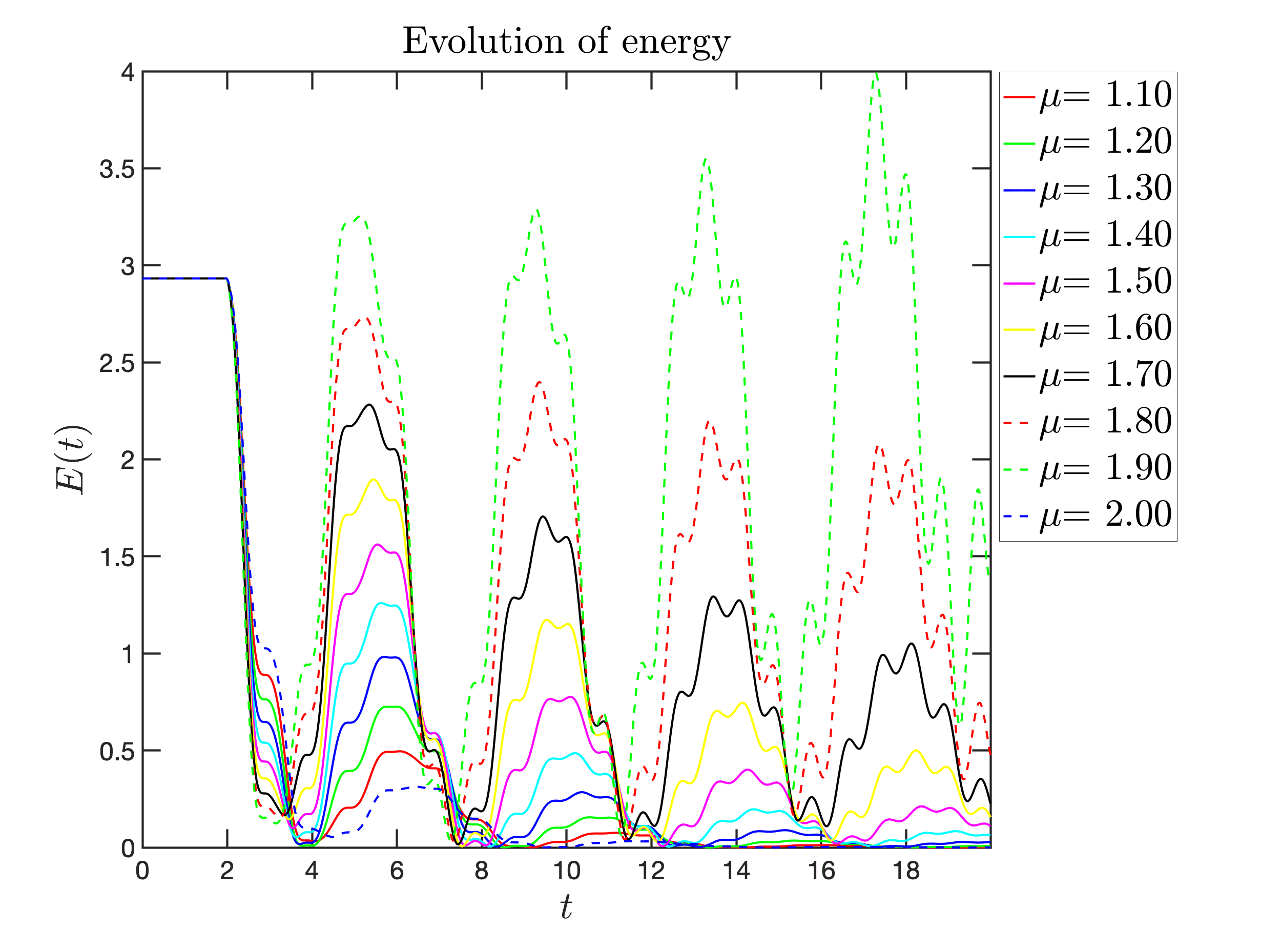}}
	\captionsetup{justification=centering}
	\caption{Internal delayed control. Energy when $0 < \mu < 2$}
	\label{Energy-internal}
\end{figure}

\begin{figure}[H]
	\centering
	\subcaptionbox{Evolution of $-\ln(E(t))$ for  $0 < \mu \leq 1 $.\label{Ln-Energy-internal-1}}
	{\includegraphics[width=0.9\textwidth]{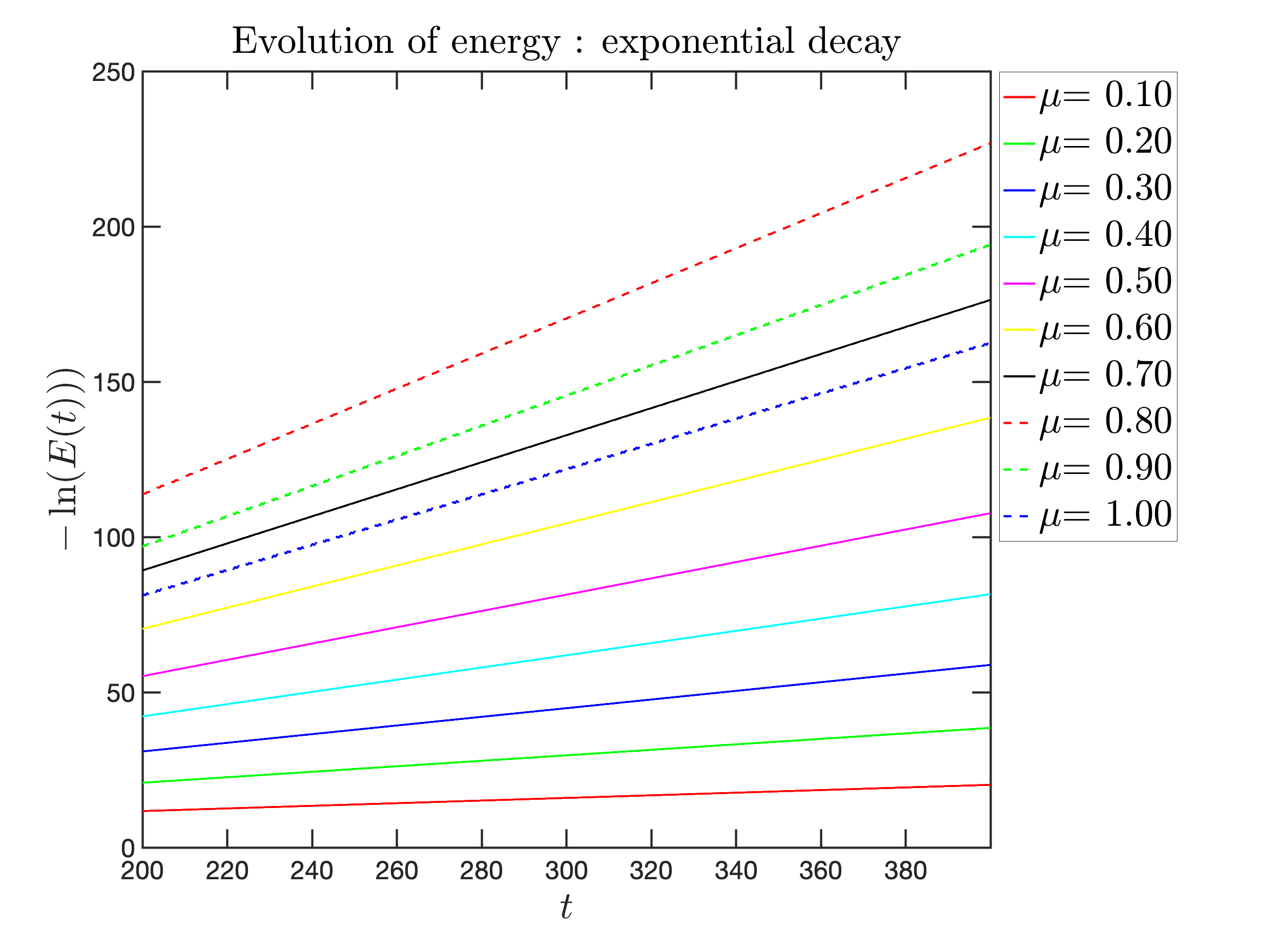}}
	\subcaptionbox{Evolution of $-\ln(E(t))$ for  $1 < \mu < 2$.\label{Ln-Energy-internal-2}}
	{\includegraphics[width=0.9\textwidth]{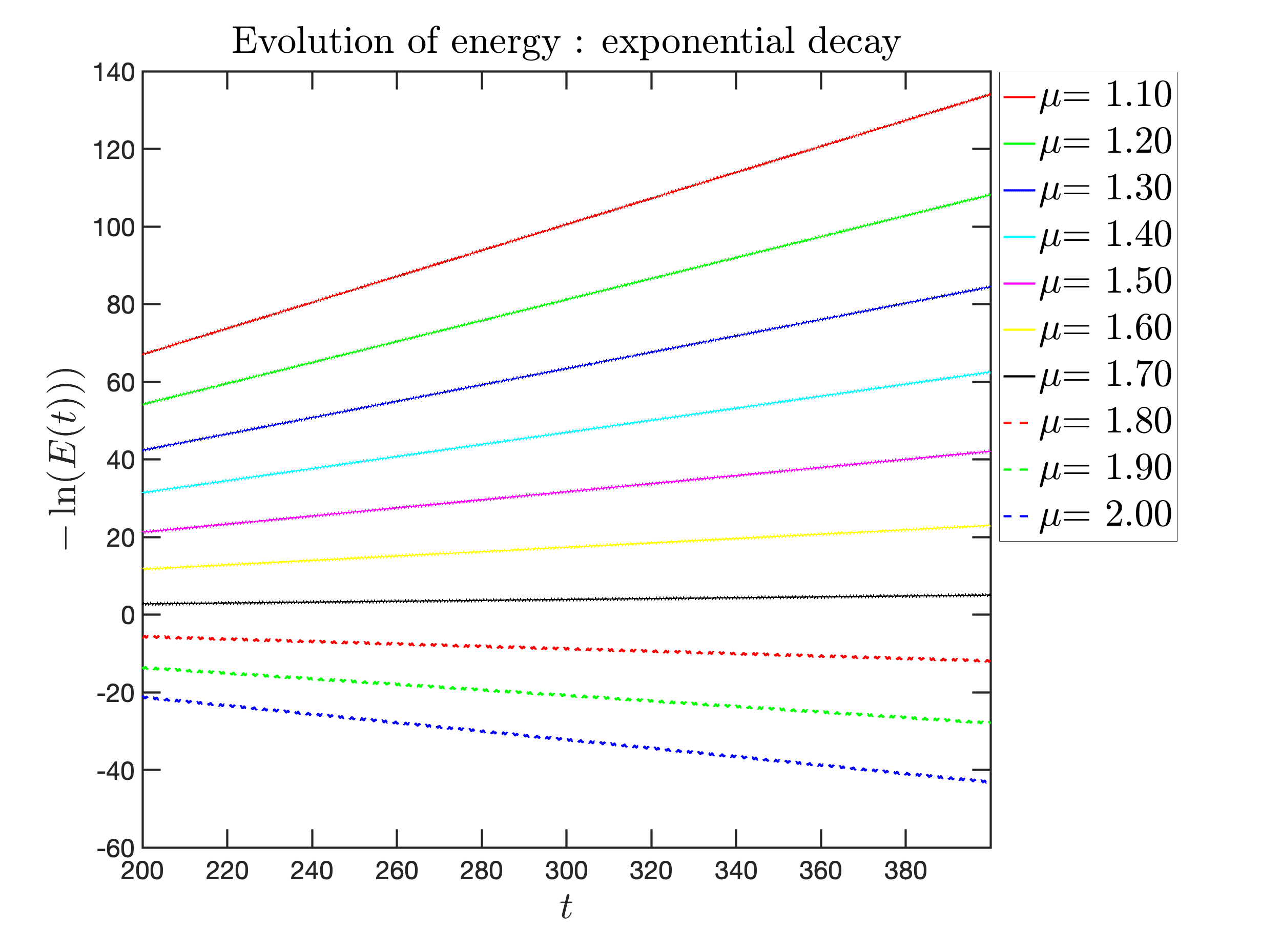}}
	\captionsetup{justification=centering}
	\caption{Internal delayed control. Exponential decay $0 < \mu < \mu_{0}$ and exponential growth $\mu_{0} < \mu$}
	\label{Ln-Energy-internal}
\end{figure}
\newpage
\begin{figure}[H]
	\centering
	\subcaptionbox{Evolution of $- \dfrac{\ln(E(t))}{t}$ for  $0 < \mu \leq 1 $.\label{Exp-Energy-internal-1}}
	{\includegraphics[width=0.8\textwidth]{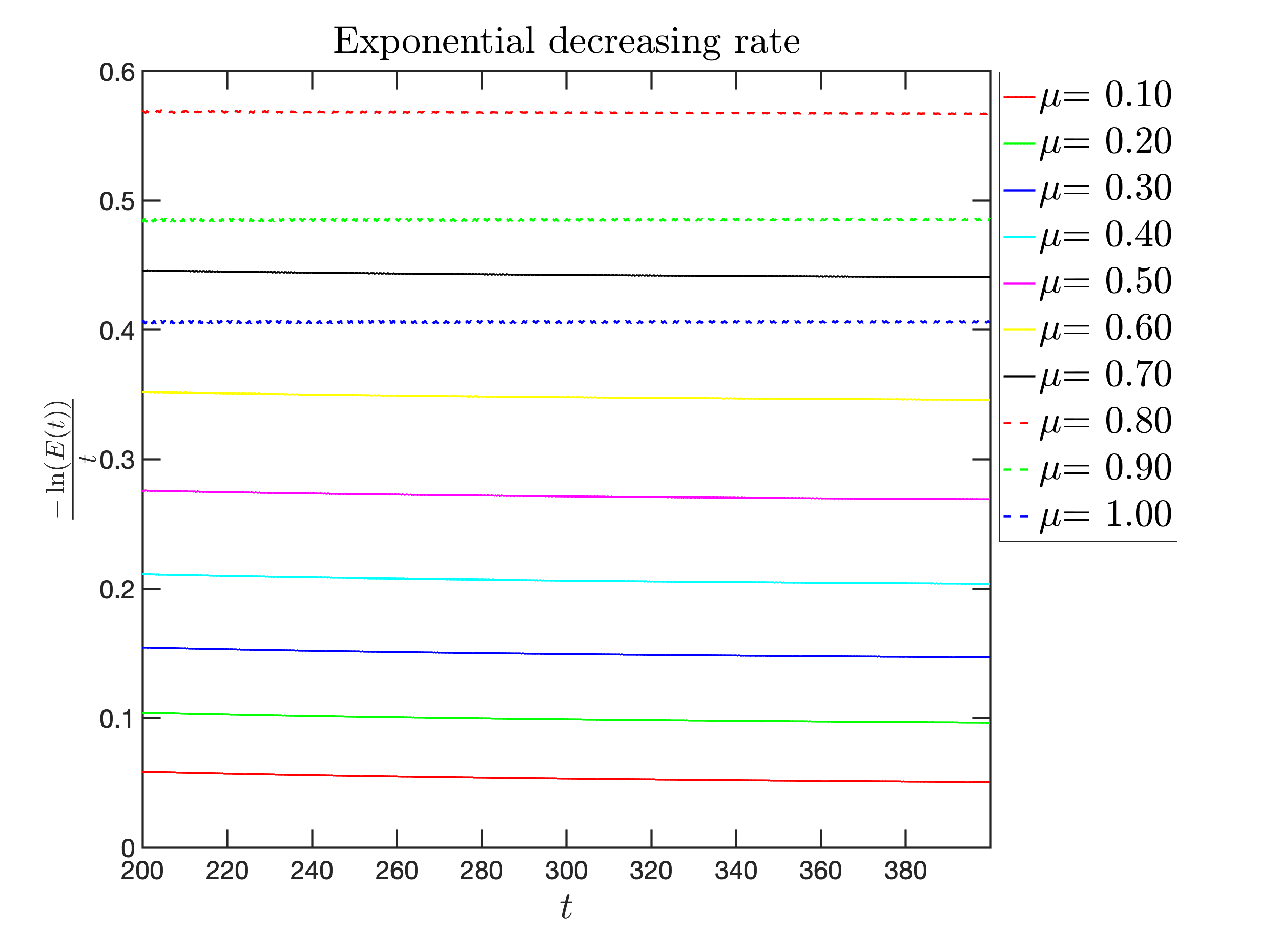}}
	\subcaptionbox{Evolution of $- \dfrac{\ln(E(t))}{t}$ for  $1 < \mu < 2$.\label{Exp-Energy-internal-2}}
	{\includegraphics[width=0.8\textwidth]{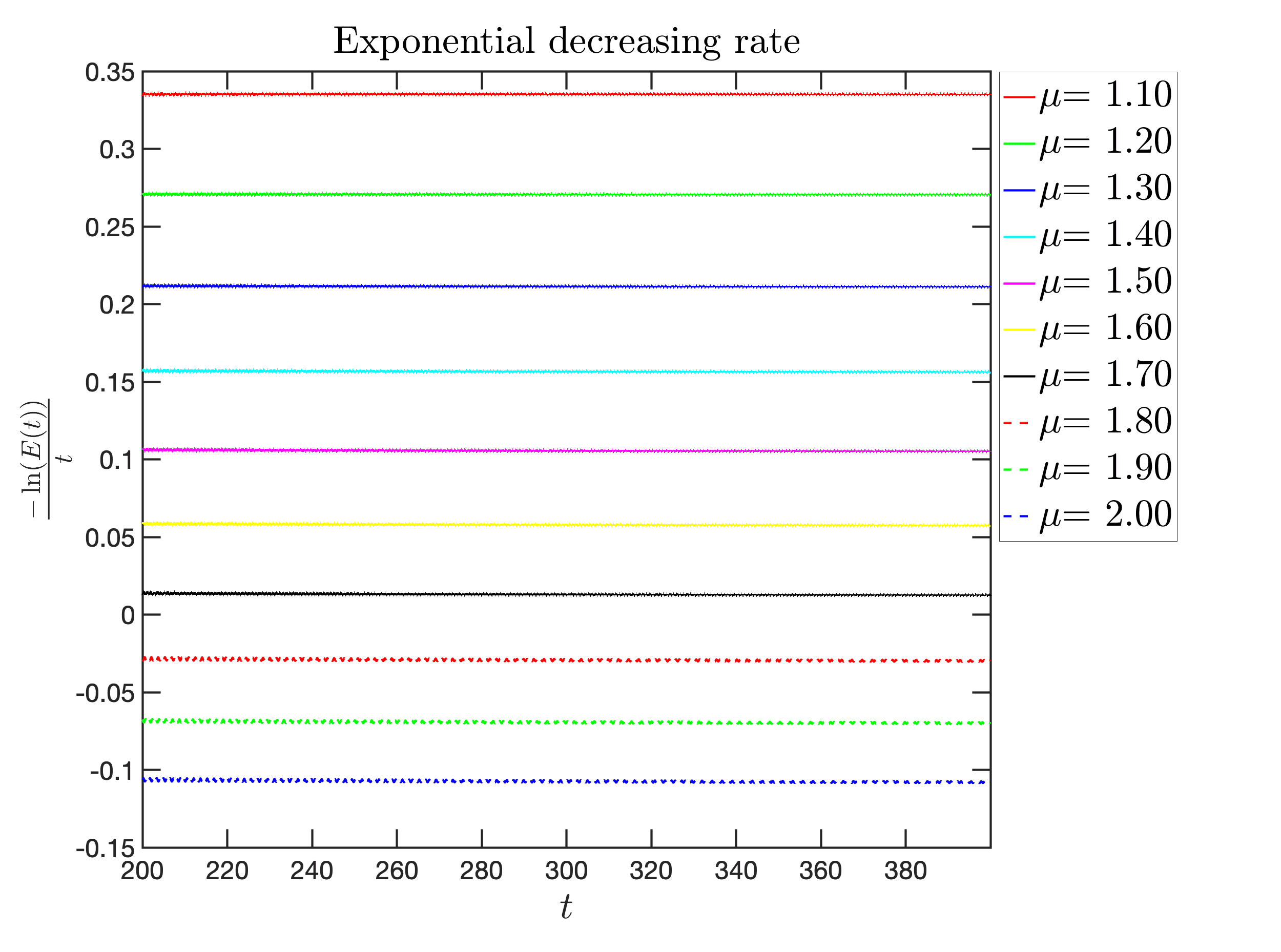}}
	\captionsetup{justification=centering}
	\caption{Internal delayed control. Exponential decay rate $0 < \mu < \mu_{0}$ and \linebreak exponential growth rate for $\mu_{0} < \mu$}
	\label{Exp-Energy-internal}
\end{figure}
\newpage
\begin{figure}[H]
	\centering
	\subcaptionbox{Evolution of the discrete energy for $-1 \leq \mu < 0 \,,\, t \in [0,20]$.\label{Energy-internal-neg-1}}
	{\includegraphics[width=0.9\textwidth]{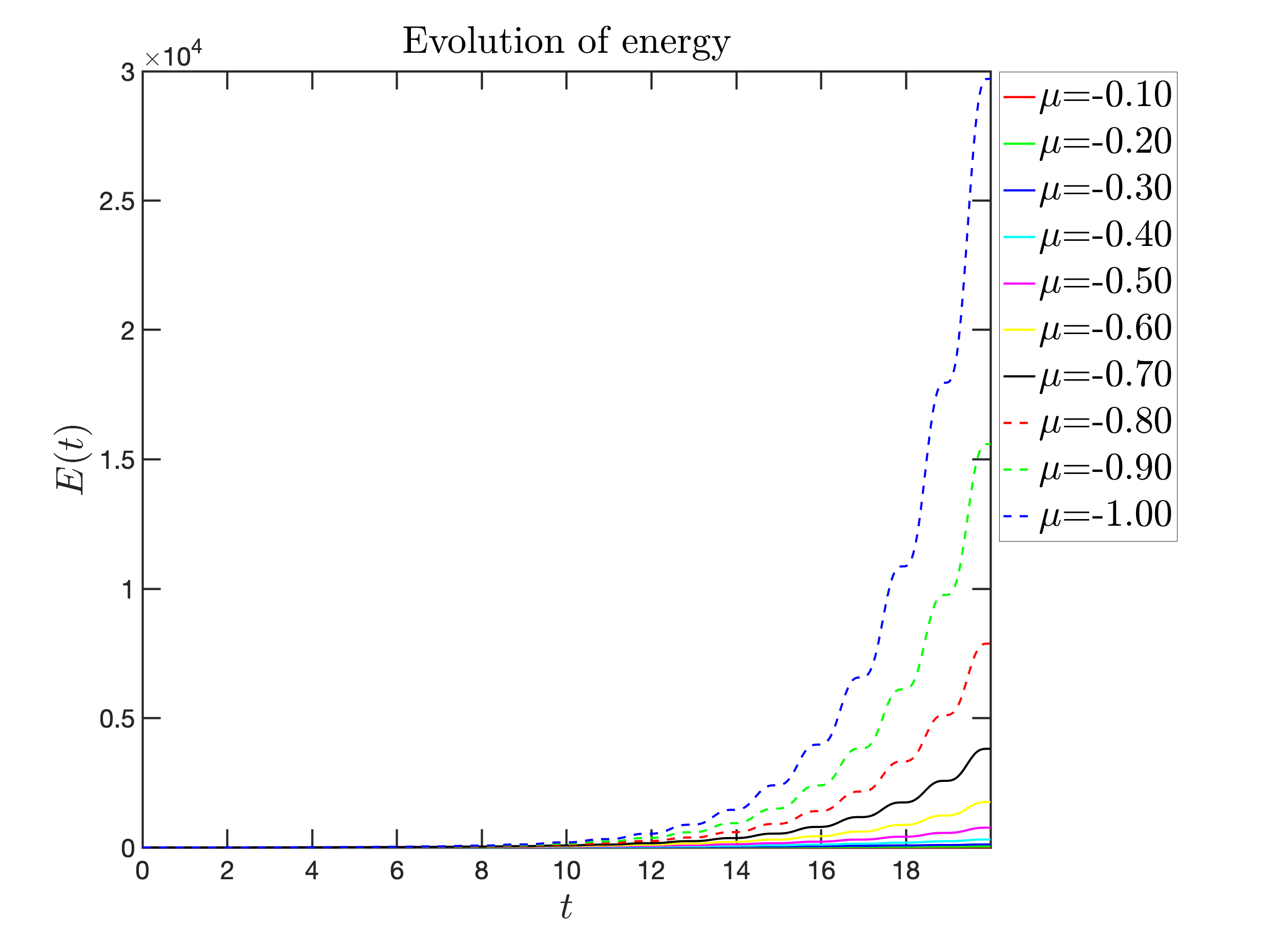}}
	\subcaptionbox{Evolution of the discrete energy for $-2 \leq \mu < -1\,,\, t \in [0,20]$.\label{Energy-internal-neg-2}}
	{\includegraphics[width=0.9\textwidth]{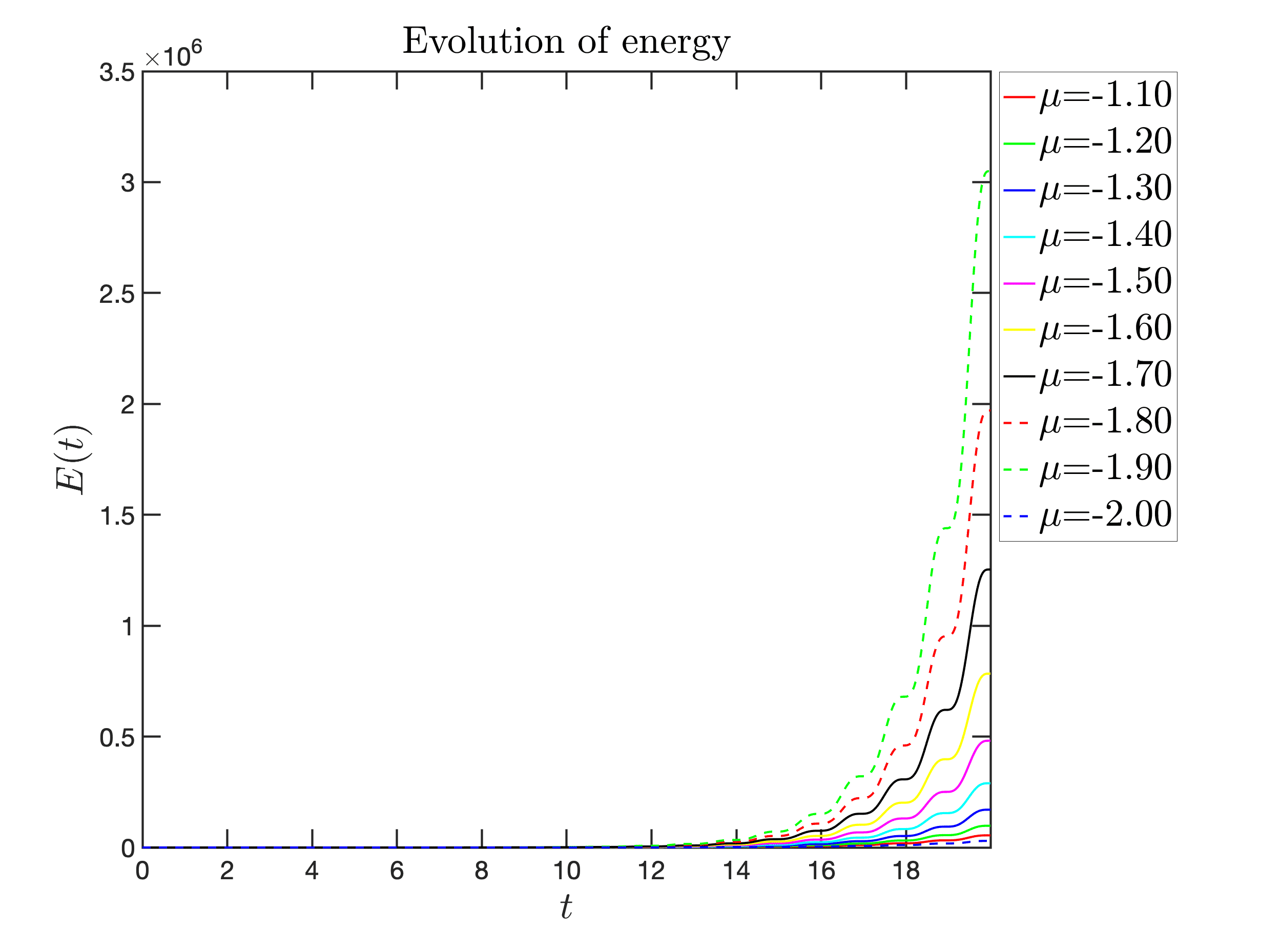}}
	\captionsetup{justification=centering}
	\caption{Internal delayed control. Energy when $ -2 \leq  \mu < 0$}
	\label{Energy-internal-neg}
\end{figure}

\newpage
%
%
%
%
%
%
%

\begin{figure}[H]
	\centering
	\subcaptionbox{Evolution of the discrete energy for $0 < \mu \leq 1 \,,\, t \in [0,20]$.\label{Energy-pointwise-1}}
	{\includegraphics[width=0.9\textwidth]{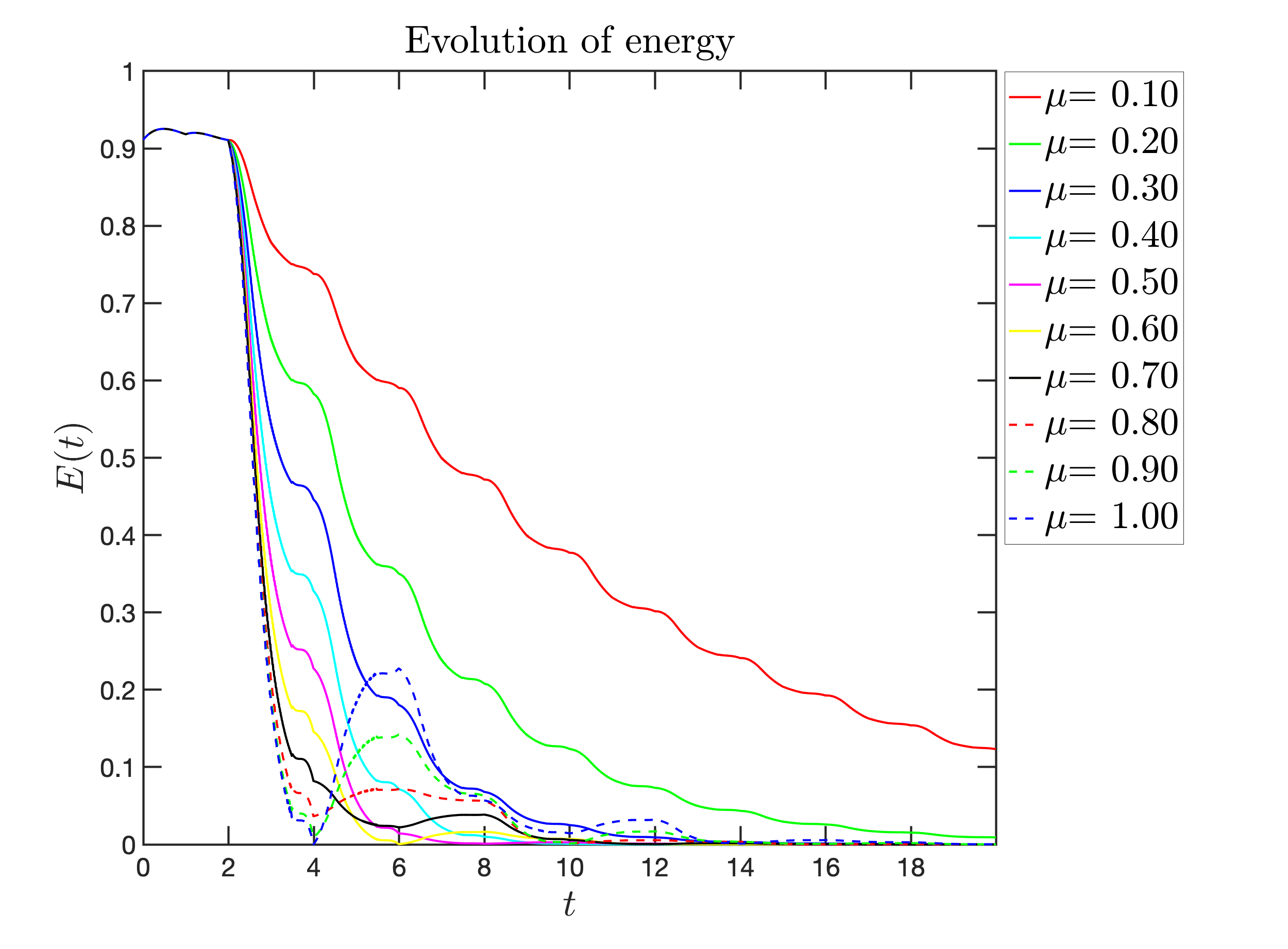}}
	\subcaptionbox{Evolution of the discrete energy for $1< \mu \leq 2 \,,\, t \in [0,20]$.\label{Energy-pointwise-2}}
	{\includegraphics[width=0.9\textwidth]{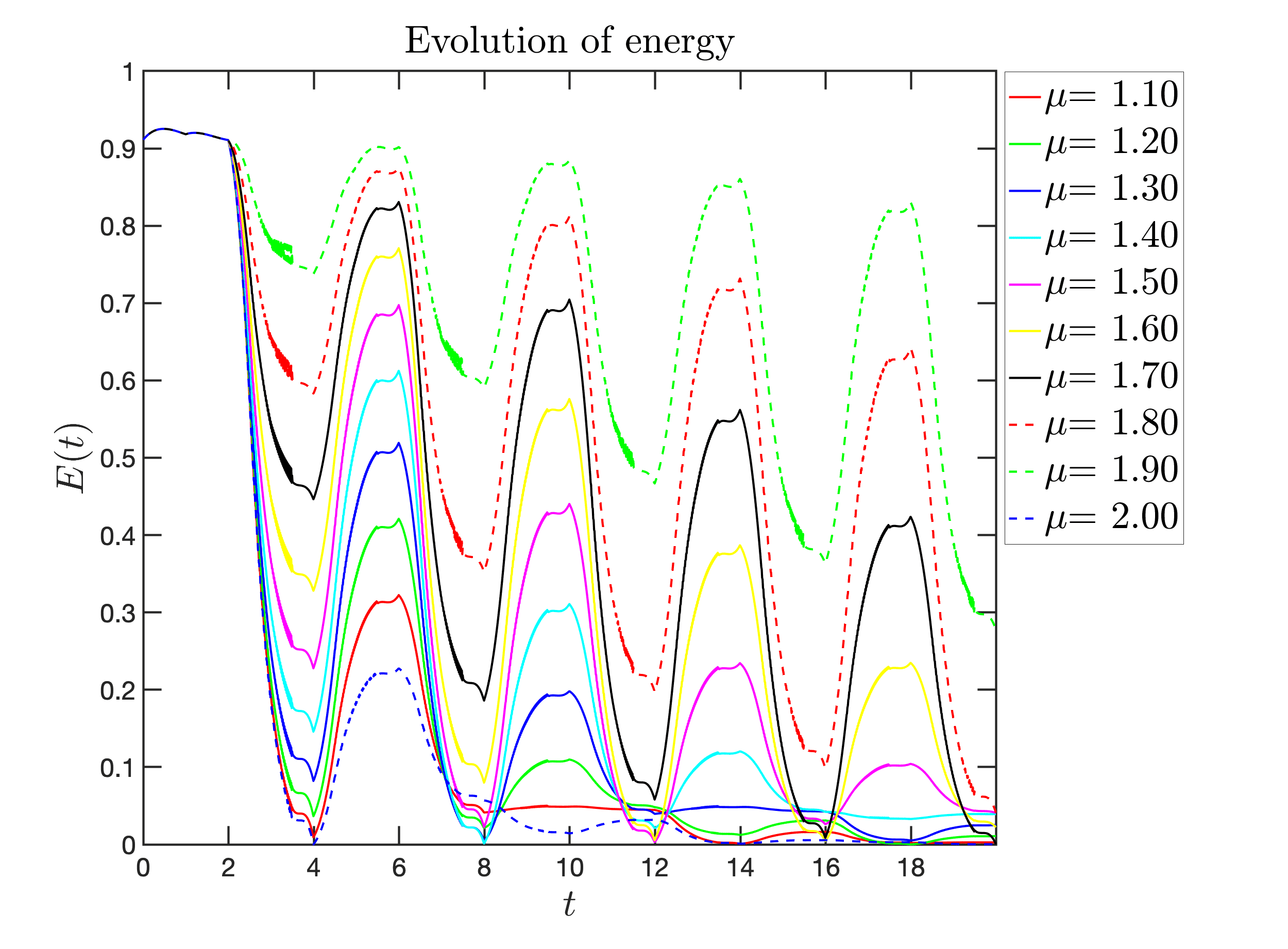}}
	\captionsetup{justification=centering}
	\caption{Pointwise delayed control. Energy when $0 < \mu \leq 2$}
	\label{Energy-pointwise}
\end{figure}

\newpage

\begin{figure}[H]
	\centering
	\subcaptionbox{Evolution of $-\ln(E(t))$ for  $0 < \mu \leq 1 $.\label{Ln-Energy-pointwise-1}}
	{\includegraphics[width=0.9\textwidth]{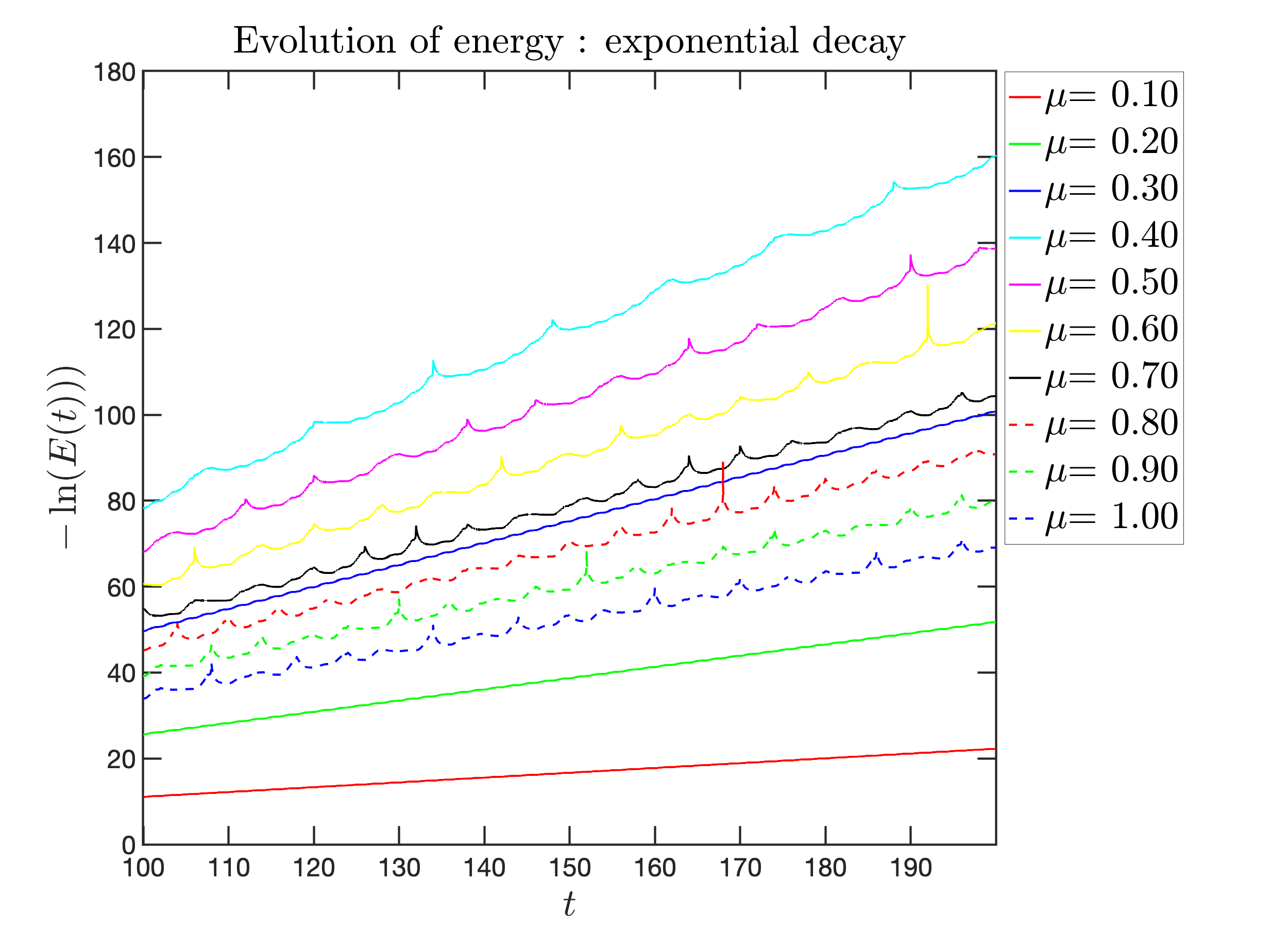}}
	\subcaptionbox{Evolution of $-\ln(E(t))$ for  $1 < \mu \leq 2$.\label{Ln-Energy-pointwise-2}}
	{\includegraphics[width=0.9\textwidth]{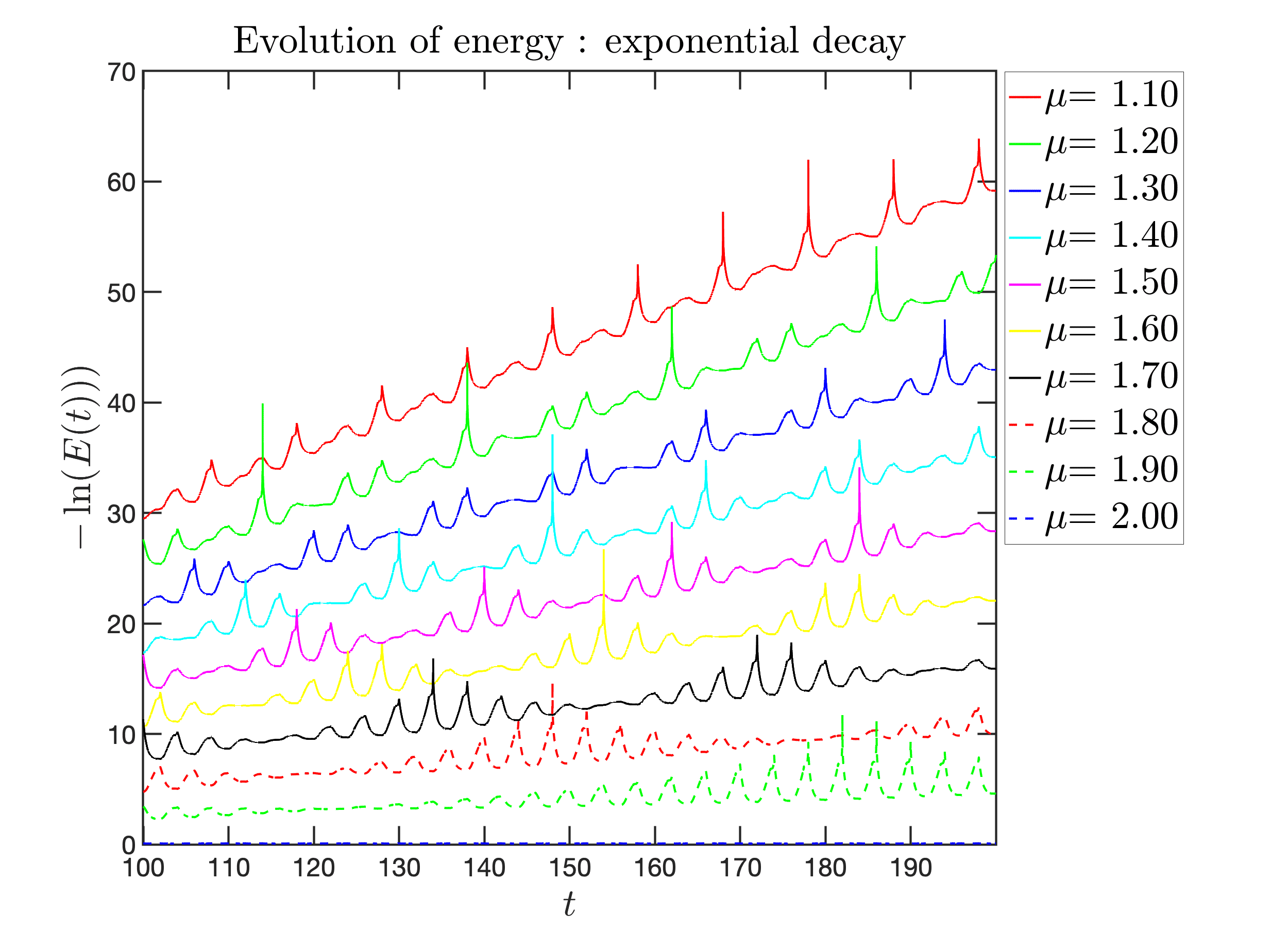}}
	\captionsetup{justification=centering}
	\caption{Pointwise delayed control. Exponential decay $0 < \mu \leq 2$}
	\label{Ln-Energy-pointwise}
\end{figure}

\newpage

\begin{figure}[H]
	\centering
	\subcaptionbox{Evolution of $- \dfrac{\ln(E(t))}{t}$ for  $0 < \mu \leq 1 $.\label{Exp-Energy-pointwise-1}}
	{\includegraphics[width=0.9\textwidth]{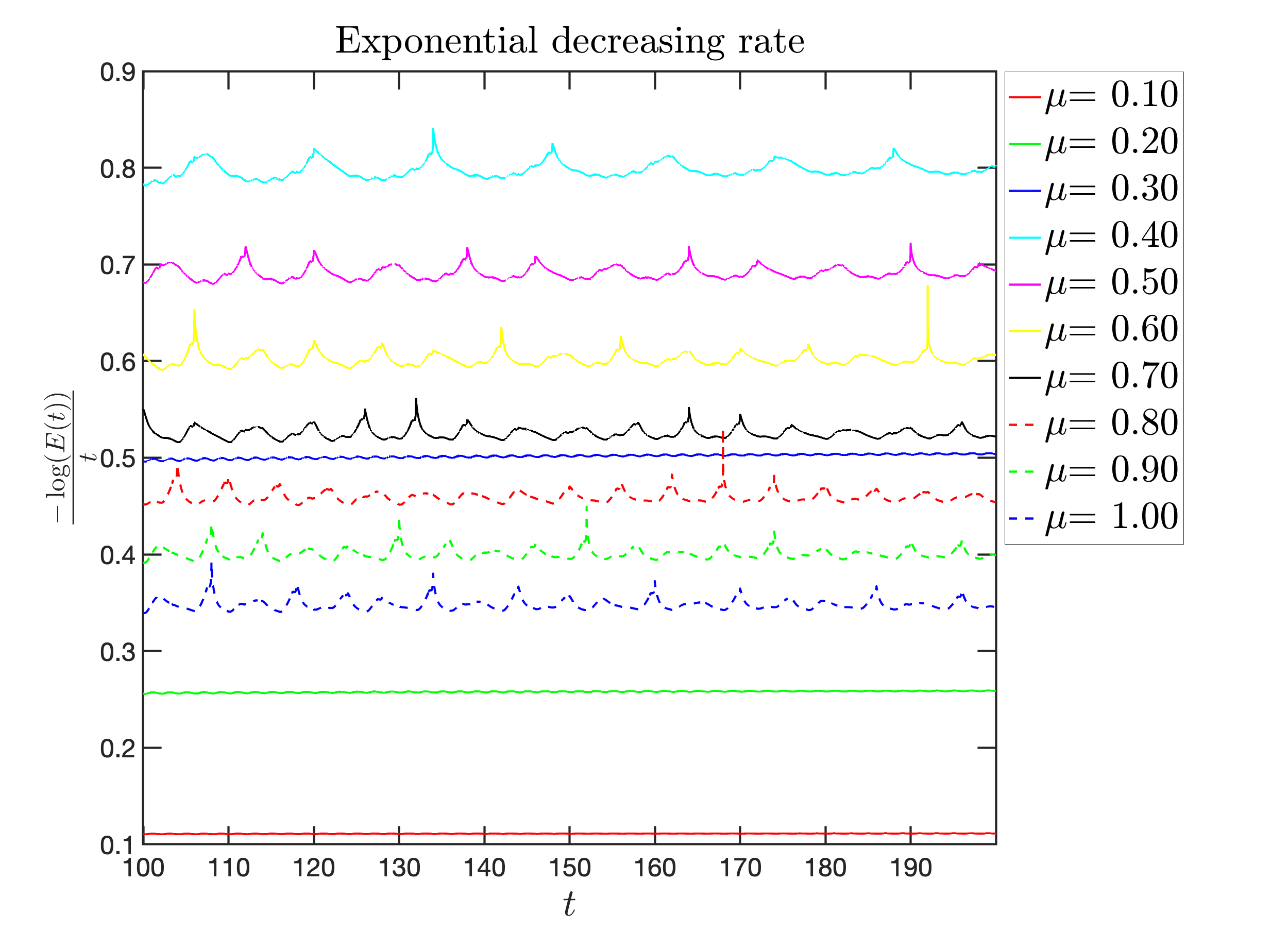}}
	\subcaptionbox{Evolution of $- \dfrac{\ln(E(t))}{t}$ for  $1< \mu \leq 2$.\label{Exp-Energy-pointwise-2}}
	{\includegraphics[width=0.9\textwidth]{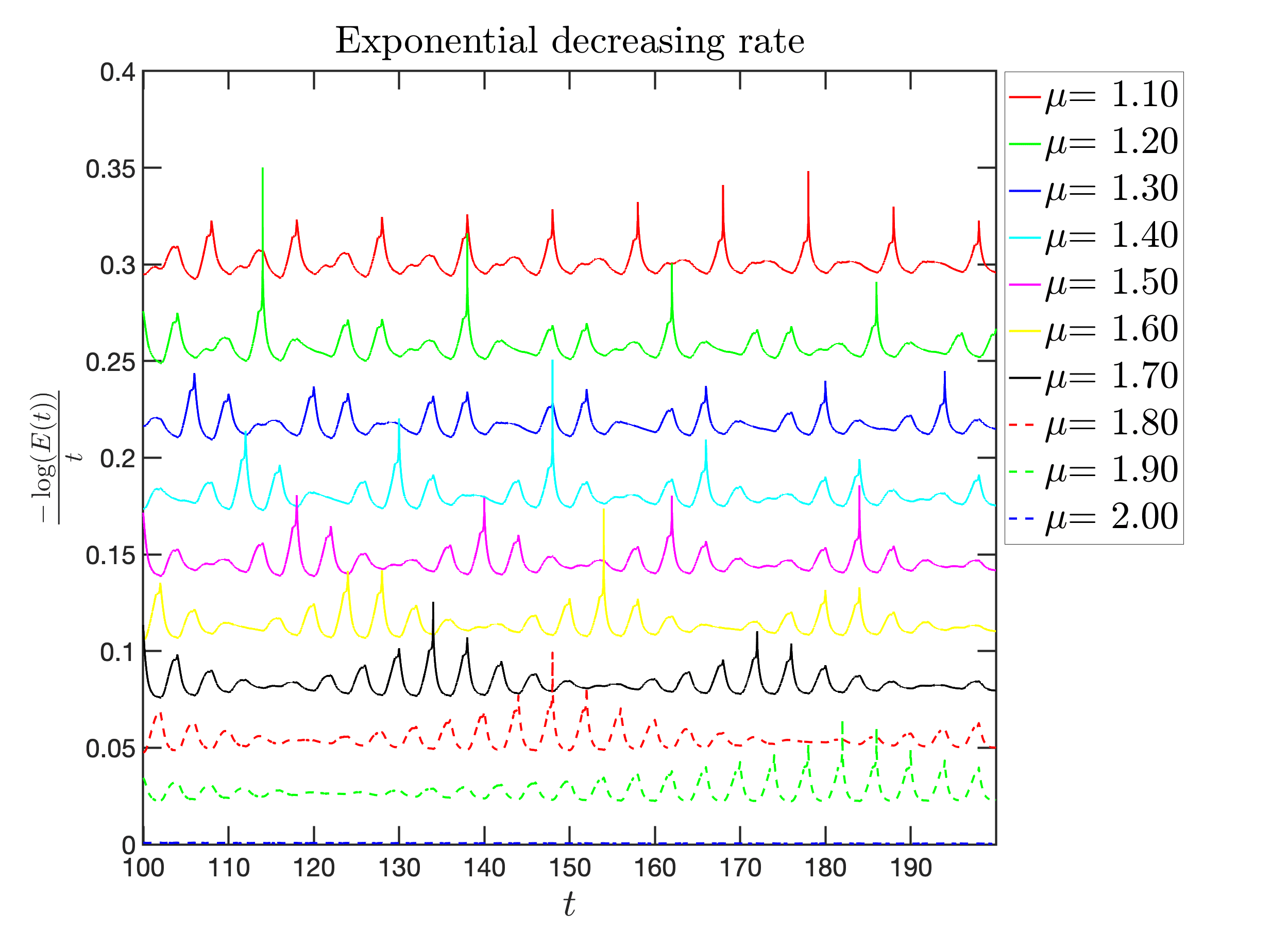}}
	\captionsetup{justification=centering}
	\caption{Pointwise delayed control. Exponential decay rate $0 < \mu \leq 2$}
	\label{Exp-Energy-pointwise}
\end{figure}
\newpage

\begin{figure}[H]
	\centering
	\subcaptionbox{Conservation of the energy for $\mu = 1$. \label{Energy-pointwise-mu-2}}	
	{\includegraphics[width=0.9\textwidth]{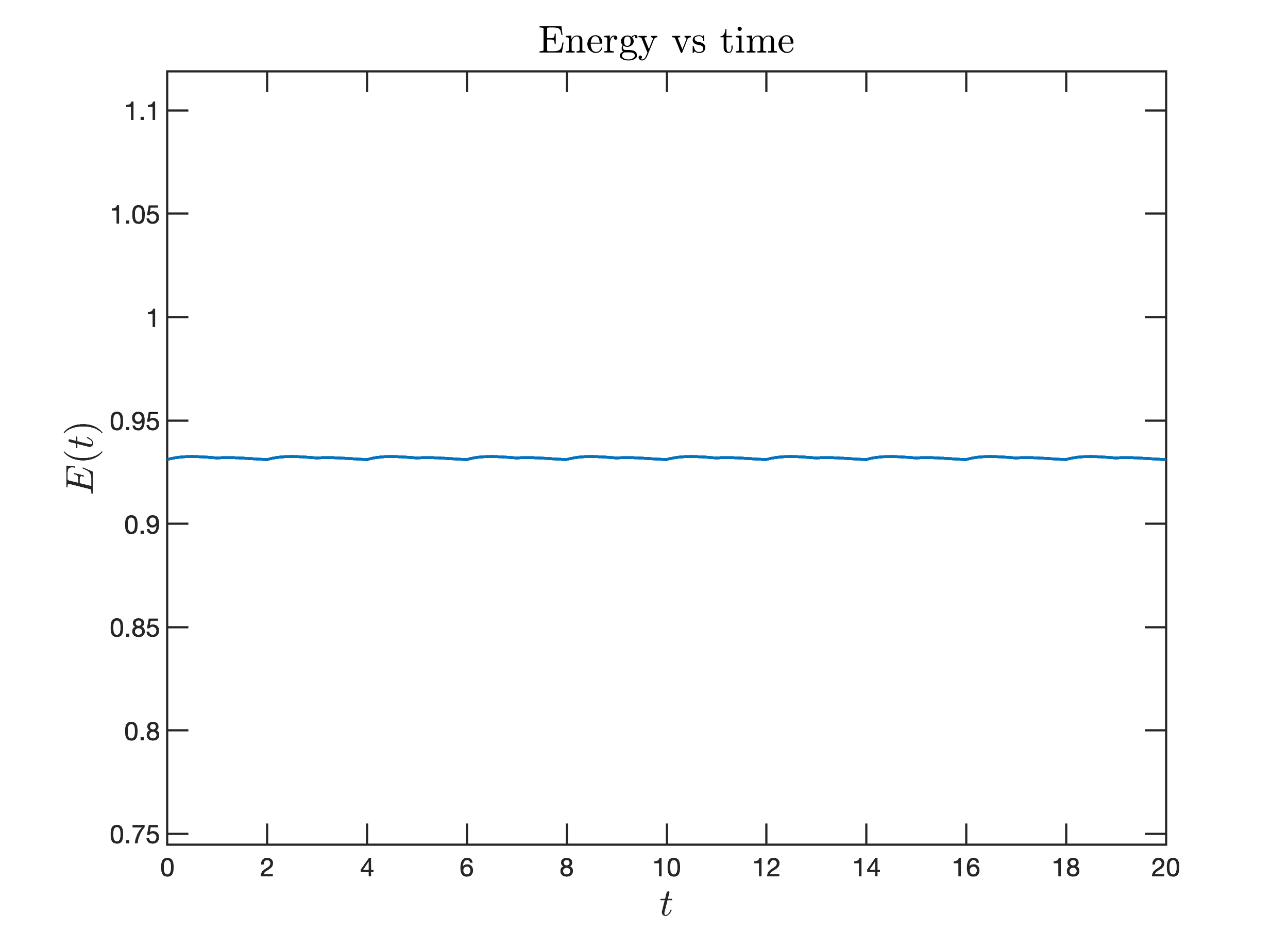}}	
	\subcaptionbox{Initial profile.\label{Profile-init-pointwise-mu-2}}	
	{\includegraphics[width=0.9\textwidth]{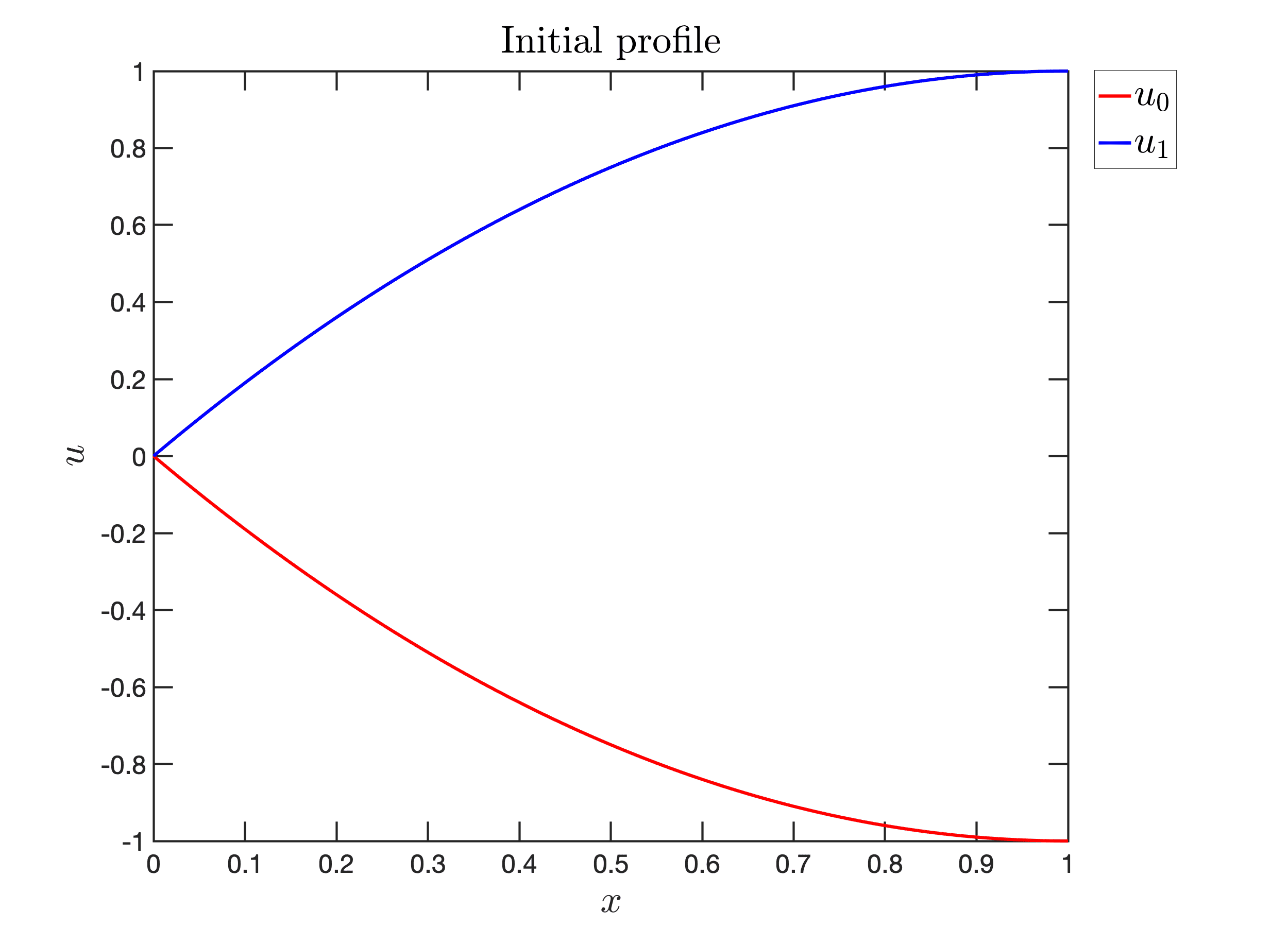}}	
	\captionsetup{justification=centering}
	\caption{Pointwise delayed control. The surprising case $\mu = 2$}
	\label{Energy-pointwise-mu-1-all}
\end{figure}
\newpage

\begin{figure}[H]
	\centering
	\subcaptionbox{Final profile for $T_{f} = 10 \times T = 20$\label{Profile-final-pointwise-1}}
	{\includegraphics[width=0.9\textwidth]{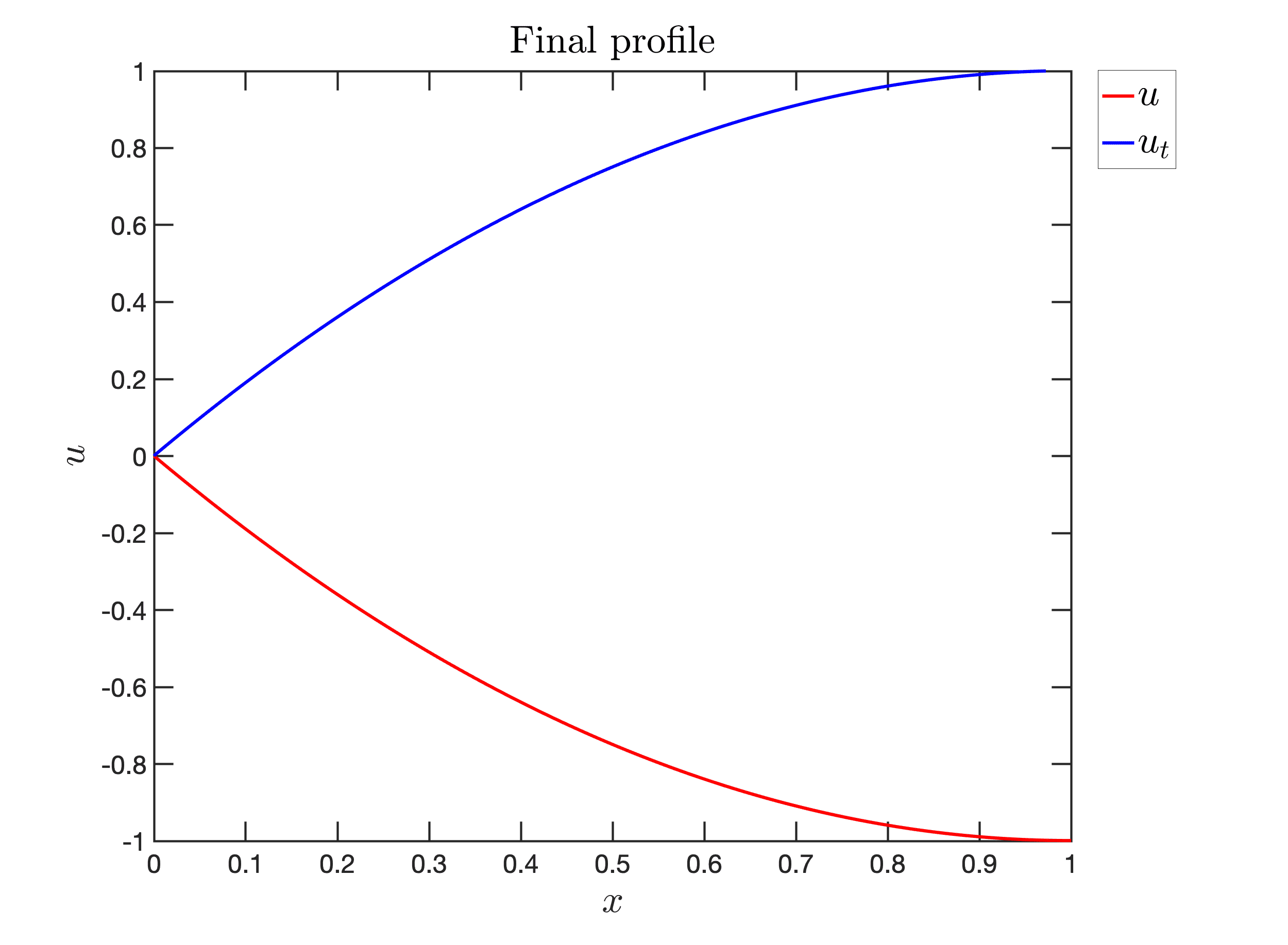}}
	\subcaptionbox{Final profile for $T_{f} = 11 \times T = 22$ .\label{Profile-final-pointwise-2}}
	{\includegraphics[width=0.9\textwidth]{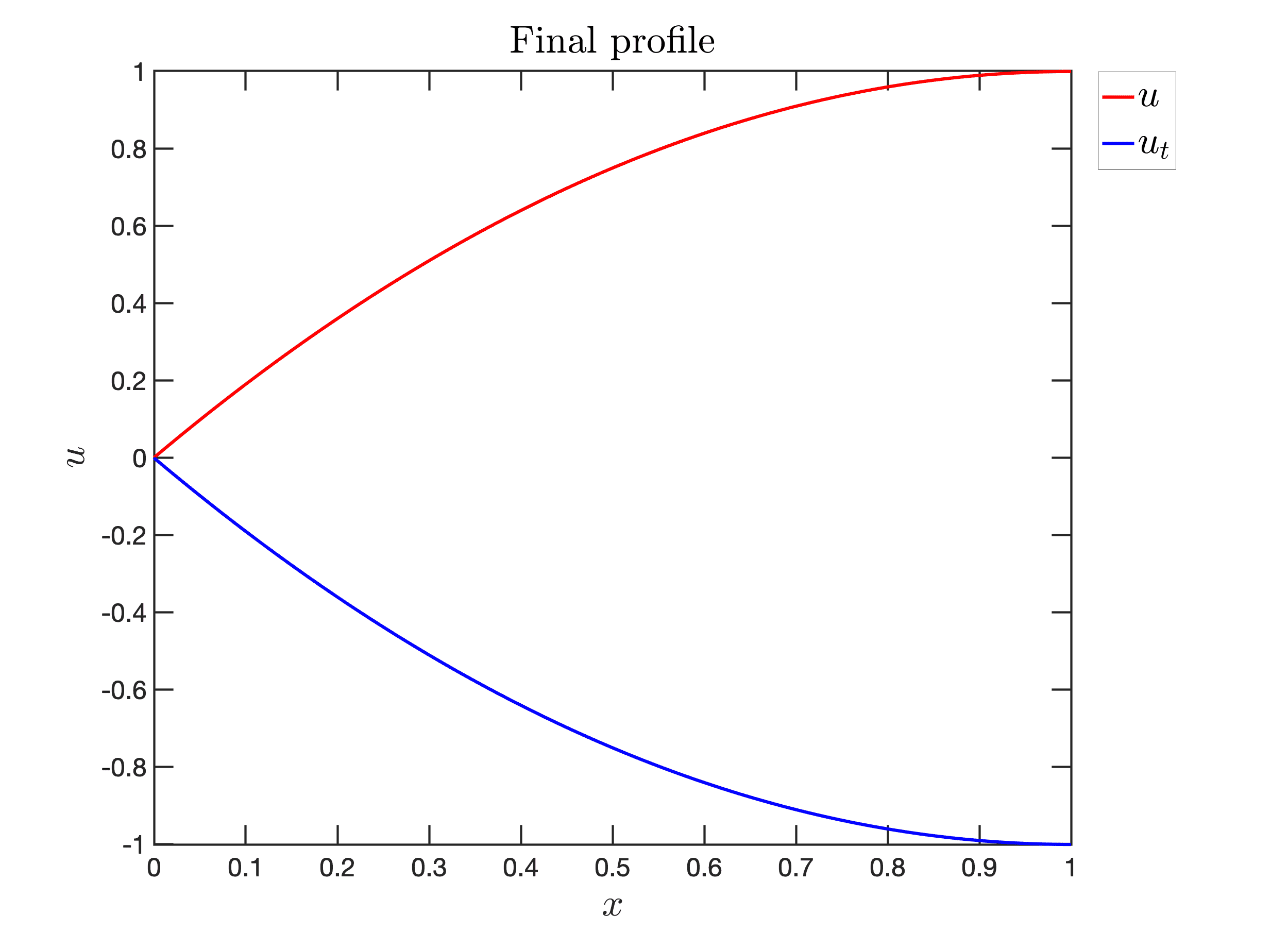}}
	\captionsetup{justification=centering}
	\caption{Pointwise delayed control. Final profiles for $\mu = 2$}
	\label{Final-pointwise}
\end{figure}

\newpage
\begin{figure}[H]
	\centering
	\subcaptionbox{Evolution of the discrete energy for $0 < \mu \leq 1 \,,\, t \in [0,20]$.\label{Energy-pointwise-sup-1}}
	{\includegraphics[width=0.9\textwidth]{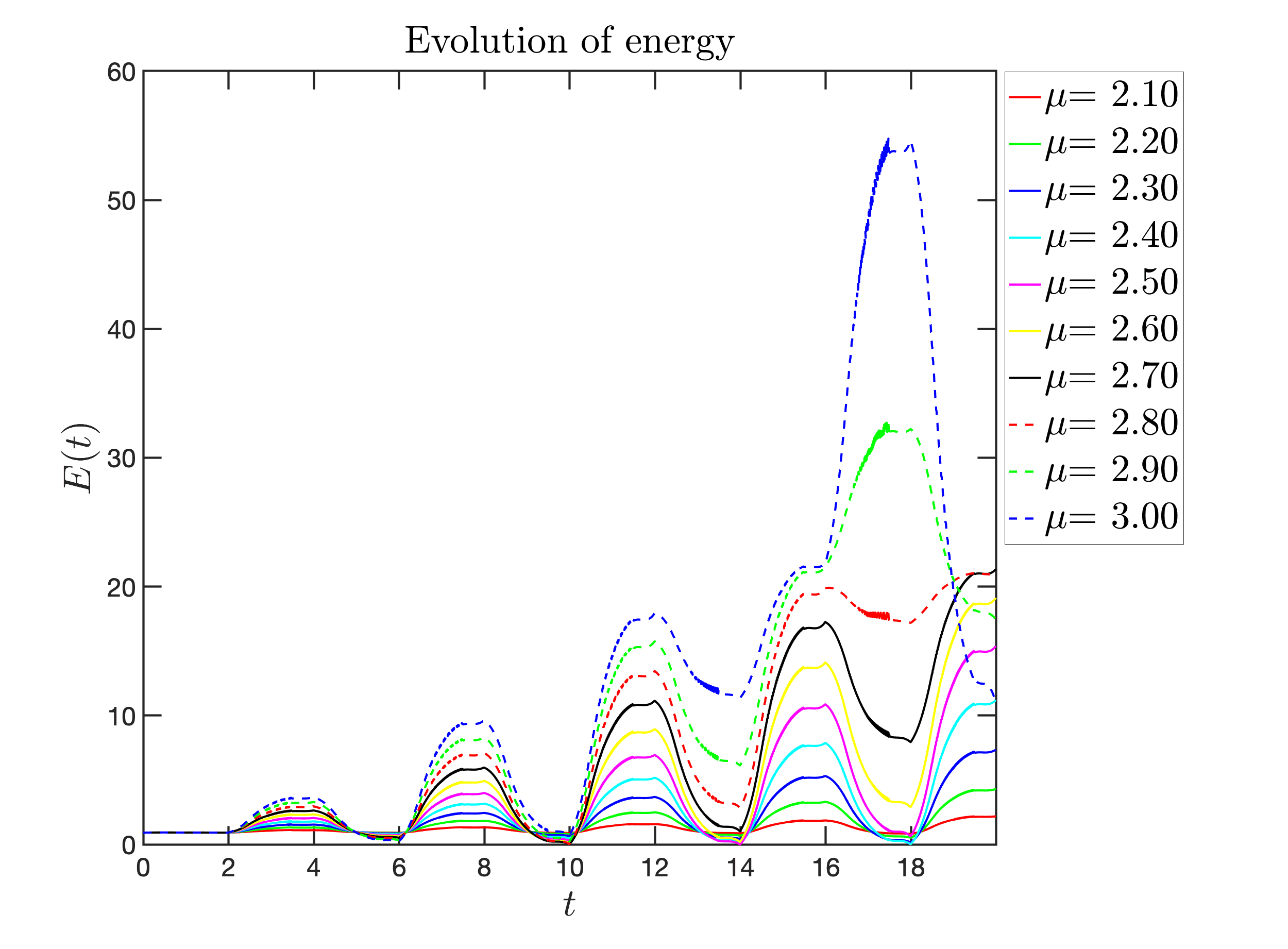}}
	\subcaptionbox{Evolution of the discrete energy for $1< \mu \leq 2 \,,\, t \in [0,20]$.\label{Energy-pointwise-sup-2}}
	{\includegraphics[width=0.9\textwidth]{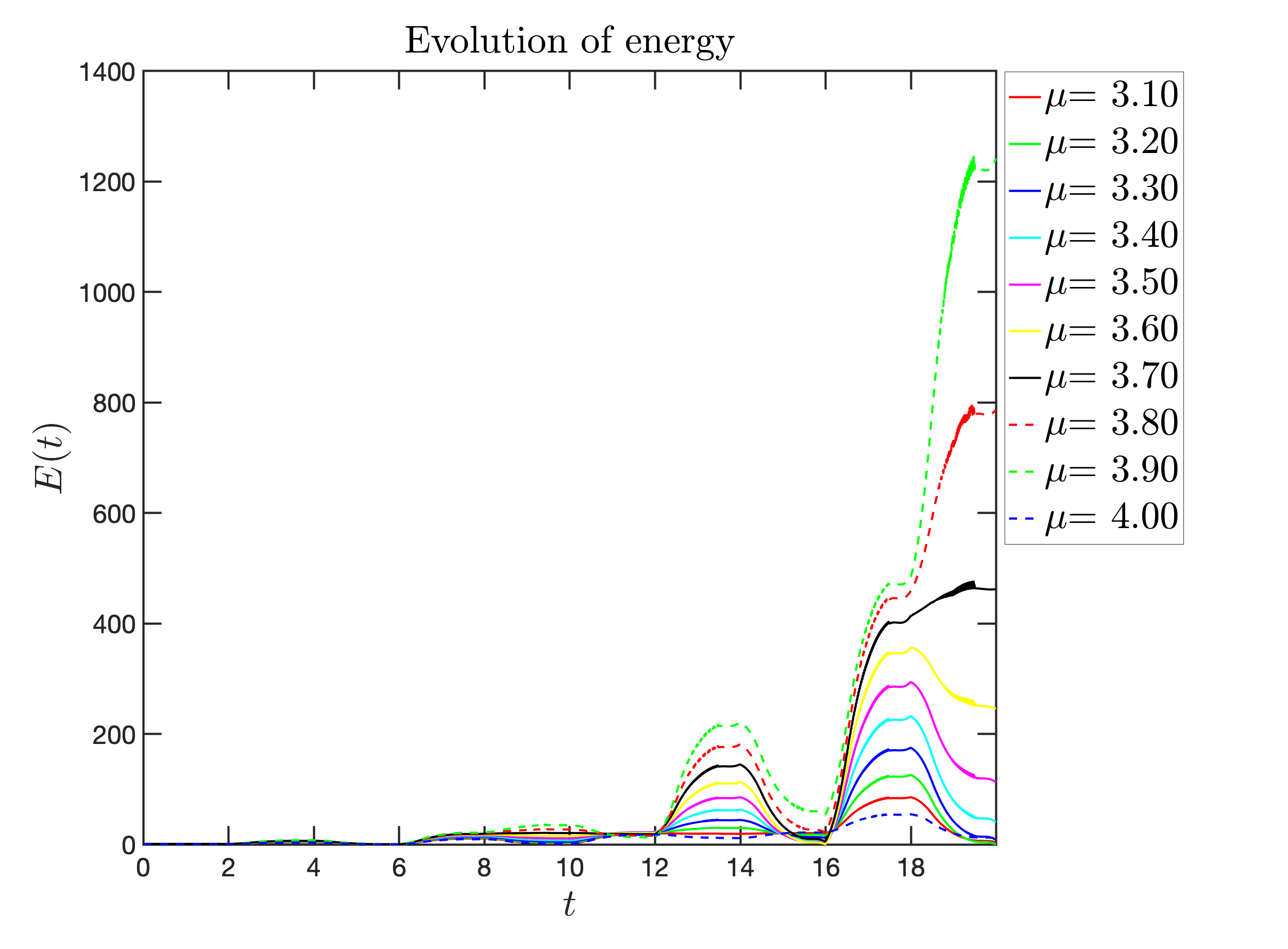}}
	\captionsetup{justification=centering}
	\caption{Pointwise delayed control. Energy when $2 < \mu $}
	\label{Energy-pointwise-sup}
\end{figure}

\newpage

\begin{figure}[H]
	\centering
	\subcaptionbox{Evolution of the discrete energy for $-1 \leq \mu < 0 \,,\, t \in [0,20]$.\label{Energy-pointwise-neg-1}}
	{\includegraphics[width=0.9\textwidth]{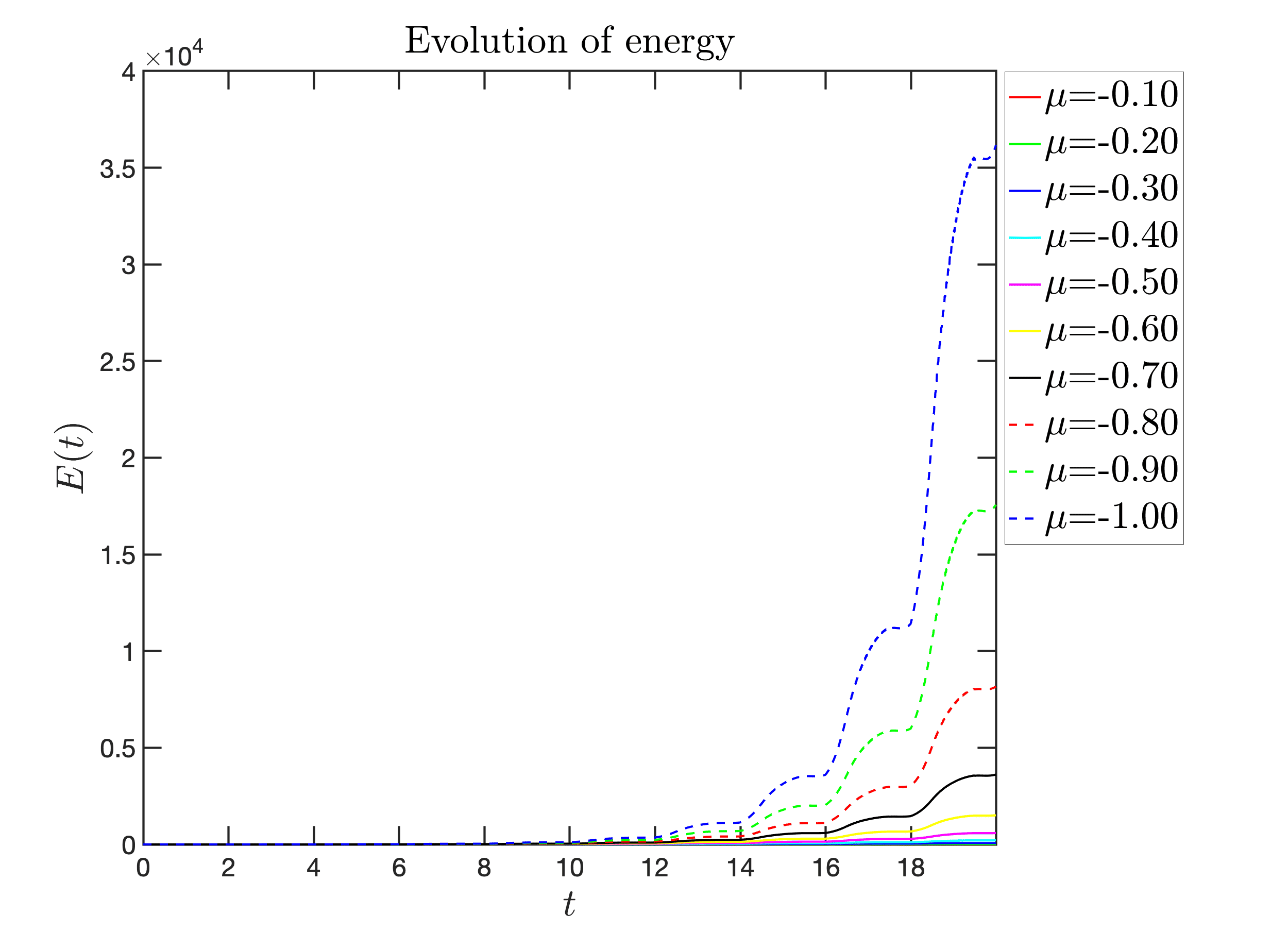}}
	\subcaptionbox{Evolution of the discrete energy for $-2 \leq \mu < -1 \,,\, t \in [0,20]$.\label{Energy-pointwise-neg-2}}
	{\includegraphics[width=0.9\textwidth]{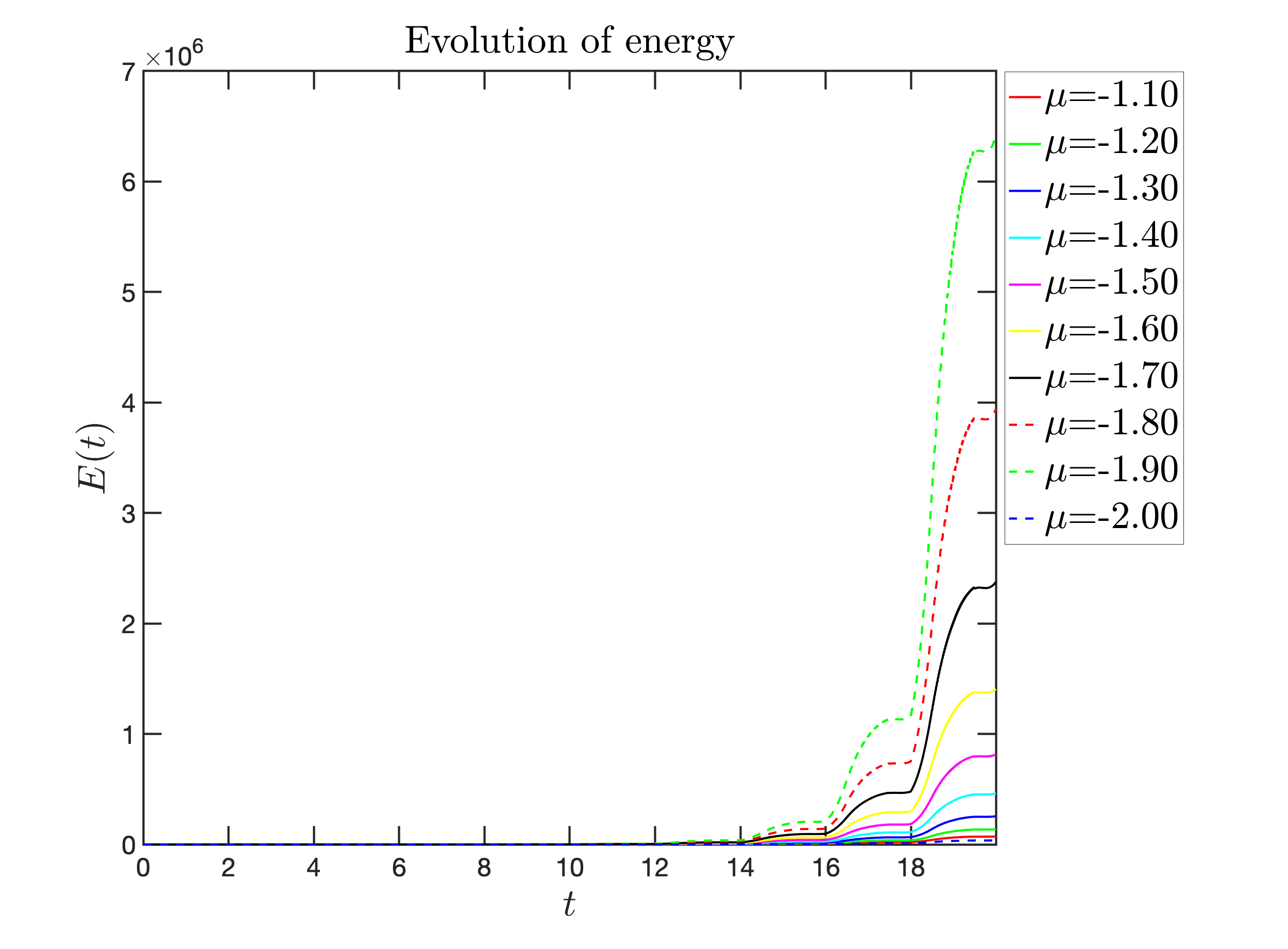}}
	\captionsetup{justification=centering}
	\caption{Pointwise delayed control. Energy when $ -2 \leq  \mu < 0$}
	\label{Energy-pointwise-neg}
\end{figure}

\end{document}